\documentclass{amsart}
\usepackage{amssymb}
\usepackage{latexsym}
\usepackage{amsmath}
\usepackage{enumerate}
\usepackage{amsthm}
\usepackage{wasysym}
\usepackage{stmaryrd}
\usepackage{mathrsfs}
\usepackage{pifont}
\usepackage[f]{esvect}
\usepackage[normalem]{ulem}
\usepackage{tikz-cd}

\usepackage{upgreek}
\usepackage{relsize}
\definecolor{uuxgreen}{cmyk}{1,0,0.75,0}
\definecolor{uuxred}{cmyk}{0.2,1,0.9,0.1}
\definecolor{uuyblue}  {cmyk}{0.9,0.55,0,0}
\definecolor{uuxblue}  {cmyk}{0.9,0.55,0,0}

\newcommand{\qee} {\hspace*{2mm}\hfill \ding{109}}

\renewcommand{\iff}{\leftrightarrow}
\renewcommand{\leq}{\leqslant}
\renewcommand{\geq}{\geqslant}
\renewcommand{\preceq}{\preccurlyeq}

\renewcommand{\phi}{\varphi}

\renewcommand{\Theta}{\varTheta}
\renewcommand{\Phi}{\varPhi}
\renewcommand{\Psi}{\varPsi}
\renewcommand{\Xi}{\varXi}
\renewcommand{\Omega}{\varOmega}
\renewcommand{\Gamma}{\varGamma}
\newcommand{\qedright}{\belowdisplayskip=-12pt}

\newtheorem{theorem}{Theorem}[section]
\newtheorem{define}[theorem]{Definition}

\newtheorem{exa}[theorem]{Example}

\newtheorem{exerc}[theorem]{Exercise}

\newtheorem{conj}[theorem]{Conjecture}

\newtheorem{ques}[theorem]{Open Question}

\newtheorem{lem}[theorem]{Lemma}

\newtheorem{cor}[theorem]{Corollary}

\newtheorem{rem}[theorem]{Remark}
\newenvironment{remark}{\begin{rem} \rm}{\qee\end{rem}}

\newcommand{\mc}[1]{\mathcal #1}
\newcommand{\num}[1]{{\underline {#1}}}
\newcommand{\mf}[1]{{\mathfrak {#1}}}
\newcommand{\verz}[1]{\{ #1 \}}
 \newcommand{\tupel}[1]{{\langle #1 \rangle}}
 \newcommand{\bles}{\mathbin{<}}
\newcommand{\bleq}{\mathbin{\leq}}

\newcommand{\axpam}{{\sf pa}}
\newcommand{\pam}{\ensuremath{{\sf PA}^-}}
\newcommand{\pamo}{\ensuremath{{\sf PA}^-}}
\newcommand{\pamj}{\ensuremath{{\sf PA}^-_{\sf jer}}}
\newcommand{\pamres}{\ensuremath{{\sf PA}^-_{\sf euc}}}
\newcommand{\pamsmu}{\ensuremath{{\sf PA}^-_{\sf smu}}}

\newcommand{\axtc}{{\sf tc}}
\newcommand{\axtcu}{{\sf tcu}}
\newcommand{\axtcl}{{\sf tc$\lambda$}}
\newcommand{\tc}[1]{\ensuremath{{\sf TC}_{#1}}}
\newcommand{\tcc}[1]{\ensuremath{{\sf TC}^{\sf c}_{#1}}}
\newcommand{\tcb}[1]{\ensuremath{{\sf TC}^{\sf b}_{#1}}}
\newcommand{\tcla}[1]{\ensuremath{{\sf TC\Lambda}_{#1}}}
\newcommand{\tclac}[1]{\ensuremath{{\sf TC\Lambda}_{#1}^{\sf c}}}
\newcommand{\utc}[1]{\ensuremath{{\sf TCU}_{#1}}}
\newcommand{\futc}[1]{\ensuremath{{\sf TCFU}_{#1}}}

\newcommand{\pcomp}{\mathbin{\textcolor{gray}{\bullet}}}
\newcommand{\pemp}{\ominus}
\newcommand{\redux}{\rightarrowtriangle}
\newcommand{\estr}{\text{\footnotesize $\varoslash$}}
\newcommand{\emp}{\ensuremath{\oslash}}
\newcommand{\X}{\mathsf{X}}
\newcommand{\jer}{Je\v{r}\'abek}
\newcommand{\sing}[1]{[#1]}
\newcommand{\eva}[1]{{\sf ev}(#1)}
\newcommand{\divi}[2]{\lfloor {#1}/{#2}\rfloor}
\newcommand{\rema}[2]{#1 \,\mathsmaller{\sf mod}\, #2}
\newcommand{\psing}[1]{\lfloor #1\rceil}
\newcommand{\polo}{\upvarpi}
\renewcommand{\omega}{\upomega}
\newcommand{\norm}[1]{\llbracket #1 \rrbracket}
\newcommand{\nnorm}[2]{\llceil #1,#2 \rrceil}
\newcommand{\ato}{{\sf atom}}

\newcommand{\spl}[4]{\begin{pmatrix}
  #1 & #2\\[0.2cm]
  #3 & #4
\end{pmatrix}}  

\newcommand{\sspl}[4]{\big(\begin{smallmatrix}
  #1 & #2\\
  #3 & #4
\end{smallmatrix}\big)}

\newcommand{\scol}[2]{\big(\begin{smallmatrix}
  #1\\
  #2
\end{smallmatrix}\big)}

\newcommand{\pamp}{\ensuremath{{\sf PA}^{-}_{\sf uc}}}
\newcommand{\docr}{{\sf DOCR}}

\title{From Numbers to Container Strings}

\author{Albert Visser}
 \address{Philosophy, Faculty of Humanities,
                Utrecht University,
               Janskerkhof 13,
                3512BL~~Utrecht, The Netherlands}
\email{a.visser@uu.nl}
\date{\today}

\begin{document}



\thanks{I am grateful to Lev Beklemishev, Daan van Gent, Hendrik Lenstra, Juvenal Murwanashyaka, and Vincent van Oostrom,  who all helped
me, at some stage  with the research for this paper. Lev's interest in this project was very encouraging.
Special thanks go to Emil \jer\ whose comments on a version of the paper
led to substantial improvement. I am grateful to Naoya Kato for sharing his results with me. His work led to
a substantial improvement in Version 5 of the preprint.
See Section~\ref{kunstmatigesmurf} for an explanation of the use of AI.}

\begin{abstract}
In this paper we examine two ways of coding sequences in arithmetical theories.
We investigate under what conditions they work. To be more precise, we study the creation of objects of 
a data-type that we call
\emph{container strings}, roughly sequences where the components are ordered but where we do not have
an explicitly given projection function and length function.

First, we have a brief look at the $\upbeta$-function which was already carefully studied by Emil \jer, who used it to show that
(a sub-theory of) \pam\ is sequential. We consider
in detail two target constructions. These constructions both employ theories of strings. The first is based on
Smullyan coding and the second on the representation of binary strings in the special linear monoid of
the non-negative part of discretely ordered commutative rings. The insight that this special linear monoid behaves as the
free monoid on two generators is due to Jakob Nielsen. It was used by Markov in a metamathematical context.
We call it Markov coding.
We employ Markov coding to obtain an alternative proof that \pam\ is sequential. 
We also prove a correspondence result for a principle for container strings to Euclidean Division under the Markov translation
over \pamo. 
\end{abstract}

\maketitle

\section{Introduction}
We examine two  ways of coding sequences in arithmetical theories.
We investigate under what conditions they work by isolating convenient base theories to
develop and verify the coding. 

Coding sequences is, in many cases,  the first step of arithmetisation. When we have it, undecidability and incompleteness follow.
However, for arithmetisation, we do not need sequences of \emph{all} objects of the domain of the given theory. Sequences for
a suitable  definable sub-domain suffice. So, why would one want to have sequences for the whole domain? The main reason is that
such sequences allow us to build partial satisfaction predicates for the given theory inside the given theory. 
Also, they allow us to \emph{extend} models of the given theory with full satisfaction predicates. 
The partial satisfaction predicates deliver the many good properties of sequential theories. See, e.g., \cite{viss:what13}
and \cite{viss:smal19}.  Extending models with satisfaction predicates has been intensely studied for models of
Peano Arithmetic. See, e.g., \cite{cies:epis17}.
 
Theories with sequences for all objects are \emph{sequential theories}. The paper \cite{viss:what13}
provides historical background and further references. 

This paper focusses on the construction of sequence-coding in very weak theories. 
We opt for the weak theories in order to see more precisely what mathematical principles are involved in
the construction of the sequences. Thus, the choice of weak theories makes our project into a study of
reasoning. In many cases, the weak theories correspond to theories known from algebra, thus strengthening the
interconnectivity of this kind of project with mathematics.

G\"odel's original way of coding sequences using the $\upbeta$-function is still one of the best
approaches.  We will outline some of its virtues below. The $\upbeta$-function is carefully
studied in \cite{jera:sequ12}. In the present paper, we zoom in on two different strategies. Both strategies
develop sequence codes via the interpretation of theories of strings, but in markedly different ways. 
The first route is what I call \emph{Smullyan coding}. See \cite{smul:theo61}. Here, we interpret a suitable theory of binary strings
in our base arithmetical theory. The interpretation of the strings employs the lexicographic  length-first ordering. It is a system
of binary notations without the zero. When we have developed the strings, we construct our sequences from the strings. This can be
done in multiple ways, but we opt for one way which is simple and delivers good properties.
The second route is what I call \emph{Markov coding}. In fact this coding was first discovered by Jakob Nielsen, but it was used for metamathematical
purposes for the firs time by Andrej Markov jr. See \cite{niel:isom18} and \cite{mark:theo54}. It expands on the insight that ${\sf SL}_2(\mathbb N)$, the special linear monoid
of $2\times 2$ matrices of natural numbers (non-negative integers) with determinant 1, is isomorphic to the monoid of binary strings. As we will see, this way of coding strings is markedly
different from the Smullyan way. The use of Markov coding provides a new proof that the weak theory \pam\ is sequential.
We will also show, using the work of Naoya Kato in \cite{kato:equi26},  that a principle for container strings, the Stack Principle,
 under the Markov translation, is equivalent to
Euclidean Division over \pamo.

\begin{remark}{\small
It would be great to be able to do something like the reverse mathematics of coding. However, this aim is too high 
for the moment. A modest reverse mathematics result can be found in Subsection~\ref{reversosmurf}.}
\end{remark}

The present paper is as much about finding the appropriate modularisation to obtain the desired results as it is about the concrete results themselves.

\subsection{Sets, Container Strings, Sequences}
The notion of \emph{sequence}, as we use it in this paper, demands that we have numbers, an explicit length function
that gives the length of the sequences, and an explicit projection function that sends numbers smaller than the
length of the sequence to the appropriate elements of the sequence.

A weaker notion is  \emph{container string}. This is the main notion studied in this paper.
The container string is a string where all elements of a designated sort ---the ur-elements--- are embedded in its alphabet.
We \emph{do} have concatenation of container strings and  an ordering on the occurrences of the elements, but we need not have
length- and projection-functions. Thus, a container string can be viewed as a multi-set with ordered occurrences.
The difference between a string and a container string is, in essence, one of perspective: the container string
is a carrier of its elements.
 For sequentiality, we are interested in the case where the sort that is embedded in the letters is
the universal sort of all objects. We will see in the paper how this idea of universality can be implemented using  Frege
functions.

We can always construct sequences from container strings (see below), but this construction has some costs.
Specifically, concatenation becomes a more complicated operation.  It seems to me, however, that
one almost never needs sequences. container strings suffice in most circumstances. For example,
think of a computation of a Turing machine. We can represent it as a container string of triples,
where the first component represents the content of the tape to the left of the head, the second component
the content of the tape to the right, and the third component the state. The main thing we need to
say is that a given triple is preceded either directly or indirectly by another triple satisfying so-and-so.
A second example is a Hilbert-style system. One never needs the exact places of the earlier
stages, just the fact that they \emph{are} earlier. Finally, consider the  assignments we need for a
partial truth predicate. These can be very well modeled as finite sets of pairs of variables and values.
The basic operation on these is union under the condition that we preserve functionality.
For this purpose, we can use the sets based on the container strings. Of course, also sequences
are used to represent assignments: here we assign values to the first $n$ variables,
where $n$ is the length of the sequence. Note however that concatenation is not a
very meaningful operation on sequences in this role. So, for this specific purpose,  the set-of-pairs representation
is better. 

One way to construct our sequences from container strings is by first constructing (not necessarily extensional) sets from container strings.
We simply forget the order and the multiplicity of the components. We only need these `sets' to satisfy the
modest demands of Adjunctive Set Theory (see Section~\ref{rijtjessmurf}). Then, we do a bootstrap
that will deliver sequences from sets. The usual representation of a sequence as a function from  a set of numbers
to the elements of the sequence is part of this construction. 

\subsection{Note on Terminology}
In this paper, I use \emph{string} for strings of letters. These are like sequences of letters but without projection and length functions.
In much of the literature, the word \emph{word} is used here. I do not like the word \emph{word} so much in this context since I believe
that words, in the everyday sense, are \emph{not} strings of letters. Change of spelling does not change the word hoard. 
The use of the word \emph{string} in our context is reasonably well established. It is used, in our sense, e.g., by
Corcoran, Ferreira, \v{S}vejdar, and Murwanashyaka. I also found it used in Wikipedia.

\subsection{Desiderata}

There is a sense in which it does not matter how we obtain the sequences or container strings. The sequences or container strings will deliver 
the goods as long as they have the 
desired properties. In this paper, we opt for a somewhat sharper focus. The paper is a reflection of the following further
desiderata. 

\begin{itemize}
\item
\emph{Scope:} We want our construction to work for as wide a class of theories as possible. In this, as we will see, G\"odel's $\upbeta$-function is the
current winner.  Emil \jer\ in \cite{jera:sequ12} has shown how to construct sequences 
over the very weak theory \pam\ minus the subtraction action.
We construct sequences using Markov coding in
\pam, but we seem to need the subtraction axiom. The base theory for Smullyan coding is still stronger. It is not even  a subtheory of {\sf IOpen}.  
\item
\emph{Simplicity:} We want both our constructions and their results to be as simple as possible. In the
constructions, we want to minimise the use of Solovay-style shortening of
definable cuts.  In the results, we want the definitions of adjunction and concatenation to be as elementary as possible.

Emil \jer,  in \cite{jera:sequ12} produces sequences with adjunction. He uses Solovay-style shortening 
of the length of the sequences to verify adjunction. This means that
the sequence lengths are explicitly constrained to a definable cut that is shorter than the surface numbers of the theory.
To get sequences with concatenation we have to shorten one step further.
The resulting concatenation is somewhat complex since we have to update two underlying numbers used in coding
the component sequences of the concatenation.

The Smullyan construction gives container strings with concatenation without any further tricks.
The concatenation is directly derived from the concatenation of the binary strings that function as an
intermediate step in the construction. The cost of the good properties of the Smullyan construction is, regrettably,
that we need a theory that is not contained in {\sf IOpen}. From a didactic standpoint the Smullyan coding
is very good since it is closely related to the familiar binary notations for numbers.

The theory \pam\ will be our base for the Markov construction. It gives us
container strings with full concatenation without any need for Solovay-style shortenings.
Markov concatenation is a simple and familiar thing, to wit matrix multiplication.
Markov coding has the advantage that the basic definitions are quantifier-free.

If we want to go from container strings to sequences, some of the disadvantages of the $\upbeta$-function
reappear. Concatenation becomes a somewhat complex operation.
\end{itemize} 

\subsection{Exponentiation}
Peter Aczel, in the context of constructive mathematics, coined the word \emph{taboo} for a principle that you do not
want to be derivable in your system more or less as if it were \emph{falsum}. For us, in this paper, the totality of exponentiation is such a
taboo. We insist on slow growth. Of course, exponentiation does still throw its platonic shadow downwards. For example,
in the treatment of Smullyan coding, we do use the notion \emph{being a power of 2}. 

In the context of the development of the $\upbeta$-function in a very weak theory, Emil \jer\ avoids exponentiation by shortening the length
of the sequences. In the context of  Smullyan and Markov coding, we have the following issue. We want to code
 container strings of numbers. We want to make the strings code sequences of all numbers.
One way in which the numbers appear in the binary strings is as \emph{tally numbers}, strings of {\tt a}'s, where the number $n$ is
coded by ${\tt a}^n$. So, why not code $n_0,\dots,n_{k-1}$ simply as ${\tt b}{\tt a}^{n_0} \dots {\tt b}{\tt a}^{k-1}$? 
As we will
see, in the case of Markov coding, this plan works perfectly. The reason is that the string representing 
${\tt a}^n$ is Markov-coded by the matrix $\sspl 1n01$, a perfectly small object in terms of $n$.
On the other hand, the same string is Smullyan-coded by the number $2^n-1$. So, the codes of tally numbers grow exponentially in the given number.
In the Smullyan case, the solution is as follows. Our coding of binary strings in the numbers  is 1-1. So, we represent sequences of numbers as sequences of 
the corresponding
strings. Thus, we need to code sequences of all strings as strings. There are several ways to do that. We will discuss those in Subsection~\ref{lerendesmurf}.  

We do not need explicit shortening in the construction of the Smullyan and of the Markov coding. This does not mean that 
we can have sequences of any length. For example, in the Markov case, the string $({\tt ba})^n$ is definable (by a Diophantine predicate) as a function of $n$
(as proved in by Yuri Matiyasevich), but this involves exponential growth.
Thus, the $n$ for which Markov-style $({\tt ba})^n$ exists,  can only live in a logarithmic cut. 
The same holds already for ${\tt a}^n$ in the Smullyan case.

\subsection{Use of AI}\label{kunstmatigesmurf}
ChatGPT substantially assisted to find the proof of Lemma~\ref{diamantsmurf}. See Remark~\ref{chatsmurf} for details.
I used ChatGPT for an extensive check-up of Version~4 of the arXiv-preprint. It found many
small mistakes and one larger one that was fortunately easily repairable. 

\section{Basics}

In this section we briefly introduce some basic definitions and results. 

\subsection{Theories, Translations, Interpretations}
In this paper, we are only dealing with specific theories. Moreover, no intensional phenomena will be studied. As a consequence,
we do not have to worry all that much about what precisely a theory is. For specificity, let us say that a theory is given by a signature plus a set of axioms.

\emph{A translation $\tau$} from a signature $\Theta$ to a signature $\Xi$ provides a domain-formula $\updelta_\tau$ in the language of signature
$\Xi$. This domain consists of (syntactic) sequences of length $n$. This $n$ gives  \emph{the dimension} of the interpretation.
The translation sends a $k$-ary predicate symbol $P$ in $\Theta$ to a formula $\phi(\vv x_0,\dots, \vv x_{k-1})$ in the language of signature $\Xi$.
Here the length of the $\vv x_i$ is the dimension $n$ of the translation.
We will designate this $\phi$  by $P_\tau$. We allow identity to be translated in the same way as the other predicate symbols.
We lift the interpretation in the obvious way to the full language relativising the quantifiers to the domain of the translation. 
For example, suppose $\tau$ is 2-dimensional and $P$ in $\Theta$ is binary. Then,
$(\forall x \, P(x,y))^\tau$ is the formula:
$ \forall x_0\,\forall x_1\,(\updelta_\tau(x_0,x_1) \to P_\tau(x_0,x_1,y_0,y_1))$.
Of course, we need some appropriate convention here for the choice of the variables $x_0,x_1,y_0,y_1$.

A translation $\tau$ from $\Theta$ to $\Xi$ carries \emph{an interpretation} of a $\Theta$-theory $U$ in a $\Xi$-theory
$V$ iff  $U \vdash \psi$ implies $V \vdash \psi^\tau$, for all $\Theta$-sentences $\psi$.

We will also consider translations and interpretations between many-sorted languages. The only difference here is
that each sort gets its own domain. Conceivably, the $\tau$-domain of one $\Theta$-sort could consist of sequences
of objects coming from different $\Xi$-sorts, but, in this paper, this does not happen.

In the many-sorted case, we will assume that there always is a designated object-sort $\mf o$.
We treat the non-sorted case as the one-sorted case with the single sort $\mf o$.
A translation $\nu$ is \emph{$\mf o$-direct} if it is one-dimensional on $\mf o$, sends the object-sort to the object-sort,
 $\updelta_\nu^{\mf o}(x^{\mf o})$ is $x^{\mf o}=_{\mf{oo}}x^{\mf o}$, and  $=_{\mf {oo}}$ translates to
$x_0^{\mf o} =_{\mf {oo}} x_1^{\mf {o}}$.
If a theory is one-sorted, the unique sort is automatically the object sort. In this case we usually simply speak of
a \emph{direct} interpretation.

See   \cite{viss:card09} or \cite{viss:predi09} or \cite{viss:what13} for more details on translations and interpretations
for many-sorted theories.

\subsection{Sequentiality}\label{rijtjessmurf}
A theory is \emph{sequential} whenever it has a good theory of sequences. This means that we need to have sequences, numbers, projection-functions from
numbers to objects and a length-function. We want to make sequences for all objects of the domain of the theory.
We demand more than just that, for every standard $n$, we can form sequences of length $n$. That idea delivers the weaker notion of \emph{Vaught theory}.
We ask that, theory-internally, whenever we have a sequence, we can adjoin any object at the end of the sequence.

As we see, the notion of \emph{sequentiality} is rather complex. Fortunately, it can be given a much simpler definition. Given that a theory satisfies this simple definition,
we can, via a huge bootstrap, build the numbers, sequences, projections, length from it.\footnote{One can show that
it is essential here that the translation is not identity preserving for the numbers.} We treat the simple definition for the one-sorted case. A one-sorted theory is sequential iff
it $\mf o$-directly interprets the theory {\sf AS}. 
The theory {\sf AS} is a theory in the language with one binary predicate $\in$, axiomatised by the following principles.

\begin{enumerate}[{\sf as}1.]
\item
$\vdash \exists x\, \forall y\; y\not\in x$
\item
$\vdash \exists z\, \forall u\, (u\in z \iff (u \in x \vee u=y))$
\end{enumerate}

It is nicer to give {\sf AS} a two sorted format, which we call {\sf FAC} or Frege-style Adjunctive Class Theory.
It is a two-sorted theory with a sort $\mf o$ of objects and a sort $\mf c$ of classes/concepts. 
We use capital letters to range over the sort $\mf c$.
We have a binary predicate $\in$ of type $\mf{oc}$ and a Frege function $\digamma$ if type $\mf{co}$.
The theory is axiomatised as follows.
\begin{enumerate}[{\sf fac}1.]
\item
$\vdash \exists X\, \forall y\, y\not \in X$
\item
$\vdash \exists Z\, \forall u\, (u \in Z \iff (u\in X\vee u=y))$
\item
$\vdash \digamma (X) = \digamma(Y) \to X=Y$
\end{enumerate}

We note that we can demand extensionality for classes, but the cost of that is that we need a total injective  Frege \emph{relation}
rather than a function (which is not really a problem).
A one-sorted theory $U$ will be sequential if it $\mf o$-directly interprets {\sf FAC}.

We note that {\sf AS} and {\sf FAC} are mutually $\mf o$-directly interpretable. The translation $\upzeta$
of {\sf AS} in {\sf FAC} has $x\in_\upzeta y :\iff \exists Y\, (\digamma(Y)= y \wedge x \in Y)$.

There is also a Frege-function-free approach. Let {\sf AC} be {\sf FAC} without the Frege Axiom.
Then, a one-sorted $U$ is sequential iff $U$ $\mf o$-directly interprets {\sf AC} where the interpretation for the
sort of classes is one-dimensional. See \cite{viss:what13} for a detailed discussion of sequentiality.

\begin{rem}
{\small We can extend the notion of sequentiality to many-sorted theories as follows. Consider a many-sorted $U$.
We expand $U$ by a universal sort which is stipulated to be the new object sort. We have collectively injective
Frege-functions from the old sorts to the new sort. We demand that $U$ interprets its extended version with an
interpretation that is identical on the old sorts. Finally, we ask that extended $U$ $\mf o$-directly interprets {\sf AS} or
{\sf FAC}. An example of a two-sorted sequential theory is ${\sf ACA}_0$.
}
\end{rem}

We will consider two strengthenings of {\sf FAC} to wit ${\sf FAC}^{\sf f}$ and  ${\sf FAC}^{\sf f+}$.
The theory ${\sf FAC}^{\sf f}$ is the functional version of {\sf FAC}. It has an additional constant $\emptyset$ of type $\mf c$ and
a binary function {\sf adj}, of type $\mf{coc}$. It is axiomatised as follows.
\begin{enumerate}[{\sf facf}1.]
\item
$\vdash  y\not \in \emptyset$
\item
$\vdash u \in {\sf adj}(X,y) \iff (u\in X\vee u=y)$
\item
$\vdash \digamma (X) = \digamma(Y) \to X=Y$
\end{enumerate}

The theory  ${\sf FAC}^{\sf f+}$ is the union version of  ${\sf FAC}^{\sf f}$. It has, in addition to the signature of {\sf FAC},
 a constant $\emptyset$ of type $\mf c$,
a unary function $\verz\cdot$ of type $\mf {oc}$,  and a binary function $\cup$ of type $\mf{ccc}$.
It is axiomatised as follows.
\begin{enumerate}[{\sf facfp}1.]
\item
$\vdash   y\not \in \emptyset$
\item
$\vdash u \in \verz x \iff u=x$
\item
$\vdash u \in X\cup Y \iff (u\in X\vee u\in Y)$
\item
$\vdash \digamma (X) = \digamma(Y) \to X=Y$
\end{enumerate}

\noindent Clearly, both ${\sf FAC}^{\sf f}$ and ${\sf FAC}^{\sf f+}$ $\mf o$-directly interpret {\sf FAC}.

\section{Arithmetical Theories and the $\upbeta$-Function}\label{emilwerk}
In Section~\ref{arithsmurf}, we introduce the arithmetical language and the various arithmetical theories we are working with.
In Section~\ref{betasmurf}, we briefly sketch Emil \jer's work on \pamj\ and the $\upbeta$-function. Section~\ref{moresmurf}
gives a few extra facts about the $\upbeta$-function. Section~\ref{megasmurf}
collects various useful results on the arithmetical theories under consideration. 

\subsection{Theories}\label{arithsmurf}
The number theoretic language that we will be using is given by the constants $0$ and  $1$, the binary operations $+$ and $\times$, and 
the binary relation $\leq$.
We will often write $xy$ or  $x\cdot y$ for $x\times y$ and we will follow the usual rules for omitting brackets. 
We  employ the following abbreviations.
\begin{itemize}
\item
$x< y$ iff $x\leq y \wedge x\neq y$.
\item
$x\mid y$ iff $\exists z\, z\cdot x = y$.
\item 
$\num 0:= 0$, $\num 1 :=1$, $\num {n+2} := \num{n+1}+1$.
\item
${\sf pow}_2(x)$ iff $\forall y \, (y\mid x \to (y=1\vee \num 2\mid y))$.\footnote{For more information about the ins and outs of this
definition of powers of 2, see Subsection~\ref{megasmurf}.}
\end{itemize}

The theories we are going to consider, will be axiomatised by selections of the following principles.
\begin{enumerate}[\axpam1.]
\item
$\vdash x+0 = x$ \label{plimp1}
\item
$\vdash x+y=y+x$\label{plimp2}
\item
$\vdash (x+y)+z = x+(y+z)$\label{plimp3}
\item
$\vdash x\cdot 1 = x$\label{plimp4}
\item
$\vdash x \cdot y = y \cdot x$\label{plimp5}
\item
$\vdash (x\cdot y) \cdot z = x \cdot (y \cdot z)$\label{plimp6}
\item
$\vdash x\cdot(y+z)= x\cdot y + x\cdot z$\label{plimp7}
\item
$\vdash x \leq y \vee y \leq x$\label{plimp8}
\item $
\vdash (x\leq y \wedge y \leq z) \to x \leq z$ \label{plimp8A}
\item
$\vdash x+1 \not \leq x$\label{plimp10}
\item
$\vdash y\leq x \to (y=x \vee y+1 \leq x)$\label{plimp9}
\item
$\vdash y\leq x \to y+z \leq x + z$\label{plimp11}
\item
$\vdash y \leq x\to y\cdot z \leq x\cdot z$\label{plimp12}
\item
$\vdash x=0 \vee \exists y\;\, x=y+1$ \label{presmurf}
\item
$\vdash y\leq x \to \exists z\; \,x=y+z$ \label{plimp15}
\item
$\vdash  \exists z \,\exists r\, (r<\num n \wedge x = z\cdot \num n + r)$, for all non-zero natural numbers $n$  \label{zringax}
\item
$\vdash y \neq 0 \to \exists z \,\exists r\, (r<y \wedge x = z\cdot y + r)$  \label{plimp16} \label{kleinebeersmurf}
\item
 $ \vdash x\mid y\cdot z \to \exists u\,\exists v\,(x= u\cdot v \wedge u \mid y\wedge v \mid z)$ \label{smurferella}
\item\label{smurfin}
$\vdash  \exists y\, ({\sf pow}_2(y) \wedge y \leq x+1 < \num 2\cdot y)$
\item\label{babysmurf}
$\vdash ( {\sf pow}_2(x) \wedge {\sf pow}_2(y)  \wedge x\leq y) \to x \mid y$
\item\label{supersmurf}
$\vdash   \exists z\, \forall u\, (u \mid z \iff (u\mid x \wedge u\mid y))$
\end{enumerate} 

\medskip
We write \axpam\ref{zringax}(n) for $\vdash  \exists z \,\exists r\, (r<\num n \wedge x = z\cdot \num n + r)$, where $n$ is non-zero.
We give names to some of our principles and, for some, we provide a little bit of background.
\begin{itemize}
\item
\axpam\ref{presmurf} is \emph{the Predecessor Principle}.
\item 
\axpam\ref{plimp15} is  \emph{the Subtraction Principle}.
\item
\axpam\ref{zringax} is  \emph{the Standard Euclidean Division Principle}.
A discretely ordered ring with the property  \axpam\ref{zringax} is called  \emph{a $\mathbb Z$-ring}. 
In this paper, we will mainly use the consequence that every number is odd or even.
\item  
\axpam\ref{plimp16} is  \emph{the Euclidean Division Principle}.
\item
\axpam\ref{smurferella} is  \emph{the Primality Principle}. 
An integral domain that satisfies  \axpam\ref{smurferella} is  \emph{a pre-Schreier domain}. 
\item
\axpam\ref{smurfin} is  \emph{the Powers Existence Principle}. This principle is called Pow$_2$-IP in \cite[Section 6]{jera:theo24} (modulo a minor difference
in formulation).
\item
\axpam\ref{babysmurf} is  \emph{the Powers Division Principle}. This principle is called Pow$_2$-Div in  \cite[Section 6]{jera:theo24}.
\item
\axpam\ref{supersmurf} is  \emph{the GCD Principle}.
A domain that satisfies \axpam\ref{supersmurf} is  \emph{a GCD domain}. 
\end{itemize}

\begin{remark}{\small
Emil \jer\ shows that $\mathrm I{\sf E}_1 + \forall x\, \exists y \mathbin{>} x\; {\sf pow}_2(y)$ proves 
\axpam\ref{smurfin} and \axpam\ref{babysmurf}. See  \cite[Theorem 6.2]{jera:theo24}. He also shows that  $\mathrm I{\sf E}_1 + \forall x\, \exists y \mathbin{>} x\; {\sf pow}_2(y)$ 
is included in $\mathrm I{\sf E}_2$. Our Appendix~\ref{luiesmurf} gives a quick argument that {\sf IOpen} does not prove \axpam\ref{smurfin}.
This fact also follows from \cite[Example 4.9]{jera:theo24} as we will explain in Appendix~\ref{luiesmurf}.}
\end{remark}

We will meet the following six arithmetical theories in this paper.
\begin{itemize}
\item
  \pamj\  is axiomatised by \axpam1-13. It is the theory of a discretely ordered commutative semiring with least element.
  It is also known as \jer 's ${\sf PA}^-$.
  \item
  The theory \pamp\ is introduced and discussed in Appendix~\ref{negatievesmurf}.
  \item
\pamo\  is axiomatised by \pamj\ plus \axpam\ref{plimp15}. It is the theory of the non-negative part of a discretely ordered commutative ring.
 \item
The theory \pamres\ is axiomatised by \pamo\, + \axpam\ref{plimp16}.
\item
The theory \pamsmu\ is axiomatised by \pamo\, +  \axpam\ref{smurfin} + \axpam\ref{babysmurf}.
\item
The theory {\sf IOpen} of open induction over \pam.
\end{itemize}

The theory  \pamo\ is a good base for weak arithmetical theories. We note that all axioms except \axpam\ref{plimp15} are universal. 
See e.g. \cite{kaye:mode91}.
The system without \axpam\ref{plimp15}, that is our \pamj, was studied by Emil \jer\ in \cite{jera:sequ12}.
Emil \jer\ verifies that we have  a good sequence coding in \pamj, i.o.w., that this theory is sequential. 
We will briefly sketch this below. The theory \pamres\ is a reasonable extension of \pamo. It is well-known that
\pamres\ is a part of {\sf IOpen}, the theory of induction for open formulas. 
See \cite[Lemma 1.15(4), p34]{haje:meta91}.
The theory \pamsmu\ adds two principles
to \pamo\ that are more or less the obvious choices to get Smullyan coding going. These principles also occur in
\jer's paper \cite{jera:theo24}. 

We have the following picture of the theories where the arrows mean \emph{theory extension}.
\[
\begin{tikzcd}[row sep = small, column sep = small]
&&& \pamsmu & \\
 \pamj  \arrow{r}  & \pamp \arrow{r} &   \pam \arrow{ur}\arrow{r} &  \pamres \arrow{r} & {\sf IOpen}
  \end{tikzcd}
\]
There are no further inclusions. For the strictness of the first two arrows, see Appendix~\ref{negatievesmurf}.
For the fact that \pamsmu\ is not part of {\sf IOpen}, see Appendix~\ref{luiesmurf}. The fact that \pam\ does not contain
\pamres, is immediate by considering the models $\mathbb M_0$ and $\mathbb M_1$ of Section~\ref{alitersmurf}. 
The fact that \pamres\ does not contain {\sf IOpen} follows from consideration of the model $\mathbb M_2$  of
Section~\ref{hulksmurf}. 
Finally, the fact that \pamsmu\ does not extend \pamres\ was proved by 
Naoya Kato. See \cite{kato:nond26}.

\subsection{Emil Je\v{r}\'abek  on the $\upbeta$-function}\label{betasmurf}
In this subsection, we briefly review the results of \cite{jera:sequ12}. 

\begin{lem}[\pamj, Je\v{r}\'abek]\label{vreugdigesmurf}
\begin{enumerate}[i.]
\item
$\vdash (x\leq y \wedge y \leq x) \to x=y$
\item
$\vdash x+z \leq y+z \to x\leq y$
\item
$\vdash 0\leq x$
\item
$\vdash (z \neq 0 \wedge x \cdot z \leq y\cdot z) \to x \leq y$
\item
 $\vdash x\leq y+1 \iff (x\leq y \vee x=y+1)$.
\end{enumerate}
\end{lem}

The next lemma gives  the functionality of Euclidean Division, but not the existence of it.

\begin{lem}[\pamj, \jer]\label{unismurf}
Suppose $zy+r = z'y+r'$, $r <y$, and $r'<y$. Then, $z=z'$ and $r=r'$.
\end{lem}

This means that the $z$ and $r$ claimed to exist in \axpam\ref{plimp16} are unique.
We also note that it implies, as a special case, that no number is both odd and even.

We briefly present \jer's results for \pamj\ concerning pairing and sequence coding.
First, \jer\ proves that $\tupel{x,y} := (x+y)^2+x$ is a pairing function in \pamj.
We note that, in \pamj, we have $x\leq \tupel{x,y}$ and $y \leq \tupel{x,y}$.
We consider the following version of the $\upbeta$-function.
\begin{itemize}
    \item {\small
   $\upbeta(x,i,w) :\iff \exists u,v,q\, (w=\tupel{u,v} \wedge u= q(1+(i+1)v)+x \wedge x \leq (i+1)v)$. 
   }
\end{itemize}
Clearly, $u$, $v$ and $q$ can be bounded by $w$. So, this gives us a $\Delta_0$-definition or, more precisely,
an ${\sf E}_1$-definition.
Also, necessarily $x \leq w$.
\cite[Lemma 4(iii)]{jera:sequ12} tells us that $\upbeta$ is functional from $i,w$ to $x$.

We define:
\begin{itemize}
    \item 
    ${\sf seq}^\ast(s) :\iff \exists z,w \bleq s\, (s = \tupel{z,w} \wedge \forall i \bles z \,\exists x\bleq w  \, \beta(x,i,w))$. 
\end{itemize}

In case $s$ is an $\ast$-sequence, we write ${\sf length}(s)$ for the first component of $s$.
Clearly, the length function has a $\Delta_0$-graph. We also write $\uppi(s,i) =x$, whenever $s$ is an $\ast$-sequence $\tupel{n,w}$, $i<n$, and
$\upbeta(x,i,w)$. We note that the graph of $\uppi$ is $\Delta_0$ and $\uppi$ is total on the $i<n$, whenever $s$ is an $\ast$-sequence.

A crucial result is \cite[Lemma 8]{jera:sequ12}. In our terminology, this says that, in \pamj, for a certain definable inductive class $J$, we have: (\dag)
for all $\ast$-sequences $s$ where the length of $s$ is $z$ in $J$ and for all $x$, there is an $\ast$-sequence $s'$ of length $z+1$, such that
$\uppi(s,i)=\uppi(s',i)$ for all $i<z$ and $\uppi(s',z)=x$. 

\begin{rem}{\small
At first sight, Jerabek's result seems pretty miraculous. Inspecting a traditional presentation of the $\upbeta$-function,
we see that it employs large numbers that are exponential both in the length and in the components of the sequence.
Since we want sequences of all numbers, this seems to tell us that we really need the totality of exponentiation. But we do not have it in our weak
theories.

A closer look at the verification of the correctness of the $\upbeta$-function reveals that the large numbers need to be exponential only in the
 \emph{length} of the sequence.
This opens the road towards considering two sorts of numbers: \emph{the surface numbers} which qualify as sequence components and \emph{the grotto numbers}
that live in the depths which can function as lengths of the sequence. We can give the grotto numbers many desirable extra properties. For example, \jer's grotto numbers
do satisfy the subtraction principle \axpam\ref{plimp15}. 

Of course, this rough idea leaves many details open which are solved in \jer's paper.}
\end{rem}

\subsection{Further Remarks on the $\upbeta$-Function}\label{moresmurf}
The discovery of the $\upbeta$-function was prompted by
John von Neumann's question whether one could prove G\"odel's result purely arithmetically.
We submit that G\"odel's answer is of great beauty. It uses a classical insight in number theory, the
Chinese Remainder Theorem, and the solution he produced is quite simple in nature.

If we inspect \jer's presentation of the graph of the function above, one sees that the existential quantifiers
can all be bounded by $w$. This makes the definition of the graph ${\sf E}_1$: a block of bounded existential
quantifiers followed by a quantifier-free formula. This enables G\"odel's sequence-coding
to function in quite weak theories.

This is a good occasion to direct the reader to the paper  \cite{bezb:gode76} by Amala Bezboruah and John Shepherdson.
The paper shows that the theory \pamo\ does not prove the consistency of \emph{any} theory, let alone itself, under very light
conditions on the proof system, where the proofs are coded as sequences using the $\upbeta$-function. It is striking how
effortless and natural the argument is, which illustrates that the $\upbeta$-sequences work very well in mathematically
meaningful models of weak theories. The paper also contains, as for as I know, the only discussion of the rather obvious question
why G\"odel used sequences coded by products of powers of prime numbers, when he already had the $\upbeta$-sequences.
This is especially puzzling in the light of the fact that G\"odel needs the $\upbeta$-sequences to define the
exponents-of-primes-sequences. 

For further discussion see also \cite{viss:bezb26}. In this paper, I also show that we can get a Bezboruah-{\&}-Shepherdson-style
result using the Markov coding of sequences. This result works for the theory \pamo\ plus all true universal sentences.

\subsection{Further Useful Insights}\label{megasmurf}
In this subsection, we collect a number of useful insights.

\begin{theorem}[\pamo] \label{lilasmurf}
Suppose $x \mid ax+b$. Then, $x\mid b$. 
\end{theorem}

\begin{proof}
The case that $x=0$ is immediate. So assume $x\neq 0$.
Suppose $cx = ax+b$.  It follows that $ax\leq cx$ and, hence, that $a\leq c$. 
So, $c=a+d$, for some $d$. It follows that $(a+d)x= ax+b$. \emph{Ergo}, $dx=b$.
\end{proof}

\noindent
One way of reading Theorem~\ref{lilasmurf} is simply that we have  distributivity for subtraction. 

\begin{theorem}[\pamj\ plus \axpam\ref{plimp16}]
We have the Subtraction Principle \axpam\ref{plimp15}.
In other words,  \pamj\ plus Euclidean Division 
coincides with \pamres.
\end{theorem}

\begin{proof}
Naoya Kato proves that $\pamj$ plus \axpam\ref{plimp16} proves the predecessor principle \axpam\ref{presmurf}.
See \cite[Appendix A]{kato:equi26}.  So, it suffices to prove  \axpam\ref{plimp15}  in \pamj\ plus \axpam\ref{presmurf} plus \axpam\ref{plimp16}.

Suppose $y \leq x$. If $y=0$, we have $y+x=x$. Suppose $y \neq 0$.
There is a $z$ and an $r<y$ such that $zy+r= x$. If 
$z=0$, we have $r=x$, so $x<y$. \emph{Quod non}. If $z= u+1$, we have
$y + (uy+r) = x$. So, $uy+r$ is the desired difference of $x$ and $y$.
\end{proof}

We note that, trivially, we have the uniqueness of the greatest common divisor.
An element $x$ of a domain is \emph{primal} if $x\mid yz$, then there are $u$ and $v$ such that
$x=uv$ and $u\mid y$ and $v\mid z$. So, modulo translation to our context, \axpam\ref{smurferella} says that every element is primal.
It is well-known, that, in a domain, the existence of all greatest common divisors implies that all elements are primal. 

We remind the reader that a domain that satisfies \axpam\ref{supersmurf} is a \emph{GCD domain}. An integral domain that
satisfies  \axpam\ref{smurferella} is a \emph{pre-Schreier domain}. So, an integral GCD-domain is pre-Schreier.
We can extend models of $\pamo$ by a standard construction to a discretely ordered ring. The construction
allows us to transfer the insight that GCD implies pre-Schreier. So, over \pamo, we find that  \axpam\ref{supersmurf} 
implies \axpam\ref{smurferella}.  Inspecting the proofs, we see that even \pamj\ supports the argument.
We repeat almost verbatim two arguments given in stackexchange, see:
{\small
\begin{itemize}
\item
{\tt https://math.stackexchange.com/questions/1137305/}\\
{\tt divisor-of-a-product-of-integers-is-a-product-of-divisors},
\item
{\tt https://math.stackexchange.com/questions/705862/}\\
{\tt prove-that-ma-mb-ma-b-gcd-lcm-distributive-law}.
\end{itemize}
}

\begin{theorem}[\pamj\ plus \axpam\ref{supersmurf}]
We have the Primality Principle  \axpam\ref{smurferella}.
\end{theorem}

\begin{proof}
We first prove the distribution law for the greatest common divisor ${\sf gcd}(x,y)$.
We have ${\sf gcd}(zx, zy) = z\cdot  {\sf gcd}(x, y)$. If $z=0$, this is immediate. Suppose $z\neq 0$. Clearly, $z$ divides
 ${\sf gcd}(zx, zy)$. Say, $w$ is the unique number such that $zw =  {\sf gcd}(zx, zy) $.
 We have:
 \begin{eqnarray*}
 u \mid {\sf gcd}(x,y) & \iff & u\mid x  \wedge u \mid y \\
 & \iff & zu \mid zx \wedge zu \mid zy \\
 & \iff & zu \mid {\sf gcd}(zx,zy) \\
 & \iff & u \mid w
 \end{eqnarray*}
 It follows that ${\sf gcd}(x,y) =w$. Ergo, $z\cdot {\sf gcd}(x,y) = {\sf gcd}(zx,zy)$.
 
 Now suppose  $x\mid yz$. We may assume that none of $x,y,z$ is 0.
  Let $u:= {\sf gcd}(x,y)$ and let $v$  be the unique number such that
 $x= uv$.  We have: $x\mid xz$ and $x\mid yz$. Hence, $x \mid {\sf gcd}(xz,yz) = {\sf gcd}(x,y)z =uz$.
 It follows that $uv \mid  uz$. So, $v \mid z$.  
 \end{proof}

We turn to some insights concerning powers of 2. We briefly consider an alternative definition of being
a power of 2.
\begin{enumerate}[A.]
\item
${\sf pow}_2(x)$ iff $\forall y \, (y\mid x \to (y=1\vee \num 2\mid y))$.
\item
${\sf pow}_2^{\sf tar}(x)$ iff $\forall y \,\forall z\;  x\neq (\num 2\cdot y+\num 3)\cdot z$.
\end{enumerate}

The first definition is used by Raymond Smullyan in his classical book \cite{smul:theo61}. It is also used by, e.g.,
 Edward Nelson in \cite{nels:pred86} and by Emil \jer\ in \cite{jera:theo24}. 
 \jer\ calls the $x$ satisfying ${\sf pow}_2$ \emph{oddless}. The second definition is due to Tarski.
 
 The next two theorems give the relationship between the two definitions.

\begin{theorem}[\pamj]
Any $x$ in ${\sf pow}_2$ is also in ${\sf pow}^{\sf tar}_2$.
\end{theorem}

\begin{proof}
Suppose $x$ is in ${\sf pow}_2$ and $2y+3$ divides $x$. It follows that $2$ divides $2y+3$.
So, $2y+3$ is both even and odd. This is impossible by Theorem~\ref{unismurf}.
\end{proof}

\begin{theorem}[\pamj + \axpam\ref{zringax}]
${\sf pow}_2$ and ${\sf pow}^{\sf tar}_2$ coincide.
\end{theorem}

\begin{proof}
Suppose $x$ is in ${\sf pow}^{\sf tar}_2$ and $y$ divides $x$.
In case $y$ is even, we are done. In case $y$ is odd,  it must be 1. 
\end{proof}

We note that in $\mathbb Z[\X]^{\mathsmaller{\geq 0}}$ we have that $\X$ is a Tarski power of 2, but not a Smullyan one.
So, the two notions are not the same over $\pamo$.

We end this subsection with three insights concerning powers of 2. 
Our first theorem is, in essence,  contained in the proof of  \cite[Theorem 6.2]{jera:theo24}.
Our proof is the same as \jer's proof.

\begin{theorem}[\pamj\  plus \axpam\ref{supersmurf}]\label{mangasmurf}
Suppose ${\sf pow}_2(x)$ and ${\sf pow}_2(y)$ and $x\leq y$. Then, we have $x\mid y$.
I.o.w., we have \axpam\ref{babysmurf}.
\end{theorem}

\begin{proof}
Suppose $x$ and $y$ are powers of 2 and $x\leq y$.
Let $d$ be the greatest common divisor of $x$ and $y$.
Say, we have $x = dw$ and $y=dz$. We find $w\mid x$ and $z \mid y$.
Since $w$ and $z$ must be relative prime, they cannot be both be divisible
by $2$. Hence, at least one of them is $1$. Since, $x \leq y$, we find that $w=1$.
\end{proof}

\begin{theorem}[\pamj\ plus \axpam\ref{smurferella}]\label{hapjessmurf}
Suppose ${\sf pow}_2(x)$ and ${\sf pow}_2(y)$. Then, ${\sf pow}_2(x\cdot y)$.
\end{theorem}

\begin{proof}
Suppose $z \mid x\cdot y$. By \axpam\ref{smurferella}, we have $z = u\cdot v$ and $u \mid x$ and $v \mid y$.
In case $u$ and $v$ are both 1, we find $z=1$, and we are done. Suppose, e.g., $u\neq 1$. Then, since $u\mid x$,
we have $2\mid u$ and, hence, $2\mid z$. The case $v\neq 1$ is symmetric.
\end{proof}

We can prove the closure of the powers of 2 under multiplication in a second way from different principles.

\begin{theorem}[\pamsmu]\label{delicatessesmurf}
Suppose ${\sf pow}_2(x)$ and ${\sf pow}_2(y)$. Then, ${\sf pow}_2(x\cdot y)$.
\end{theorem}

\begin{proof}
Suppose $x$ and $y$ are powers of 2. 
Let $u$ be a power of 2 such that $u \leq xy+1 < 2u$.
In case $x=1$ or $u=1$, we are easily done. Suppose
$x\neq 1$ and $u \neq 1$. Then, both $x$ and $u$ are even,
and, hence $xy+1$ is odd. It follows that $u\neq xy+1$. 
We may conclude $u \leq xy <2u$. We consider two cases.

Suppose $u\leq x$. Then, since both $x$ and $u$ are powers of 2,   for some $w$, we have $uw=x$.
We note that $w\neq 0$. We find $u \leq uwy< 2u$. So, $1\leq wy < 2$. So, $y=1$ and, thus,
$xy=x$. So, $xy$ is a power of 2.

Suppose $x\leq u$. Then, $xv=u$ for some $v$.
We find $xv\leq xy < 2xv$. So, $v\leq y <2v$.
Trivially, since $v$ divides $u$, we find that $v$ is a power of 2. So,
$vz=y$, for some $z$. Clearly, $z\neq 0$.
We have $v\leq vz < 2v$, so $1\leq z <2$. \emph{Ergo},
$z=1$ and $v=y$. Hence, $u=xv=xy$ and we may conclude that
$xy$ is a power of 2.
\end{proof}

\section{String Theory}

We switch our perspective from numbers to strings. We present a number of principles concerning strings, formulate
a selection of theories, and derive some consequences of the theories.

\subsection{Language and Principles}
The string-language has a constant $\estr$, for the empty string, 
a binary operation $\ast$, for concatenation.

Here are some basic abbreviations.
\begin{itemize}
\item
$x\preceq y$ iff $\exists u\,\exists v\,\; y = (u\ast x) \ast v$.
\item
$x\preceq_{\sf ini}y$ iff $\exists v\;\,  y = x \ast v$.
\item
$x\preceq_{\sf end}y$ iff $\exists u\;\,  y = u \ast x$.
\item
$a$ is an atom or ${\sf atom}(a)$ iff $a \neq \estr$ and, for all $x$ and  $y$, 
if $x\ast y = a$, then $x=\estr$ or $y= \estr$.
\end{itemize}

Here is the list of principles that we will consider in this article.

\begin{enumerate}[\axtc1.]
\item
$\vdash \estr \ast x = x \wedge x\ast \estr = x$
\item
$\vdash x\ast y = \estr \to (x=\estr \wedge y = \estr)$
\item 
$\vdash (x\ast y) \ast z = x\ast (y \ast z)$
\item\label{weakeditors}
$ \vdash  x \ast y = u \ast v \to  
\exists w\;(x\ast w = u \vee x= u \ast w)$
\item\label{editors}
$\vdash x \ast y = u \ast v \to  
\exists w\;((x\ast w = u \wedge y = w\ast v) \vee
(x= u \ast w \wedge  w\ast y = v))$
\item\label{wcancellsmurf}
$\vdash x\ast y=x \to y=\estr$
\item\label{cancellsmurf}
$\vdash x\ast y=x\ast z \to y=z$
\item
$\vdash x=\estr \vee \exists y\, \exists a\, (\ato(a) \wedge x=y\ast a)$ \label{thorsmurf}
\end{enumerate}

\noindent
Some of our principles have names:
\begin{itemize}
\item
Principle \axtc\ref{weakeditors} is \emph{the Weak Editors Principle}.
\item
 Principle \axtc\ref{editors} is \emph{the Editors Principle}.
 \item
 Principle \axtc\ref{wcancellsmurf} is \emph{the Weak Left Cancellation Principle.}
 \item
 Principle \axtc\ref{cancellsmurf} is \emph{the Left Cancellation Principle}.
 \item
 Principle \axtc\ref{thorsmurf} is \emph{the Stack Principle}.
\end{itemize}
We note that the Weak Editors Principle  \axtc\ref{weakeditors} is equivalent to the proposition that $\preceq_{\sf ini}$ is \emph{semi-linear} (a.k.a. tree-ordered):
\begin{itemize}
\item
$\vdash (x\preceq_{\sf ini} z \wedge y \preceq_{\sf ini} z) \to (x\preceq_{\sf ini} y \vee y \preceq_{\sf ini} x)$. 
\end{itemize}

Of course, there are also the obvious principles of Weak Right Cancellation and Right Cancellation. We will not use these in our development,
but, of course, this is just an accidental feature of our design choices.
A principle \emph{Bi-Cancellation} that strictly extends both Left and Right Cancellation will be studied in
Appendix~\ref{bismurf}.

We proceed with the introduction of our string theories.
Our base theory $\tc0$ is given by Axioms \axtc1-3.    The theory \tc1\ is axiomatised by \tc0 plus the plus the Editors Principle \axtc\ref{editors}. 
The theory \tc2\ is \tc1\ plus the Stack Principle  \axtc\ref{thorsmurf}.

A superscript {\sf c} after the name of a theory means addition of the Left Cancellation, i.e.\ Axiom
\axtc\ref{cancellsmurf}, to the given theory without the superscript.

It will be handy to have a notation for adding atoms to string theories. Suppose $U$ is a string theory.
We write, e.g., $U({\tt a}_0,\dots, {\tt a}_k)$ for $U$ where the language is extended with
constants ${\tt a}_i$, for $i\leq k$ and the axiom set is extended with the axioms
$\vdash \ato({\tt a}_i)$, for $i\leq k$, and $\vdash {\tt a}_i\neq {\tt a}_j$, for
$i<j\leq k$. We write  $U[{\tt a}_0,\dots, {\tt a}_k]$, for  $U({\tt a}_0,\dots, {\tt a}_k)$ extended with
$\vdash \forall x\, (\ato(x) \to  \bigvee_{i\leq k}x={\tt a}_i)$. 

In contexts where we do not also have multiplication of numbers to cause ambiguity, we will often omit $\ast$, writing $xy$ for $x\ast y$.

\subsection{Basics of \tc1}\label{goedzaksmurf}
In this subsection, we prove some basic facts the theory \tc1.
The theory \tc1\ is part of a theory introduced by
Alfred Tarski in \cite{tars:wahr35}. Its extension $\tc1({\tt a},{\tt b})$ is a variant of a
 theory of strings employed by Andrzej Grzegorczyk in  \cite{grze:unde05}.
Grzegorczyk's version does not have the empty string. In \cite[Appendix A]{viss:grow09}  it is shown that
$\tc1({\tt a},{\tt b})$ is bi-interpretable with Grzegorczyk's theory. See e.g.  \cite{corc:stri74} and 
 \cite{viss:grow09} for  further background.

We have both left and right cancellation \emph{for atoms} in \tc1.

\begin{theorem}[\tc1]\label{hippesmurf}\label{smeuigesmurf}
Suppose $a$ and $b$ are atoms. 
\begin{enumerate}[i.]
\item
If $xa=yb$, then $x=y$ and $a=b$.
\item
If $ax=by$, then $a=b$ and $x=y$.
\end{enumerate}
\end{theorem}

\begin{proof}
We prove (i).
Suppose $xa = yb$. 
For some $w$, we have ($xw=y$ and $a= wb$) or ($x=yw$ and $wa=b$).
Since $a\neq \estr$ and $b\neq \estr$, it follows in both cases that $w=\estr$ and, hence $x=y$ and $a=b$.

The proof of (ii) is similar.
\end{proof}

We can recover simple theories for sets of atoms. The atoms are treated here both as elements and  as singletons or, if one wishes, as auto-singletons.
The $\in$-relation is simulated by $\preceq$. 
We have the following insights.

\begin{theorem}[\tc1] \label{verenigingssmurf}
\begin{enumerate}[i.]
\item
Suppose $a$ is an atom, then $a\not\preceq \estr$.
\item
Suppose $a$ and $b$ are atoms. Then, $b\preceq a$ iff $a=b$.
\item
Suppose $b$ is an atom. Then $b\preceq xy$ iff $b\preceq x$ or $b\preceq y$.
\end{enumerate}
\end{theorem}

\begin{proof}
We prove (iii).
The right-to-left direction is immediate. Suppose $xy = (ub)v$. Then,
there is a $w$ such that (a) $xw =ub$ and $y=wv$ or (b)
$x=ubw$ and $wy=v$.

We consider Case (a). There is a $z$ such that (aa) $xz=u$ and $w=zb$ or
(ab) $x=uz$ and $zw=b$. In case (aa), we find $b\preceq y$. In case (ab),
we find $z=b$ and, hence, $b\preceq x$, or $w=b$ and, hence $b\preceq y$. 

In Case (b), we have $b\preceq x$.
\end{proof}

\begin{remark}\label{varkentjesmurf}{\small
We can read Theorem~\ref{verenigingssmurf}\textup(i\textup) as saying that $\estr$ is an empty set \textup(of atoms\textup). It is important to note
that there may be other such empty sets unequal to $\estr$. Similarly, Theorem~\ref{verenigingssmurf} provides
functional singletons and functional union. Again, this does not mean that the functionally provided singletons
are all singletons and that there are no other unions apart from the provided ones.}
\end{remark}

\subsection{\tc1 meets other Principles}
This subsection gives some information about the interaction of  \tc1\ with other principles.

In the presence of Left Cancellation \axtc\ref{cancellsmurf},  the Editors Principle follows from the
Weak Editors Principle \axtc\ref{weakeditors} over $\tc 0$.

\begin{theorem}[ \tcc0 + \axtc\ref{weakeditors}]\label{kaartjessmurf}
We have  the Editors Axiom, \axtc\ref{editors}.
\end{theorem}

\begin{proof}
We work in  \tcc0\ plus the Weak Editors Principle.
Let $xy=uv$. Then, there is a $w$ such that  $xw=u$ or  $x=uw$.
In the first case, we have $xwv = xy$. So, $wv=y$. In the second case,
we have $uwy = uv$. So $wy=v$.
\end{proof}

Left Cancellation will make the refinement guaranteed by the Editors Principle unique
as shown in the next theorem.

\begin{theorem}[\tcc1]\label{lokismurf}
Suppose $xy=uv$. Then, there is a unique $w$ such that
\textup(i\textup) $xw=u$ and $y=wv$ or \textup(ii\textup) $x=uw$ and $wy=v$.
Moreover, if both cases obtain, we have $w=\estr$.
\end{theorem}

\begin{proof}
Suppose $xy=uv$. Then, there is a $w$ such that (i) $xw = u$ and $y=wv$ or
(ii) $x=uw$ and $wy=v$.
Suppose we also have (i$'$) $xw' = u$ and $y=w'v$ or
(ii$'$) $x=uw'$ and $w'y=v$.

If we have case (i) and (i$'$), we find $xw=xw'$, so $w=w'$. Similarly, for (ii) and (ii$'$).
Suppose we have (i) and (ii$'$). So, $xw=u$ and $x=uw'$. It follows that $x=xww'$ and, hence,
that $ww'=\estr$, from which we may conclude that $w=w'=\estr$.
The case of (i$'$) and (ii) is similar.
\end{proof}

In \tc1, we can simplify Left Cancellation, \axtc\ref{cancellsmurf}, to Weak Left Cancellation, \axtc\ref{wcancellsmurf}.

\begin{theorem}[\tc1+ \axtc\ref{wcancellsmurf}]\label{villeinesmurf}
We have  Left Cancellation, \axtc\ref{cancellsmurf}. 
\end{theorem}

\noindent
In other words, \tc1+ \axtc\ref{wcancellsmurf} is 
$\tc1^{\sf c}$.

\begin{proof}
We work in \tc1\ plus \axtc\ref{wcancellsmurf}.
Suppose $xy=xz$. Then, for some $w$, we have either $xw =x$ and $y=wz$, or $x=xw$ and $wy=z$. In both cases
it follows that $w=\estr$.
\end{proof}

\begin{theorem}
\begin{enumerate}[i.]
\item
In \tc1, $\preceq_{\sf ini}$ is a semi-linear tree-ordering, i.e.  $\preceq_{\sf ini}$ is reflexive, transitive and has the tree-property:
\begin{itemize}
\item
$(x\preceq_{\sf ini} z \wedge y \preceq_{\sf ini} z) \to (x\preceq_{\sf ini} y \vee y\preceq_{\sf ini}x)$.
\end{itemize}
\item
In \tcc1, $\preceq_{\sf ini}$ is a linear tree-ordering, i.e., it is a semi-linear tree-ordering that anti-symmetric.
\end{enumerate}
\end{theorem}

\noindent
We leave the simple proof to the reader. A semi-linear (pre)ordering is also known as a tree-(pre)ordering or a 
forest (pre)ordering.

We define, in the language of \tc0({\tt b}):
\begin{itemize}
\item
$x$ is {\tt b}-\emph{free} if $\tt b \not\preceq x$. 
\end{itemize}

In one of our our intended models, we have precisely two letters {\tt a} and {\tt b}. In that model, a {\tt b}-free string is an
{\tt a}-string. However, in our formal development, we do not need the {\tt a}.
We note that in $\tc1({\tt b})$, by Theorem~\ref{verenigingssmurf}, the {\tt b}-free strings are
closed under concatenation.

\begin{theorem}[$\tc1({\tt b})$]\label{scherpesmurf}
Suppose $x$ and $y$ are {\tt b}-free and $x{\tt b}u = y{\tt b}v$. Then $x=y$ and $u=v$.
\end{theorem}

\begin{proof}
We have either (a) $x{\tt b}z = y{\tt b}$ and $u =zv$ or (b)  $x{\tt b} = y{\tt b}z$ and $zu =v$. We treat case (a).
Case (b) is entirely analogous. 

We have (aa) $x{\tt b}w=y$ and $z=w{\tt b}$ or (ab) $x{\tt b} = yw$ and $wz= {\tt b}$. Case (aa) drops out since $y$ is {\tt b}-free.
In Case (ab), we have cases (aba) $w= \estr$ and $z={\tt b}$ or (abb) $w= {\tt b}$ and $z=\estr$.
In Case (aba), we find $x{\tt b} = y$, contradicting the fact that $x$ is {\tt b}-free. Finally, in Case (abb),
we find, $x{\tt b}= y {\tt b}$ and $u=v$, and, so, by Theorem~\ref{smeuigesmurf}, $x=y$ and $u=v$ as promised.
\end{proof}

\begin{remark}{\small
It is easy to prove that our theory $\tc1({\tt b})$  only has infinite models.
We have, $\tc1({\tt b})$-verifiably,   ${\tt b}^n \neq {\tt b}^m$, whenever $n\neq m$.
However, it does have a decidable extension, since, modulo the obvious translation,
Presburger Arithmetic extends $\tc1({\tt b})$.}
\end{remark}

\subsection{Occurrences and Profiles}
Theorem~\ref{verenigingssmurf} shows that, in \tc 1, we can view strings as sets of atoms, with
empty-set, singletons and union. In this section, we discuss the more refined perspective that strings
can be viewed as multi-sets with ordered occurrences. We show that we can view strings like this over
\tcc 1. We define:
\begin{itemize}
\item
$x$ is an \emph{occurrence of atom $a$} iff $x$ is of the form $ya$. 
\item
$x$ is an occurrence of $a$ \emph{in $z$} if $x$ is an occurrence of $a$ and $x\preceq_{\sf ini} z$.
\end{itemize}

By Theorem~\ref{hippesmurf}, over \tc1, every occurrence is the occurrence of a unique atom.
The Stack Principle \axtc\ref{thorsmurf} tells us that the occurrences are precisely the non-empty strings. 

In \tc1,  the occurrences in a string $z$ form a semi-linear preorder and,
in \tcc1, this becomes a linear order. 

In the context of a model, we can  make pictures of the strings. We call such pictures  \emph{profiles}.
We work in a model $\mc M$ of $\tc {1}$.
\begin{itemize}
\item
\emph{A labeled partial preorder type} is given by a partial preorder plus a mapping from the elements to atoms, modulo
the following equivalence relation: $\tupel{X, \leq, F}$ is equivalent to $\tupel{Y, \sqsubseteq, G}$ iff there is a bijection
$\phi$ between $X$ and $Y$ such that we have (i) $x\leq x'$ iff $\phi(x) \sqsubseteq \phi(y)$, and (ii) $F(x) =G(\phi(x))$.
\item
 \emph{A profile} ${\sf pro}(s)$ of a string $s$ is given as the labeled partial preorder type of the set of occurrences in $s$ ordered by $\preceq_{\sf ini}$, labeled by their last atom.
\end{itemize}

We note that $u$ is an occurrence of $a$ in $st$ iff  $u$ is an occurrence of $a$ in $s$ or $u=sw$, where $w$ is an occurrence of $a$ in $t$.
However, the profile of $st$ need not be the sum of the profiles of $s$ and $t$, since the same occurrence may be both in the $s$-part and the $t$-part.
If we also demand Left Cancellation, the behaviour of the occurrences becomes the behaviour we know and love. 
The preordering becomes an ordering and  
 ${\sf pro}(s t)= {\sf pro}(s)+{\sf pro}(t)$, where $+$ is the obvious composition of labeled partial order types. 

We note that properties like \emph{being {\tt b}-free} can be viewed as properties of the profile of the string.

In Sections~\ref{hulksmurf} and \ref{alitersmurf}, we wil provide examples of profiles. We will see that the same profile can correspond to different strings.

\begin{rem}\label{leftrightsmurf}
{\small  In this paper we use \emph{left occurrences}. We could, of course, also have used \emph{right occurrences}.
A right occurrence in a string that starts with an atom. A right occurrence $x$ is an occurrence in a string $y$ iff
$x\preceq_{\sf end} y$. We define the following relation between left and right occurrences in $z$:
\begin{itemize}
\item
 $x\smile_z y$ iff, for some $u$, $v$, $a$, we have: $a$ is an atom, $x=ua$, $y=av$, and $z=uav$.
 \end{itemize}
 We note that $\smile_z$ is total and surjective between initial and end strings of $z$. Moreover,
 If $x_0 \smile_z x_1$ and $y_0\smile_z y_1$, then $x_0\preceq_{\sf ini}y_0$ iff $y_1\preceq_{\sf end}x_1$.
 
 If we have Left Cancellation $\smile_z$ is functional and if we have
 Right Cancellation $\smile_z$ is injective.  
 It follows that, in the presence of both Left and Right Cancellation, switching from left to right occurrences will simply invert the
 order of the profile.}
\end{rem}

\subsection{Partitions of Strings}
Following an idea of Pavel Pudl\'ak, we can draw a very useful consequence from the Editors Axiom \axtc\ref{editors}.
In this section, we present the basic idea.

Consider any model of \tc1.
Fix an element $w$. A \emph{partition} of $w$ is a sequence $\alpha = (w_0,\ldots,w_{k-1})$, where the $w_i$ are non-empty and
 $w_0 \ast \dots \ast w_{k-1}=w$.
We allow the empty partition, say, $\pemp$. We write $\alpha\beta$ or $\alpha \pcomp \beta$ for concatenation of partitions.\footnote{\emph{Par abus de langage}, 
we use the same variables for partitions as  for container strings.}

Let $\psing w$ be $\pemp$ if $w=\estr$, and $(w)$ if $w\neq \estr$. We define: $\eva{w_0,\ldots,w_{k-1}} = w_0\dots w_{k-1}$.

The partitions of $w$  form a category with the following morphisms.
We have $f:(u_0,\ldots, u_{n-1})\to(w_0,\ldots,w_{k-1})$ iff
$f$  is a surjective and weakly monotonic function from $n$  to $k$,
such that, for any $i< k$, we have $w_i = u_s \cdots u_{\ell}$,
where  $ f(j)=i$ iff   $ s\leq j \leq \ell$.
We write $\alpha\leq \beta$ 
for: $\exists f\; f:\alpha\to\beta$.
In this case we say that $\alpha$
is {\em a refinement}   of $\beta$.

The following insight is trivial.
\begin{theorem}
The relation $\leq$ is a partial ordering on partitions. 
\end{theorem}

The following theorem is Theorem~3.5 of \cite{viss:grow09}.
\begin{theorem}\label{plezantesmurf}
Consider a $w$  in a model $\mc M$  of \tc1.
 Then, any two partitions $\alpha$ and $\beta$ of $w$ have a common refinement $\gamma$ in $\mc M$.
\end{theorem}

\begin{proof}
We work in  a model of \tc1. 
We prove our Theorem by course-of-values induction on the sum of the lengths of our partitions, where we construct
$\gamma$ recursively. If the sum of lengths is zero, we are immediately done, taking $\gamma=\oslash$.

Suppose the sum of the lengths is ${>}\,0$. In this case each of the lengths is ${>}\,0$. Let $\alpha$ and $\beta$ be partitions of $w$, say with length
$n$ respectively $k$.
Suppose $\alpha=\alpha_0\pcomp (u)$ and $\beta = \beta_0\pcomp(v)$.
 By the Editors Axiom, we can find a $z$ 
such that  (a) $\eva{\alpha_0} z = \eva{\beta_0}$  and 
$u= z v$, or (b) $\eva{\alpha_0} = \eva{\beta_0} z$   and $z u= v$.
We note that $z$ is not uniquely determined and that  both (a) and (b) may apply simultaneously.
We choose some such  $z$ and one case that applies.

By  symmetry, we only need to consider case (a).
Let $\alpha_1 := \alpha_0 \pcomp \psing{z}$. We find that $\eva{\alpha_1} = \eva{\beta_0}$.
By the induction hypothesis, there is a common refinement
$\gamma_0$ of  $\alpha_1$ and $\beta_0$. Let the length of $\gamma_0$ be $m$.
Suppose $f_0$ and  $g_0$ witness that $\gamma_0$ is a common refinement.
It is easily seen that  $\gamma :=\gamma_0 \pcomp (v)$ is
the desired refinement with witnessing functions $f$  and $g$, where
$f:=f_0[m:n-1]$, $g:=g_0[m:k-1]$.
\end{proof}

Inspection of the argument shows that, if the partitions are non-empty and have, respectively, length $n$ and $k$,
then we can find a common refinement of length at most $n+k-1$ elements. Thus, the insight can also
be expressed by a big disjunction in \tc1\ without the use of the model.
The model is just there to avoid big disjunctions.

A further insight is that the result of our construction has the following property. If $x_i$ and $x_j$ are both mapped
to $u_s$ and both mapped to $v_t$, then $i=j$. In other words, if the embeddings are $f^\ast:\xi \to \nu$  and $g^\ast:\xi\to \rho$, then the mapping
$i \mapsto \tupel{f^\ast(i),g^\ast(i)}$ is injective. 

Theorem~\ref{plezantesmurf} is very useful, since it facilitates easy visualisation of arguments involving
the Editors Axiom. 

Our construction is indeterministic. If we add Left Cancellation the situation changes dramatically.
We discuss these matters in Appendix~\ref{smulpaapsmurf}.

\section{Container Strings}\label{neandersmurf}
The basic language of container strings has two sorts: the object sort $\mf o$ and the string sort $\mf s$.
We will use $x,y,z, \dots$ to range over the objects/urelements and $\alpha,\beta,\gamma,\dots$ to range over
the container strings. On the container string sort, we have basically ordinary string theory, but it is pleasant to diverge a bit notationally in order to stress that
the container strings are containers.
We have a constant $\oslash$ of type $\mf s$, a unary operation $[\cdot]$ of type $\mf{os}$, and a binary operation
$\star$ of type $\mf{ss}$. We will have an extended language in which we have the Frege function $\digamma$ of type $\mf{so}$.
We consider the following principles.
 
 \begin{enumerate}[\axtcu1.]
\item
$\vdash \oslash \star \alpha = \alpha \wedge \alpha\star \oslash = \alpha$
\item
$\vdash \alpha \star \beta = \oslash \to (\alpha=\oslash \vee \beta = \oslash)$
\item 
$\vdash (\alpha \star \beta) \star \gamma = \alpha\star (\beta \star \gamma)$
 \item
$\vdash \alpha \star \beta = \gamma \star \delta \to  
\exists \eta\;((\alpha\star \eta = \gamma \wedge \beta = \eta\star \delta) \vee
(\alpha= \gamma \star \eta \wedge \eta\star \beta = \delta))$ \label{ureditors}
\item
$\vdash {\sf atom}([x])$
\item
$\vdash [x]=[y] \to x=y$ \label{injectiesmurf}
\item\label{urthor}
$\vdash \alpha = \oslash \vee \exists \beta\,\exists x\; \alpha=\beta\star[x]$ 
\item
$\vdash \digamma(\alpha)=\digamma(\beta) \to \alpha=\beta$ \label{fregesmurf}
\end{enumerate}

We have the following theories of container strings in the base language: 
 \utc 1\ is axiomatised by
\axtcu 1-6. The theory \utc{2} is axiomatised as \utc 1 plus \axtcu\ref{urthor}. 
The theories of Fregean container strings  \futc 1  and \futc {2} are axiomatized by
respectively \utc 1\ plus \axtcu\ref{fregesmurf} and \utc 2\ plus \axtcu\ref{fregesmurf}.

Clearly,  have the following interpretations of
$\utc i$ in $\tc i+ \exists x\, \ato(x)$.
Both interpretations are based on the translation $\uptau$.
\begin{itemize}
\item
$\updelta_\uptau^{\mf o}(x) :\iff \ato(x)$
\item
$\updelta_\uptau^{\mf s}(x) : \iff x=x$
\item
$\oslash_\uptau := \estr$
\item
$[x]_\uptau := x$
\item
$x\star_\uptau y := x \ast y$
\item
$x=_\uptau^{\mf{oo}} y :\iff x=y$
\item
$x=_\uptau^{\mf{ss}}y :\iff x=y$
\end{itemize} 

\noindent
The Frege axiom delivers sequentiality.
We have:

\begin{theorem}\label{koperensmurf}
$\futc 1$ $\mf o$-directly interprets ${\sf FAC}^{\sf f+}$.
\end{theorem}

\begin{proof}
Let our interpretation be based on $\upsigma$. We take:
\begin{itemize}
\item
$\updelta_\upsigma^{\mf o}(x) :\iff x=x$
\item
$\updelta_\upsigma^{\mf c}(\alpha) :\iff \alpha=\alpha$
\item
$\emptyset_{\upsigma} := \oslash$
\item
$\verz{x}_\upsigma := [x]$
\item
$\alpha\cup_\upsigma\beta := \alpha\star \beta$
\item
$\digamma_\upsigma(\alpha) = \digamma(\alpha)$
\item
$x=_\upsigma^{\mf{oo}}y :\iff x=^{\mf{oo}}y$
\item
$\alpha=_\upsigma^{\mf{cc}}\beta :\iff \alpha=^{\mf{ss}}\beta$
\item
$x \in_\upsigma \alpha :\iff [x]\preceq \alpha$
\end{itemize}
We easily verify that our interpretation is correct. We  note that  Theorem~\ref{verenigingssmurf} gives us union for
sets of atoms. This is, of course, inherited by the range of $[\cdot]$, even if this range needs not be all atoms.
\end{proof}

We end this section by giving an interpretation $\mf C$ of \utc{1} in $\tc 1({\tt b})$ that will be relevant later.
We define the translation $\upgamma$ as follows.

\begin{itemize}
\item
$\updelta^{\mf o}_{\upgamma}(x) :\iff x = \estr \vee  {\tt b}\not \preceq x$
\item
$\updelta^{\mf s}_{\upgamma}(x) :\iff x = \estr \vee  {\tt b} \preceq_{\sf i} x$
\item
$\emptyset_{\upgamma} := \estr$
\item
$[x]_{\upgamma} := {\tt b}\ast x$ 
\item
$x \star_\upgamma y := x\ast y$
\item
$x=_\upgamma^{\mf{oo}} y :\iff x=y$
\item
$x=_\upgamma^{\mf{ss}} y :\iff x=y$
\end{itemize}

\begin{theorem}[$\tc 1({\tt b})$]\label{aktievesmurf}
The translation $\upgamma$ supports an interpretation $\mf C$ of \utc{1}.
\end{theorem}

\begin{proof}
The axioms \axtcu1,2,3 are trivial. 

We check \axtcu 4, to wit that $[x]_{\upgamma}$ is an atom. Suppose $x$ is is {\tt b}-free and 
${\tt b}\ast x = {\tt b} \ast y \ast {\tt b} \ast z$. Then, by Theorem~\ref{smeuigesmurf}, we have
$x= y \ast {\tt b}\ast z$, contradicting that $x$ is {\tt b}-free.

The axiom \axtcu\ref{injectiesmurf} is immediate from Theorem~\ref{smeuigesmurf}.

Finally we check \axtcu\ref{ureditors}. We can do this using partitions, but, in this case, a direct verification is
easy. There are a number of cases. We treat the only interesting one.
Suppose ${\tt b}\ast x\ast {\tt b}\ast y = {\tt b}\ast u\ast {\tt b}\ast v$. We apply the Editors Principle to
find a $w$ such that (${\tt b}\ast x\ast w = {\tt b}\ast u$ and ${\tt b}\ast y = w\ast {\tt b}\ast v$) or
(${\tt b}\ast x = {\tt b}\ast u \ast w$ and $w \ast {\tt b}\ast y =  {\tt b}\ast v$). We have to show that
$w$ is either $\estr$ or $w$ is of the form ${\tt b}\ast z$. 
We treat the first disjunct, the other being similar.  We apply the Editors Principle to 
${\tt b}\ast y = w\ast ({\tt b}\ast v)$. We find that, for some $z$, either ${\tt b}\ast z = w$ or $ {\tt b}= w\ast z$.
In the first case we are done. In the second case, we have $w=\estr$ or $z=\estr$. In both of these last cases we are immediately done.  
\end{proof}

We note that the range of $[x]$ in our interpretation contains all atoms of the interpreted theory. In other words, we interpret the principle:
$\alpha$ is an atom iff, for some $x$, $\alpha = [x]$. 

\section{String Theory as Intermediary}\label{mediatorsmurf}
This section is a small introduction to Sections~\ref{smulsmurf} and \ref{marcoroni}.
The idea of this paper is that we use string theory as an intermediate stage in the construction of container strings from numbers. But what does this precisely mean?
Our aim is to provide an $\mf o$-direct interpretation $N$ of our sequential theory \futc1\ in our basic arithmetical theory $\mc A$, thus showing that
$\mc A$ is sequential. We want to
do this using a string theory $\mc S$ as intermediate stage. Naively one would think that it can proceed in two stages:
we have an interpretation $K$ of $\mc S$ in $\mc A$ and an interpretation $M$ of \futc1\ in $\mc S$ and we can take
$N := K\circ M$. 
In the Smullyan case, things do indeed work out like that, but in the Markov case, we proceed slightly differently. 

Let us first describe the Smullyan case.
We provide an $\mf o$-direct  interpretation $\mf B$  in \pamsmu\ of  \tclac1, a string theory explained in Section~\ref{langesmurf}.
Then, we construct an $\mf o$-direct interpretation $\mf A$ of \futc1\ in  \tclac1. The composition $\mf B \circ \mf A$ will be our
desired $\mf o$-direct interpretation of the sequential theory \futc1\ in \pamsmu.

We outline the Markov case. We start with the theory \pam. We give an interpretation $\mf H$ of $\tc1(\tt b)$\ based on translation $\upeta$ in \pam.
(In fact, we specify $\mf H$ as an interpretation of a stronger theory.)  
We already have given an interpretation $\mf C$ of \utc1\ in  $\tc1(\tt b)$ based on translation $\upgamma$ in Section~\ref{neandersmurf}.
The composition $\upeta \circ \upgamma$ provides the kernel of our desired interpretation $\mf U$ that will be based on a translation $\uptheta$.
The object domain of $\uptheta$ will be the $\pamo$-numbers. 
The restriction of $\uptheta$ to the string domain, the constants and concatenation will coincide with $\upeta \circ \upgamma$.
We will provide two extra data. A \pamo-definable injection $G$ of the numbers into the object domain of $\upeta \circ \upgamma$ and a
\pamo-definable injection $F$ from the string domain of $\upeta \circ \upgamma$ to the numbers. 
We then define:
\begin{itemize}
\item
$[x]_\uptheta =y :\iff  \exists z \, (Gx=z \wedge ([z]=y)_{\upeta\circ \upgamma})$,
\item
$\digamma_\uptheta(u)=v :\iff Fu=v$.
\end{itemize} 

So the idea here is that the string theory provides a container string theory without a Frege function.
We expand with the Frege function and allow the $[\cdot]$-images of the numbers to be
a part of the object domain provided by the string theory. (The second example of  Section~\ref{alitersmurf} shows  that
the images of the numbers may be a strict part of the object domain of the string theory.)

\section{Smullyan Coding}\label{smulsmurf}
In this section we give an $\mf o$-direct interpretation of \futc 1\ in \pamsmu. We do this in three steps.
We first develop the extension \tclac1\ of \tcc1\ with a length-function in Section~\ref{langesmurf}. We show how we
can $\mf o$-directly interpret \futc 1\ in \tclac1\ in Section~\ref{naledismurf}. Finally, in Section~\ref{graaftelsmurf}, we show that we
can directly interpret \tclac1\ in \pamsmu. Section~\ref{lerendesmurf} gives the basic heuristics.

\subsection{Heuristics}\label{lerendesmurf}
The main aim of this paper is to develop container strings in weak arithmetical theories and see what is involved in such constructions.
We zoom in on two constructions that use string theory as auxiliary.

\subsubsection{Strings to Container Strings}\label{stus}
There 
are many ways to define container strings of numbers using strings. A first question is: how will the numbers of our object domain be
represented in the container strings we aim to construct?
Does it suffice to provide container strings for the tally numbers that correspond 1-1 to the numbers of the ambient
arithmetical theory? In the case of the Markov coding, the answer here will be a resounding \emph{yes}. 
We can represent the container string $n_0\dots n_{k-1}$ by ${\tt b}{\tt a}^{n_0}\dots {\tt b}{\tt a}^{n_{k-1}}$.
This is the simplest possible coding. However, in the Smullyan case, there is a problem with this strategy.
The map from $n$ to \emph{the number coding ${\tt a}^n$} is exponential. Thus, in the Smullyan coding,
we have to proceed differently in order not to violate our taboo against exponentiation. 
A first step towards our aim is provided by the insight  that the Smullyan coding associates all binary strings 1-1 with numbers.
Thus, it suffices to build, in the syntax theory, container strings of  strings. 

We run into the \emph{comma problem} here.
Since the comma is also a string, how do we distinguish uses-as-comma from uses-as-string?
There are various solutions to the comma problem. Here are three of them.

Willard Van Orman Quine in \cite{quin:conc46} simply stipulates that the commas are strings ${\tt b}{\tt a}^n{\tt b}$ where
$n$ is chosen in such a way that all substrings consisting of {\tt a}'s in the component strings representing numbers are
shorter in length than $n$. This is a fine solution. It is followed, in the context of weak theories, by Raymond Smullyan is his
classical book \cite{smul:theo61}.
Note that in concatenating,
we have to choose the  commas so that they are the same in both strings. 

A second strategy is to
have growing commas. Here the commas are allowed to `grow' in the representation of the
 container string. Here we do not have to update already chosen commas. 
 See \cite{viss:howt86}, \cite{viss:grow09}, and \cite{damn:mutu22} for implementations of this strategy. 
 
 A third strategy is to develop strings for an alphabet of four letters and then represent 
 the components of our container string as strings of two letters, say {\tt a} and {\tt b}, separated by a third, say {\tt c}.
 Then we translate the {\tt a},{\tt b}-strings of the four letter alphabet back to   {\tt a},{\tt b}-strings
 of the two letter alphabet and use that the coding of the last is bijective. This strategy is 
 followed by Edward Nelson in \cite{nels:pred86}.
 
We turn to our strategy in this paper. This strategy  is, for example, used in \cite[Chapter V, 3]{haje:meta91}. We code container strings as pairs
 \[({\tt b}\Lambda(w_0) \dots {\tt b}\Lambda(w_{k-1}), {\tt b}w_0 \dots {\tt b}w_{k-1}),\] where $\Lambda(w)$ is the tally-function that
 replaces all letters in $w$ by {\tt a}'s. We use the first string as a measure stick to find the appropriate places where the $w_i$ are
 stored in the second string. The bad news for this strategy is that we have to define $\Lambda$. Of course, we can do this by recursion (see 
 \cite{damn:mutu22}), but then we run in a circle because to implement the relevant recursion we need sets, container strings or sequences.
 The good news is that the Smullyan coding of string theory in arithmetic allows us to define $\Lambda$ without extra effort.
 Thus, instead of the honest toil of defining $\Lambda$ in the string world, we get away with theft: we simply add it as primitive, smuggling it in
 from the side of arithmetic. 
 
 \subsubsection{Numbers to Strings}
 We consider strings of  {\tt  a}'s and {\tt b}'s with concatenation.
The development works equally well for any alphabet of prime cardinality.
 
We order our strings length-first. 
 \[
\begin{tabular}{|r|r||r|r||r|r||r|r|} \hline
0 & $\estr$ & 5 & {\tt ba}  & 10 & {\tt abb} & 15 & {\tt aaaa}\\ \hline
1 & {\tt a}     & 6 & {\tt bb} & 11 & {\tt baa} & 16 & {\tt aaab}  \\ \hline
2 & {\tt b}      & 7 & {\tt aaa} & 12 & {\tt bab} & 17 & {\tt aaba} \\ \hline
3 & {\tt aa}    & 8 & {\tt aab} & 13 & {\tt bba} & 18 & {\tt aabb} \\ \hline
4 & {\tt ab}     & 9  & {\tt aba} & 14 & {\tt bbb} & 19 & {\tt abaa}\\ \hline
\end{tabular}
\]

This way of coding is called the \emph{length-first} or the  \emph{dyadic} or the \emph{bijective base-2} numeration.
We can define the association of strings to numbers recursively by:
\begin{itemize}
\item $\estr^{\sf sm} := 0$
\item $(\sigma \ast {\tt a})^{\sf sm} := 2 \cdot \sigma^{\sf sm}+1$
\item $(\sigma \ast {\tt b})^{\sf sm} := 2 \cdot \sigma^{\sf sm}+2$
\end{itemize}

We define:
\begin{itemize}
\item
 $\ell(n) :=  \text{the largest power of $2$ smaller or equal to $n+1$}$.
\item
$m\circledast n:=m\cdot \ell(n)+n$.
\end{itemize}

\noindent
Then, $\ell(\sigma^{\sf sm})=2^{{\sf length}( \sigma)}$ and
$\sigma^{\sf sm}\circledast \tau^{\sf sm}:= (\sigma\ast \tau)^{\sf sm}$.
It is easy to see that
the growth rate of $\circledast$ is the same as the growth rate of multiplication.
 
We can code the development of strings using the Smullyan coding since we can
define $n$ being a power of 2 in the Smullyan way as $\forall m \, (m\mid n \to (m=1 \vee 2\mid m))$.
As a consequence, we can define $\ell$ without defining exponentiation.

 $\Lambda$ maps a string to a string of {\tt a}'s of the same length.
 We have ${\sf sm}( \Lambda({\sf sm}^{-1} (n))) = \ell(n) -1$.
 So, it easy to define the arithmetical shadow (or \emph{tracking function}) of $\Lambda$.
 
 \begin{ques}
 {\small
 In \cite{viss:grow09}, we work out the strategy of obtaining container strings with growing commas.
 This employs the theory $\tc1[{\tt a},{\tt b}]$ plus an extra axiom which says that for any string
 $x$ there is an {\tt a}-string $y$ that contains all {\tt a}-substrings of $x$, but is not contained
 in any such substring. Is this theory  also
 directly interpretable in \pamsmu?
 } 
 \end{ques}
 
 We can view our two-letter strings as dyadic number notations: we replace the {\tt a} by 1 and the {\tt b} by 2.
 I did not work this out, but it seems clear that we can also do our development of Smullyan coding using
 ordinary binary notations instead. We code a binary notation as a pair $(\ell,x)$ where $\ell$ is a power of 2 with $x< \ell$,
 The $\ell$ is motivated as being $2^{{\sf length}(x)}$.
 E.g. $0101$ is $(16,5)$.

\subsection{Adding the Length Function}\label{langesmurf}
We extend the language of theories of concatenation with the constant {\tt b} and with
a unary function symbol $\Lambda$.
The intended meaning of $\Lambda$ is that $\Lambda(x)$ is the result of replacing all letters in
$\Lambda$ by {\tt a}'s. In our axiomatic treatment, we just use that $\Lambda(x)$ is {\tt b}-free.

We give a list of possible principles.

\begin{enumerate}[\axtcl1.]
\item
$\vdash \Lambda(\estr)= \estr$
\item
$\vdash \Lambda(x) = \estr \to x=\estr$
\item
$\vdash \Lambda(x\ast y) = \Lambda(x)\ast \Lambda(y)$ \label{spookjessmurf}
\item
$\vdash \Lambda(\Lambda(x)) = \Lambda(x)$
\item
$\vdash {\tt b} \not\preceq  \Lambda(x)$
\end{enumerate}

There are many more possible principles for $\Lambda$, but is seems we can make do with just these.

The theory \tcla i\ is axiomatised by  \tc i({\tt b})\ 
plus \axtcl1-5.
Our conventions for adding Left Cancelation and adding atoms are the same as before. \tcla i\ will automatically
contain {\tt b}, but if we add further atoms, we will also include {\tt b} in the list. E.g., we write
$\tcla i({\tt a},{\tt b})$ and $\tcla i[{\tt a},{\tt b}]$.  

If we have Weak Left Cancellation \axtc\ref{wcancellsmurf}, the principle \axtcl1 can be derived over \tc 0.\footnote{This observation is due to
Emil \jer.}

\begin{theorem}[\tc0 + \axtc\ref{wcancellsmurf} + \axtcl\ref{spookjessmurf}]
We have $\Lambda(\estr)=\estr$.
\end{theorem}

\begin{proof}
For any $x$, we have $\Lambda(x) =\Lambda(x\estr)= \Lambda(x)\Lambda(\estr)$. So, $\Lambda(\estr)= \estr$.
\end{proof}

In \tclac1\ we can strengthen the Left Cancellation Principle.

\begin{theorem}[\tclac1]\label{lachebeksmurf}
Suppose $xy=uv$ and $\Lambda(x)= \Lambda(u)$. Then,  $x=u$ and $y=v$.
\end{theorem}

\begin{proof}
Suppose $xy=uv$ and $\Lambda(x)=\Lambda(u)$. 
By the Editors Axiom, for some $z$, we have ($xz =u$ and $y=zv$) or ($uz=x$ and $v=zy$). 
In the first case, it follows that $\Lambda(x)\Lambda(z) = \Lambda(u) = \Lambda(x)$. So, $\Lambda(z)=\estr$,
and, hence, $z=\estr$. We may conclude that $x=u$ and  $y=v$. The second case is similar.
\end{proof}

In \cite{viss:grow09}, it is shown that pairing cannot be defined in $\tc1({\tt a},{\tt b})$. So,
more principles are needed. Here we define pairing in \tclac1.
\begin{itemize}
\item
$\tupel{x,y}= \Lambda(x){\tt b}xy$.
\end{itemize}

\begin{theorem}[\tclac1]
$\tupel{x,y} = \tupel{u,v}$ iff $x=u$ and $y=v$.
\end{theorem}

\begin{proof}
Suppose $ \Lambda(x){\tt b}xy =  \Lambda(u){\tt b}uv$.
 By Theorem~\ref{scherpesmurf} and the fact that $\Lambda(x)$ and $\Lambda(u)$ are {\tt b}-free,
 we find that $\Lambda(x) = \Lambda(u)$ and $xy=uv$.  
 It follows, by Theorem~\ref{lachebeksmurf}, that $x=u$ and $y=v$. 
 \end{proof}

\subsection{From Strings to Container Strings}\label{naledismurf}
We  provide an $\mf o$-direct interpretation of  \futc1\ in \tclac1.
We call the relevant translation $\upalpha$ and the interpretation it carries $\mf A$.
We employ the pairing function $\tupel{\cdot,\cdot}$ introduced in Subsection~\ref{langesmurf}. 
 \begin{itemize}
 \item
 We take $\updelta^{\mf o}_{\upalpha}(x)$ to be $x=x$.
 \item
 A \emph{pre-container-string} $x$ is either of the form  $\estr$ or ${\tt b}z$.
 \item
 An \emph{container string}  $\alpha$ is a pair $(x,y)$ where $x$ is a pre-container-string and $\Lambda(x)=\Lambda(y)$. 
 We take $\updelta^{\mf {s}}_{\upalpha}$ to be the class of container strings.
 \item
 $\oslash_{\upalpha} := (\estr,\estr)$.
 \item
 $\sing u_\upalpha := ({\tt b}\Lambda(u), {\tt b}u)$.
 \item
$(x,y) \star_\upalpha (u,v) := (xu,yv)$.
\item
$\digamma_\upalpha(x,y) := \tupel{x,y}$.
\item
$x=_\upalpha^{\mf{oo}}y :\iff x=y$
\item
$(x,y) =^{\mf{ss}}_\upalpha(u,v) :\iff x=u \wedge y=v$
\end{itemize}

\noindent
We will omit the subcript $\upalpha$, wherever possible.
The round brackets in the above definitions, indicate syntactic pairs and not pairing as a binary function in the theory.
Clearly, we could also have chosen the internal pairs here, making $\digamma$ the identity function.

It is easy to see that $\oslash$, $[u]$ and $\alpha \star \beta$ are indeed container strings.

\begin{rem}{\small
In a slightly different set-up, we could have avoided the use
of our string pairing employing the Cantor Pairing provided by \pamj. }
\end{rem}

We note that the non-emptiness of the domains and the validity of the
 identity axioms under $\upalpha$ is trivial. We have \axtcu1-3.

\begin{theorem}[\tcla0]
\begin{enumerate}[i.]
\item
$\oslash \star \alpha = \alpha \star \oslash = \alpha$.
\item
$\alpha\star\beta = \oslash \to (\alpha=\oslash \wedge \beta = \oslash)$.
\item
$(\alpha\star \beta) \star \gamma = \alpha\star (\beta\star \gamma)$.
\end{enumerate}
\end{theorem}

\noindent
The proof is easy. Next we verify \axtcu4.

\begin{theorem}[\tcla1]\label{snerendesmurf}\label{gekkesmurf}
$[x]$ is an atom of the $\upalpha$-interpreted theory.
\end{theorem}

\begin{proof}
Suppose $[x] =\oslash$. 
It follows that $ {\tt b} \Lambda(x) = \estr$. So, ${\tt b}=\estr$. \emph{Quod non.}

Let $\alpha =(a_0,a_1)$ and $\beta = (b_0,b_1)$ be container strings. Suppose that
$\alpha\star \beta = ({\tt b} \Lambda(x), {\tt b}x)$. If both $a_0$ and $b_0$ are not $\estr$, then
they are of the form ${\tt b}a_0'$, respectively ${\tt b}b_0'$. Consider any refinement of 
$({\tt b}, \Lambda(x))$ and $({\tt b}, a_0',{\tt b}, b_0')$. This would have to contain {\tt b} twice, which is impossible, since no
partition of $\Lambda(x)$ contains a {\tt b}.
It follows that one of $a_0$ and $b_0$ is $\estr$. 
If, e.g., $a_0=\estr$, it follows that $\Lambda(a_1) = \estr$, and, thus, that $a_1=\estr$. So,
$\alpha=\oslash$. Similarly, for the case that  $b_0=\estr$.
\end{proof}

\noindent
We note that not every atom needs to be of the form $[x]$. 
We verify \axtc5.

\begin{theorem}[\tcla1]
$[\cdot]$ is injective. 
\end{theorem}

\begin{proof}
Suppose  $[x] = [y]$. It follows that ${\tt b}x={\tt b}y$ and, by Theorem~\ref{hippesmurf}, that $x=y$.
\end{proof}

We turn to the verification of the Editors Axiom, \axtcu\ref{ureditors}.

\begin{theorem}[\tclac1]
Suppose $\alpha\star\beta = \gamma \star \delta$. Then there is an $\eta$ such that either
$\alpha\star \eta = \gamma$ and $\beta = \eta \star \delta$, or $\alpha = \gamma\star \eta$ and $\eta\star \beta = \delta$.
\end{theorem}

\begin{proof}
Let $\alpha = (a_0,a_1)$, $\beta = (b_0,b_1)$, $\gamma = (c_0,c_1)$ and $\delta =(d_0,d_1)$.
We have $a_0b_0 = c_0 d_0$. It follows that there is an $e_0$ such that either $a_0e_0= c_0$ and
$b_0 = e_0 d_0$, or $a_0=c_0 e_0$ and $e_0 b_0 = d_0$. We zoom in on the first case. The second
case is similar.

In case $e_0=\estr$, we have that $e_0$ is a pre-container-string.
Otherwise, $b_0\neq \estr$, and, thus, $b_0={\tt b} b_0'$. Consider a refinement of
$ ({\tt b}, b_0')$ and $(e_0, d_0)$. Since $e_0$ is non empty, the first non-empty element of
the refinement in $e_0$ must be {\tt b} by the atomicity of {\tt b}. So, again $e_0$ is a pre-container-string. 

We have an $e_1$ such that 
 either $a_1 e_1= c_1$ and
$b_1 = e_1 d_1$, or $a_1=c_1e_1$ and $e_1 b_1 = d_1$.
In the first case we have:
\begin{eqnarray*}
\Lambda(a_0)  \Lambda(e_0) & = & \Lambda(a_0 e_0) \\
& = &  \Lambda(c_0) \\
& = & \Lambda(c_1) \\
& = & \Lambda(a_1 e_1) \\
&=& \Lambda(a_1)  \Lambda(e_1) \\
& = &  \Lambda(a_0)  \Lambda(e_1) 
\end{eqnarray*}  
It follows by Left  Cancellation, that $\Lambda(e_0) = \Lambda(e_1)$.
So, $\eta := (e_0,e_1)$ is a container string and
$\alpha\star \eta = \gamma$ and $\beta = \eta \star \delta$.

In the second case we have $a_1 = c_1 e_1$.
So, 
\begin{eqnarray*}
\Lambda(a_1) & = &  \Lambda(c_1 e_1) \\
&=&    \Lambda(c_1)\Lambda(e_1) \\
&=& \Lambda(c_0) \Lambda(e_1) \\
& = &  \Lambda(a_0 e_0)  \Lambda(e_1) \\
& = &   \Lambda(a_0) \Lambda(e_0)  \Lambda(e_1) \\
& = &  \Lambda(a_1) \Lambda(e_0)  \Lambda(e_1) 
\end{eqnarray*}  
It follows, by Left Cancellation,  that $\Lambda(e_0) \Lambda(e_1) = \estr$. 
So,  $\Lambda(e_0) = \Lambda(e_1) =\estr$, and, hence, $e_0=e_1=\estr$.
We easily see that $\eta = \oslash$ is the desired witness.
\end{proof}

\begin{theorem}[\tclac1]
$\digamma$ is injective. 
\end{theorem}

\noindent This is immediate from the properties of the pairing function.
Thus, we have seen that $\upalpha$ indeed supports the promised $\mf o$-direct
interpretation $\mf A$ of \futc1\ in \tclac1.

\subsection{From Numbers to Strings}\label{graaftelsmurf}
In this subsection, we will explicitly use $\ast$ and $\cdot$ to avoid confusion with the usual omitting convention. 

We define a direct interpretation $\mf B$ of \tclac1 in \pamsmu.
Our translation will deliver a bit more: we interpret $\tclac1({\tt a},{\tt b})$.
Moreover, if we extend \pamsmu\ with the Standard Euclidean Division Principle, \axpam\ref{zringax}, or, more economically, with the
principle that every number is odd or even, we find that {\tt a} and {\tt b} are the only atoms and that we have the stack principle. Thus,
we obtain
an interpretation of  $\tclac2[{\tt a},{\tt b}]$.

We define our  translation $\upbeta$.
\begin{itemize}
\item
$\ell (x) = y$ iff ${\sf pow}_2(y)$ and $y \leq x+1 < \num 2\,y$.
\item
$\updelta_{\upbeta}(x) :\iff x=x$.
\item
$\estr_\upbeta := 0$.
\item
${\tt a}_\upbeta :=1$.
\item
${\tt b}_\upbeta:=\num 2$.
\item
$ x=_\upbeta y :\iff x=y$.
\item
$\Lambda_{\upbeta}(x) = \ell(x)-1$.
\item
$x \ast_\upbeta  y := x\cdot \ell(y)+y$.
\end{itemize}

\noindent
We will omit the subscript $\upbeta$ and the underlining of $2$. In the present context, we will both exhibit the $\cdot$ of multiplication and the
$\ast$ of concatenation to avoid confusion. 

\begin{theorem}[\pamsmu]\label{ieniemieniesmurf}
$\ell$ defines a function. It follows that $\Lambda$ and $\ast$ are functions.
\end{theorem}

\begin{proof}
The existence part is Axiom~\axpam\ref{smurfin}. We prove uniqueness.
Suppose $y$ and $y'$ are powers of 2 and  $y \leq x+1 < 2\cdot y$ and $y' \leq x+1 < 2\cdot y'$.
Without loss of generality, we may assume $y \leq y'$. It follows that $y'= y \cdot z$, by Axiom \axpam\ref{babysmurf}.
But, then, $z$ must be 1. 
\end{proof}

\begin{theorem}[\pamsmu]\label{spoorsmurf}
$\ell(x\ast y) = \ell(x) \cdot \ell (y)$.
\end{theorem}

\begin{proof}
Since $\ell(x)$ and $\ell(y)$ are powers of two, so is $\ell(x) \cdot\ell(y)$, by Theorem~\ref{delicatessesmurf}.
Thus, it suffices to show that
$\ell(x)\cdot \ell(y) \leq x\cdot \ell(y) + y +1 < 2\cdot \ell(x) \cdot \ell(y)$.
We have:\qedright
\begin{eqnarray*}
\ell(x) \cdot \ell(y) & \leq & (x+1)\cdot \ell (y) \\
& = & x \cdot \ell(y)+ \ell(y) \\
& \leq & x\cdot \ell(y) + y+1 \\
& < & x\cdot \ell(y)+ \num 2\cdot \ell(y) \\
& = & (x+ 2) \cdot \ell(y) \\
& \leq &  2\cdot \ell(x) \cdot \ell(y)
\end{eqnarray*}
\end{proof}

\begin{theorem}[\pamsmu]\label{paarsesmurf}
Suppose $x \leq y$, then $\ell(x) \leq \ell(y)$, and, hence, $\ell(x) \mid \ell(y)$. 
\end{theorem}

\begin{proof}
Suppose $x \leq y$ and $\ell(y) \leq \ell(x)$.
Then, \[\ell(x) \leq x+1 \leq y+1 <  2\cdot \ell(y) \leq  2\cdot\ell(x).\] So, $\ell(x)=\ell(y)$, by
Theorem~\ref{ieniemieniesmurf}. 
\end{proof} 

We verify \axtc1-4.

\begin{theorem}[\pamsmu]
\begin{enumerate}[i.]
\item
$x\ast \estr = \estr \ast x = x$.
\item
$x\ast y = \estr \to (x=\estr \wedge y = \estr)$.
\item 
$(x \ast y) \ast z = x\ast (y\ast z)$.
\item
$\ato({\tt a})$ and $\ato({\tt b})$ and ${\tt a}\neq {\tt b}$.
\end{enumerate}
\end{theorem}

\begin{proof}
Ad (i):
we have $x \cdot \ell(0) + 0 = x\cdot  1+0 =x$ and $0\cdot \ell(x) +x = x$.

Ad (ii) Suppose $x\cdot \ell(y) +y =0$, then $x$ and $y$ must be 0, since $\ell(y) \neq 0$.

Ad (iii): We have:
\begin{eqnarray*}
(x \ast y) \ast z & = &  (x\cdot \ell(y)+y) \cdot \ell(z) +z \\
& = & x\cdot \ell(y) \cdot \ell(z) + y\cdot \ell (z) +z \\
& = & x\cdot \ell(y\ast z) + y\ast z\\
& = & x \ast (y\ast z)
\end{eqnarray*}

Ad (iv): Suppose $x\ast y = {\tt a}$. Then, $x\cdot \ell(y)+y =1$. So, either $y=0$ or $y=1$. In the first case, it follows that $\ell(y)=1$ and so $y=0=\estr$ and $x=1={\tt a}$.
In the second case, we have $\ell(y)= 2$ and, so, $y=1={\tt a}$ and $x=0=\estr$.

Suppose  $x\ast y = {\tt b}$. Then, $x\cdot \ell(y)+y = 2$. So, either $y=0$ or $y=1$ or $y=2$. In the first case, it follows that $\ell(y)=1$ and, so,  $y=0=\estr$ and $x=2={\tt b}$.
In the second case, we have $\ell(y)= 2$ and, so, $x\cdot 2+1 =2$. This is impossible. In the third case, we have  
 $\ell(y)= 2$. It follows that $y= 2= {\tt b}$ and $x=0=\estr$. 
 
 Finally, ${\tt a}\neq {\tt b}$ is immediate.
\end{proof}

We verify Left Cancelation \axtc\ref{cancellsmurf}, and Right Cancellation.

\begin{theorem}[\pamsmu]\label{blauwesmurf}
\begin{enumerate}[i.]
\item
If $x\ast y = x \ast z$, then $y=z$. 
\item
If $x \ast z= y\ast z$, then $x=y$.
\end{enumerate}
\end{theorem}

\begin{proof}
Ad (i): Suppose $x\ast y = x\ast z$.
It follows that $\ell(x)\cdot\ell(y) = \ell(x) \cdot\ell(z)$. So, $\ell(y)=\ell(z)$ and, thus, $x\cdot\ell(y) +y = x\cdot \ell(y)+z$.
\emph{Ergo}, $y=z$.

Ad (ii): This is immediate.
\end{proof}

We proceed to the verification of the Editors Axiom, \axtc\ref{editors}.

\begin{theorem}[\pamsmu]\label{editorsmurf}
Suppose $x \ast y = u \ast v$. Then there is a $w$ such that \textup($x \ast w = u$ and 
$y = w\ast v$\textup) or \textup($ x = u\ast w$ and $ w\ast y =v $\textup).
\end{theorem}

\begin{proof}
Suppose that  $x \ast y = u \ast v$, in other words, $x\cdot \ell(y) + y = u \cdot \ell(v) +v$.
Without loss of generality, we may assume $y \leq v$.
We have $x\cdot \ell(y) = u\cdot \ell(v) + (v-y)$. By Theorem~\ref{paarsesmurf}, we find that $\ell(v)$ is divisible by $\ell(y)$.
So, by Theorem~\ref{lilasmurf}, we obtain that $(v-y)$ is divisible by $\ell(y)$, say $v-y= s\cdot \ell(y)$.
We may conclude that $v= s\cdot \ell(y) +y$, i.e., $v=s\ast y$. So,
$x\ast y = u\ast (s\ast y) = (u\ast s) \ast y$. By Right Cancellation, we have $x= u\ast s$.  
\end{proof}

Thus, we have verified that we interpret \tcc1. 
We turn to the interpretation of the $\Lambda$-axioms.

\begin{theorem}[\pamsmu]
We have $\Lambda(\estr)=\estr$ and, whenever $\Lambda(x)=\estr$, then $x=\estr$.
\end{theorem}

We leave the simple proof to the reader.

\begin{theorem}[\pamsmu]
$\ell(\Lambda(x)) = \ell(x)$ and, hence $\Lambda(\Lambda(x))=\Lambda(x)$.
\end{theorem}

\begin{proof}
We have $\ell(m) =  \Lambda(m)+1 < 2 \cdot\ell(m)$.
\end{proof}

\begin{theorem}
$\Lambda(n\ast m) = \Lambda(n) \ast \Lambda(m)$.
\end{theorem}

\begin{proof}
We have: \qedright
\begin{eqnarray*}
\Lambda(n \ast m) & = & \ell(n\ast m)-1 \\
& = & \ell(n)\cdot\ell(m)-1 \\
& = & (\ell(n)-1) \cdot \ell(m) + \ell(m) -1\\ 
& = & \Lambda(n) \cdot \ell(\Lambda(m)) + \Lambda(m)\\ 
& = & \Lambda(n) \ast \Lambda(m)
\end{eqnarray*}
\end{proof}

\begin{theorem}[\pamsmu]
 ${\tt b} \not\preceq \Lambda(x)$.
\end{theorem}

\begin{proof}
Suppose $y\ast {\tt b} \ast z = \Lambda(x)$. Then,
$ (y\cdot  2+  2) \cdot \ell(z)+ z + 1= \ell(x)$.
It follows that $\ell(z) \leq \ell(x)$ and, thus, by \axpam\ref{babysmurf}, that $\ell(z) \mid \ell(x)$.
So, by Theorem~\ref{lilasmurf}, we find $\ell(z) \mid z+1$.
Since $\ell(z) \leq z+1 < 2\cdot \ell(z)$, we may conclude that $\ell(z)= z+1$.
It follows that $(y\cdot 2+ 3) \cdot \ell(z)= \ell(x)$. \emph{Quod non.}
\end{proof}

This finishes our interpretation of \tclac1 in \pamsmu. 
Note that we did a bit more. We interpreted $\tclac1({\tt a},{\tt b})$.

We also have the following simple insights which are not used in our main development.

\begin{theorem}[\pamsmu]
\begin{enumerate}[i.]
\item
$\Lambda({\tt a}) = {\tt a}$, $\Lambda({\tt b}) = {\tt a}$,
\item
$\Lambda(n\ast m) = \Lambda(m \ast n)$.
\end{enumerate}
\end{theorem}

\begin{proof}
We treat (ii):
 $\Lambda(n\ast m) =   \ell(n)\cdot\ell(m)-1 = \ell(m)\cdot\ell(n)-1 = \Lambda(m\ast n) $.
\end{proof}

If we add Euclidean Division to \pamsmu, we can do a bit more.
We can interpret the Stack Principle \axtc\ref{thorsmurf} plus the fact that {\tt a} and {\tt b} are the only atoms.

\begin{theorem}[\pamsmu + \axpam\ref{zringax}(2)]
Every $x$ is of one of the forms $\estr$ or $y\ast {\tt b}$ or $y\ast {\tt a}$.
\end{theorem}

\begin{proof}
 Consider any $x$. We have that $x$ is even or odd. If $x$ is even, it is either $0$, i.e., $\estr$, or of the form $ 2 \cdot y + 2$, i.e., $y\ast {\tt b}$.
 If $x$ is odd, it is of the form $ 2\cdot y +1$, i.e., $y\ast {\tt a}$.
\end{proof}

We note that it follows that {\tt a} and {\tt b} are the only atoms. Thus, we get $\tclac2[{\tt a},{\tt b}]$.
We also have the following result.

\begin{theorem}[\pamsmu + \axpam\ref{plimp16}]\label{projectionsmurf}
Suppose $\Lambda(x\ast y) = \Lambda z$. Then, there are $u$ and $v$ such that $\Lambda(u)=\Lambda(x)$,
$\Lambda(v)= \Lambda(y)$ and $u\ast v=z$.
\end{theorem}

\begin{proof}
Suppose $\Lambda(x\ast y) = \Lambda z$.
We find that $\ell(z) = \ell(x)\cdot \ell(y)$. So, $\ell(y) \leq \ell(z)$ and, thus, $\Lambda(y) \leq \Lambda(z)$. 
We note that, for any $a$, we have $\Lambda(a) \leq a$. 
By Euclidean Division, there are (unique) $w$ and $t$ such that
$z -\Lambda(z) = w \cdot \ell(y) + t$ and $t< \ell(y)$.
We take $u := w+ \Lambda(x)$ and $ v := t+\Lambda(y)$.

We note that $\Lambda (z) =  \Lambda(x\ast y) = \Lambda(x)\cdot\ell(y)+ \Lambda(y)$.
So, \[z = (w+ \Lambda(x))\cdot\ell(y)+ (t+ \Lambda(y)).\]

\noindent We have $\ell(y) \leq t+\Lambda(y)+1   <  2\cdot \ell(y) $. So, $\ell(y) = \ell(t+\Lambda(y)) = \ell(v)$.
We may conclude that $z= u\cdot \ell (v) + v = u \ast v$.

Finally, we have $\Lambda(x)\ast \Lambda(y) = \Lambda(z) = \Lambda(u)\ast \Lambda(v) = \Lambda(u)\ast \Lambda(y)$.
So, by Right Cancellation, $\Lambda(x)=\Lambda(u)$.
\end{proof}

We note that the last two results justify introducing a projection function that given a string $x$ and a tally number $\Lambda(y)\preceq_{\sf i} \Lambda(x)$
produces the letter on the `$\Lambda(y)$-th place' in $x$.

\begin{rem}{\small
Reflection of the proof of Theorem~\ref{projectionsmurf} shows that we are really switching back and forth
between dyadic numbers and binary notations.}
\end{rem}

\begin{ques}
\small{We left many question open.
Over what base theory extending \pam\ can we interpret the following principles via $\upbeta$?
\begin{enumerate}[a.]
\item
If $x$ is {\tt b}-free, then $\Lambda(x)=x$.
\item
If $x$ and $y$ have the same length and the same projections, then they are equal.
\item
$x$ is empty or $x$ is {\tt b}-free or of the form $u\ast {\tt b} \ast v$, where $v$ is {\tt b}-free.
\end{enumerate}
}
\end{ques}

\section{Markov Coding}\label{marcoroni}
In this section, we study the creation of strings and container strings using Markov coding in the context
of weak theories. We prove that the sequential theory \futc1 can be interpreted in \pamo.

The idea of Markov coding is to code strings as $2\times 2$ matrices with determinant 1.
These form, in the standard case, the structure ${\sf SL}_2(\mathbb N)$.
The canonical place where the insight that ${\sf SL}_2(\mathbb N)$ is isomorphic to
the free monoid on two generators is used in a metamathematical context is Andrej Markov's book~\cite{mark:theo54}.
However, the basic insight goes back to Jakob Nielsen. See \cite{niel:isom18}.  The idea
to use the Special Linear Monoid over the non-negative part of a suitable ring to study weak theories
of concatenation is due to Juvenal Murwanashyaka. See \cite{murw:weak22},  \cite{murw:pape23}, and \cite{murw:weak24}. 
In his paper \cite{murw:weak24}, Murwanashyaka provides a Markov style interpretation of $\tc1$ in {\sf Q}. 
He uses a (well-known) shortening argument to interpret \pamo\ in {\sf Q}. To obtain the Editors Principle, he uses a (new) shortening argument
that reflects a proof by induction of the  Editors Principle. The main difference of what Murwanashyka is doing in his paper
and our present work is that we are aiming at providing a \emph{direct} interpretation. So it is better to avoid
shortening arguments.

\subsection{Heuristic Remarks}\label{snelleersmurf}
The starting point of our approach is the observation that the special linear monoid ${\sf SL}_2(\mathbb N)$ is
isomorphic to the monoid of strings over the alphabet {\tt a}, {\tt b} with concatenation.  The elements of 
${\sf SL}_2(\mathbb N)$ are $2\times 2$-matrices of non-negative integers with determinant 1.
The matrix ${\tt A} := \sspl 1101$ will play the
role of {\tt a} and the matrix ${\tt B} := \sspl 1011$ will play the role of {\tt b}.\footnote{We could have used
{\tt a} and {\tt b} simply as names for these matrices. However, typographically, the capitals look better.}
The crucial observation is that ${\tt A}^n = \sspl 1n01$. In other words: tally numerals grow linearly.

 The monoid ${\sf SL}_2(\mathbb N)$ is part of the special linear group  ${\sf SL}_2(\mathbb Z)$ that has many remarkable
 properties. See e.g. \cite{conra:SL2Z12}.

One striking feature of ${\sf SL}_2(\mathbb N)$ is its  connection with the Fibonacci numbers.
Let ${\sf F}_{-1} := 1$, ${\sf F}_0:= 0$ and ${\sf F}_{n+1} := {\sf F}_{n-1}+{\sf F}_{n}$. We have:
\begin{itemize}
\item
$({\tt BA})^n = \sspl {{\sf F}_{2n-1}}{{\sf F}_{2n}}{{\sf F}_{2n}}{{\sf F}_{2n+1}}$.
\end{itemize} 

\noindent
We see that ${\sf F}_{2n-1}{\sf F}_{2n+1} - {\sf F}_{2n}^2 = 1$.  As is well-known the Fibonacci numbers have exponential growth.
The remarkable fact that ${\tt A}^n$ is linear and $({\tt BA})^n$ is exponential combined with
Yuri Matijasevic's  beautiful result that the function $\lambda n\mathbin{.}{\sf F}_{2n}$ is Diophantine is an important
part of the proof of the MRDP Theorem. 

We will code the container string $n_0\dots n_{k-1}$ as ${\tt B}{\tt A}^{n_0}\dots {\tt B}{\tt A}^{n_{k-1}}$.
Note that the string ${\tt A}^n$ is $\sspl1n01$,  the container string $0^n$ is $\sspl10n1$, and 
the container string $1^n =   \sspl {{\sf F}_{2n-1}}{{\sf F}_{2n}}{{\sf F}_{2n}}{{\sf F}_{2n+1}}$.
So, the container string of 1's has much faster growth than the container string of 0's.

\subsection{Basics of the Special Linear Monoid}
In this subsection, we present the basics of the Special Linear Monoid. We treat this over \pamo. 
It is perhaps possible to develop some of the material over \pamj, but, since all further results will be over \pamo, we wil not explore this.

Every model $\mc M$ of \pamo\ can be standardly extended by the usual pairs construction to
a discretely ordered commutative ring ${\sf z}(\mc M)$. 
Conversely, we can restrict  a discretely ordered commutative ring $\mc R$
to its non-negative part $\mc R^{\mathsmaller{\geq 0}}$ obtaining a model of \pamo.
Modulo definable isomorphism, these operations are inverses. On the level of theories, we get a bi-interpretation
between \pamo\ and the theory of discretely ordered commutative rings. This result can even be improved to a definitional
equivalence between these theories. These strong sameness results enable us to switch between
\pamo\ and the theory of discretely ordered commutative rings as a matter of course.
See Appendix~\ref{negatievesmurf} for details.

Let $\mc M$ be a model of $\pamo$.
We consider the $2\times 2$ matrices $\alpha$ over $\mc M$. We define  ${\sf det}(\alpha) \simeq e$ iff $bc+e=ad$.
It follows that {\sf det} is functional. Moreover, in the extension $\mathsf z(\mc M)$, the function {\sf det} extends to a total function.

\begin{lem}\label{minismurf}
The $2\times 2$ matrices over a model of \pamo\  with the usual matrix multiplication and  identity element  
form a monoid. Moreover, if  ${\sf det}(\alpha)$ and ${\sf det}(\beta)$ are both defined, 
we have ${\sf det}(\alpha\beta) = {\sf det}(\alpha){\sf det}(\beta)$.
\end{lem}

\begin{proof}
The matrices of $\mc M$  inherit the desired properties from the corresponding properties of the matrices over the commutative ring
$\mathsf z(\mc M)$. That a commutative ring has these properties is proved in most Algebra textbooks. See,
e.g., \cite[Chapter XIII]{lang:alge72}.
\end{proof}

By Lemma~\ref{minismurf}, we can meaningfully define  ${\sf SL}_2(\mc M)$, \emph{the special linear monoid} of the $2\times 2$-matrices with determinant 1.

\begin{lem}\label{microsmurf}
Let $\mc M$ be a model of \pamo.
Suppose that $\alpha$ is any $2\times 2$ matrix in $\mc M$  and that $\beta$ and $\gamma$ are in ${\sf SL}_2(\mc M)$.
Suppose further that $\alpha \beta = \gamma$. Then $\alpha$ is in the special linear monoid.
\end{lem}

\begin{proof}
We have $\alpha = \gamma\beta^{-1}$ in ${\sf SL}_2(\mathsf z(\mc M))$. So, ${\sf det}(\alpha)\simeq 1$ in $\mathsf z(\mc M)$.
Since all entries of $\alpha$ are in $\mc M$, we find that ${\sf det}(\alpha)\simeq 1$ in $\mc M$.
\end{proof}

\begin{lem}
Let $\mc M$ be a model of \pamo.
The special linear monoid over $\mc M$ is cancellative. 
\end{lem}

\begin{proof}
This is immediate since ${\sf SL}_2(\mathsf z(\mc M))$ is a group.
\end{proof}

\subsection{Transposition and Anti-transposition}\label{omkeersmurf}

We work in $\pamo$.
Consider $\alpha$ in the special linear monoid.
It is easy to see that the transpose of $\alpha$ is again in the special linear monoid. 
The transpose reverses the order of concatenation and interchanges {\tt A} and {\tt B}.
For example, we have the dual of \axtc\ref{thorsmurf}.

\begin{theorem}[\pamo]\label{sterkesmurf}
Every $\alpha$ in the special monoid is of one of the forms \emp\ or ${\tt A}\beta$ or ${\tt B}\beta$, where $\beta$ is in the special monoid.
\end{theorem}

We note that  taking the anti-transpose, i.e., the transpose w.r.t the anti-diagonal is also meaningful here: it interchanges the order but keeps the letters the same.
Doing both transpose and anti-transpose  keeps the order but switches the letters. 

The transpose and the anti-transpose are operations with quantifier-free definitions. So, we can give the corresponding operations of strings
quantifier-free translations.

\subsection{From Matrices to Strings}\label{matersmurf}
 We provide an interpretation of $\tcc{2}[{\tt a},{\tt b}]$ in \pamo.  This is more than we really need for our main development, but it is nice to have the fuller picture.

Our interpretation will be given by the translation $\upeta$. 
We define:
\begin{itemize}
\item
$\updelta_\upeta$ consists of all matrices with determinant 1, given as syntactic quadruples.
\item
$\estr_\upeta := \oslash = \sspl 1001$.
\item
${\tt a}_\upeta := {\tt A} := \sspl 1101$.
\item
${\tt b}_\upeta = {\tt B} := \sspl 1011$.
\item
$\alpha\ast_\upeta \beta = \alpha\beta$
\item
$\sspl abcd =_{\upeta} \sspl {a'}{b'}{c'}{d'} :\iff a=a' \wedge b=b' \wedge c=c' \wedge d=d'$.
\end{itemize}

We note that our translation is entirely quantifier-free. This tells us that interpretation it supports provides a functor from models to models
where the arrows are embeddings of models.
We have:

\begin{theorem}[\pamo]
$\upeta$ supports an interpretation $\mf H$ of  $\tcc{2}[{\tt a},{\tt b}]$.  
\end{theorem}

\begin{proof}
We trivially have the interpretation of  the axioms of identity,
 \axtc1, and \axtc3. The further verifications will be given in the lemmas below.
 \end{proof}
 
We first verify \axtc2\ plus the fact that {\tt A} and {\tt B} are atoms.
In combination with the obvious fact that {\tt A} and {\tt B} are different this gives us
$({\tt A},{\tt B})$.

\begin{lem}[\pamo]\label{freyasmurf}
Suppose $\alpha$ and $\beta$ are in the special linear monoid.
\begin{enumerate}[i.]
\item 
 If  $\alpha\beta = \oslash$, then $\alpha = \oslash$ and $\beta=\oslash$.
 \item
If  $\alpha\beta = {\tt A}$, then $\alpha = \oslash$ or $\beta=\oslash$.
\item
 If  $\alpha\beta = {\tt B}$, then $\alpha = \oslash$ or $\beta=\oslash$.
 \end{enumerate}
\end{lem}

\begin{proof}
We treat (ii). Case (i) is easier and case (ii) is similar. We have:
\[\alpha\beta =\sspl abcd \sspl efgh = \sspl{ae+bg}{af+bh}{ce+dg}{cf+dh}.\] 
Suppose both $\alpha$ and $\beta$ are not $\oslash$ and $\alpha\beta={\tt A}$. If $c \neq 0$ or $g\neq 0$,
we have $ce+dg>0$. \emph{Quod non}. So, we must have $b\neq 0$ and $f\neq 0$. It follows that
$af+bh>1$. \emph{Quod non.}
\end{proof}

We verify Left Cancellation \axtc{\ref{cancellsmurf}}.

\begin{lem}[\pamo]
Suppose $\alpha\beta = \alpha\gamma$. Then, $\beta=\gamma$.
\end{lem}

\begin{proof}
The existence of the bi-interpretation given by $\mathsf z$ and restriction to the non-negative part allows us
  to work with the axioms of a discretely ordered commutative ring and then return
to our original non-negative part. In the ring, $\alpha$ has an inverse and thus we get $\beta=\gamma$. 
\end{proof}

We write $\scol ac < \scol bd$ for: ($a < b$ and $c\leq d\,$) or  ($a \leq b$ and $c< d\,$).

\begin{lem}[\pamo]\label{keukensmurf}
Suppose $\alpha =\sspl abcd$ in the special linear monoid  Then, $\scol ac < \scol bd$ or $\scol bd < \scol ac$ or $\alpha = \emp$.
\end{lem}

\begin{proof}
Suppose $a< b$ and $d< c$. Then $ad < bc$, contradicting $bc+1 = ad$.

Suppose $b< a$ and $c< d$. Then $(b+1) \leq a$ and $(c+1)\leq d$.
So, $bc+b+c +1 \leq ad$. Since $bc+1 = ad$, it follows that
 $b+c=0$, and, hence, $b=c=0$. Hence, $ad=1$, and, thus, $a=d=1$.
We conclude that $\alpha = \emp$.
\end{proof}

\begin{lem}[\pamo]\label{ducksmurf}
Suppose $\alpha = \sspl abcd$ is in the special linear monoid. Then,
\begin{enumerate}[i.]
\item
$\scol ac < \scol bd$ iff, for some $\beta$ in the special linear monoid, we have $\alpha=\beta {\tt A}$.
\item
$\scol bd < \scol ac$ iff, for some $\beta$ in the special linear monoid, we have $\alpha=\beta {\tt B}$.
\end{enumerate}
\end{lem}

\begin{proof}
We treat (i). Item (ii) is similar.
The right-to-left direction is immediate using Lemma~\ref{keukensmurf} to get inequality.
We prove the right-to-left direction. Suppose   $\scol ac < \scol bd$. Then, moving to the extension where
we have the special linear group, we find that $\beta:= \alpha {\tt A}^{-1}$ is in the special linear monoid
and we are done. The other case is similar.
\end{proof}

We verify the fact that {\tt A} and {\tt B} are the only atoms and that we have the Stack Principle \axtc\ref{thorsmurf}.

\begin{lem}[\pamo]\label{substanulsmurf}
Every $\alpha$ in the special monoid is of one of the forms \emp\ or $\beta{\tt A}$ or $\beta{\tt B}$, where
$\beta$ is in the special monoid. 
\end{lem}

\begin{proof}
By Lemma~\ref{keukensmurf}, we have $\alpha = \emp$ or $\scol ac < \scol bd$ or $\scol bd < \scol ac$.
By Lemma~\ref{ducksmurf}, we are immediately done. 
\end{proof}

Finally, we verify the Editors Principle.
 
 \begin{theorem}[\pamo]
 The translation $\upeta$ supports the Editors Principle \axtc{\ref{editors}}.
 \end{theorem}
 
 \begin{proof}
 The desired result is immediate from Lemmas~\ref{voorsmurf} and \ref{diamantsmurf}, using the fact that, due to bi-interpretability, we can switch
 back and forth between the theory of a discretely ordered commutative ring and the theory of its non-negative part.
 \end{proof}
 
 \begin{lem}\label{voorsmurf}
 Let $\mc R$ be a discretely ordered commutative ring. Let $\alpha$, $\beta$, $\gamma$ and $\delta$ be
 in ${\sf SL}_2(\mc R^{\mathsmaller{\geq 0}})$, i.e., elements of  ${\sf SL}_2(\mc R)$ with non-negative entries.
 Suppose $\alpha\beta = \gamma\delta$. Then, we have the Editors property for $\alpha$, $\beta$, $\gamma$, $\delta$ iff
 the elements of the diagonal of $\gamma^{-1}\alpha$ are positive.
 \end{lem}
 
 \begin{proof}
 Let $\alpha$, $\beta$, $\gamma$, $\delta$ be as stipulated in the lemma.
 We first note that, since we have an interpretation of left cancellation, we only need to
 prove the lemma for  the Weak Editors Property given by \axtc{\ref{weakeditors}}. 
 
 Suppose $\alpha\mu=\gamma$ or $\alpha = \gamma\mu$, where $\mu$ is in
  ${\sf SL}_2(\mc R^{\mathsmaller{\geq 0}})$. Clearly, $\mu$ is either $\alpha^{-1}\gamma$ or $\gamma^{-1}\alpha$.
  We note that $(\alpha^{-1}\gamma)^{-1}= \gamma^{-1}\alpha$. Since the inverse interchanges the elements
  of the diagonal, it follows that the elements of the diagonal must be positive for both $\mu$ and $\mu^{-1}$.
  
 Conversely, suppose the elements of the diagonal of $\gamma^{-1}\alpha$ are positive. If one of the elements of the
 anti-diagonal were positive and the other negative, it would follow that the determinant of $\gamma^{-1}\alpha$
 is $>1$. \emph{Quod non}. So  both elements of the anti-diagonal are non-negative or both are non-positive.
 Hence  $\gamma^{-1}\alpha$ or its inverse $\alpha^{-1}\gamma$ is in ${\sf SL}_2(\mc R^{\mathsmaller{\geq 0}})$.
 \end{proof}
 
 In the next lemma, we do not use discreteness.
 
 \begin{lem}\label{diamantsmurf}
 Let $\mc R$ be an ordered commutative ring. Suppose $\alpha$, $\beta$, $\gamma$ and $\delta$ are
 in ${\sf SL}_2(\mc R^{\mathsmaller{\geq 0}})$ and $\alpha\beta = \gamma\delta$. Then, the elements on
 the diagonal of $\gamma^{-1}\alpha$ are positive.
 \end{lem}
 
 \begin{proof}
  Let $\alpha$, $\beta$, $\gamma$, $\delta$ be as stipulated in the lemma. Let 
   $\alpha = \sspl abcd$, $\beta = \sspl efgh$, $\gamma = \sspl ijk\ell$, $\delta = \sspl mnop$, $\mu := \gamma^{-1}\alpha := \sspl qrst$.
We note that
\[ \mu = \sspl \ell {-j} {-k} i \sspl abcd = \sspl{\ell a-jc}{\ell b -jd}{-ka+ic}{-kb+id}.\]

Suppose $q\leq 0$. Then,  $\ell a \leq jc$. It follows that
\[ar= a \ell b-ajd \leq jcb -jad = j(cb-ad) = -j.\] 
Since $a>0$ and $j\geq 0$, we may conclude that $r\leq 0$.
On the other hand, 
\[ \mu \beta = \sspl qrst \sspl efgh = \sspl{qe+rg}{qf+rh}{se+tg}{sf+th} = \sspl mnop = \delta.\]
Since $q\leq 0$, $e> 0$,  $g \geq 0$, $m>0$, we find $r>0$. A contradiction. So, $q>0$. 

A similar argument shows that $t>0$.
 \end{proof}

\begin{rem}\label{chatsmurf}{\small
I asked ChatGPT \textup(GPT-5, OpenAI, conversation on 9 October 2025\textup)  the following.
Consider  $\alpha$, $\beta$, $\gamma$ and $\delta$ in ${\sf SL}_2(\mathbb R)$ and $\alpha\beta=\gamma\delta$. Suppose all the entries
of these matrices are non-negative. Show that the elements on the diagonal of $\gamma^{-1}\alpha$ are positive.
In its answer, it provided  two key ideas of the proof of Lemma~\ref{diamantsmurf}.}
\end{rem}

\begin{rem}
{\small We note that Lemma~\ref{voorsmurf} uses discreteness, but Lemma~\ref{diamantsmurf} does not.
The following example shows that discreteness is essential for the Editors Property. We do not
have it in ${\sf SL}_2(\mathbb Q^{\mathsmaller{\geq 0}})$.

As usual ${\tt A} = \sspl 1101$. We have:
\[ \sspl{\frac 75}{\frac 15}{\frac 35}{\frac 45}{\tt A}  = {\tt A}\sspl {\frac 45} {\frac 15}{\frac 35}{\frac 75} \text{ and }\;
{\tt A}^{-1} \sspl{\frac 75}{\frac 15}{\frac 35}{\frac 45} = \sspl{\frac 45}{-\frac 35}{\frac 35}{\frac 45}.\]
Of course, the same example works for ${\sf SL}_2(\mathbb R^{\mathsmaller{\geq 0}})$.
}
\end{rem}

Our translation $\upeta$ is  entirely quantifier-free. This
means that it is absolute. A pleasant fact is that $\preceq_{\sf i}$ is also quantifier-free (modulo
provable equivalence).

\begin{theorem}[\pamo]\label{nozelsmurf}
$(a\preceq_{\sf i}b)^{\upeta}$ is equivalent to a quantifier-free formula.
\end{theorem}

\begin{proof}
We reason switching back and forth to $\mathsf z$.
Suppose, for some $\gamma$,  $\alpha\gamma = \beta$. Then $\gamma =  \alpha^{-1}\beta$ has non-negative entries.
Conversely, if $\alpha^{-1}\beta$ has non-negative entries it is a $\gamma$ such that $\alpha\gamma = \beta$.
We have \[\sspl d{-b}{-c}a \sspl efgh = \sspl{de-bg}{df-bh}{-ce+ag}{-cf+ah}.\] So, $\exists \gamma\; \alpha\gamma=\beta$ is
equivalent to: $de \geq bg$, $df\geq bh$, $ag\geq ce$, $ah \geq cf$.  
\end{proof}

We note that we did not try to find the maximal string theory interpretable via $\upeta$. For the example, by the results of Section
\ref{omkeersmurf}, we can read all principles in reverse order. For example, we also have right cancellation.
Moreover, we can add functions to our string theory that invert the ordering and functions that interchange {\tt A} and {\tt B} and interpret
the resulting theory in  quantifier-free way via an extension of $\upeta$.

\subsection{About ${\tt A}^n$}
We interpolate a subsection about the treatment of ${\tt A}^n$. 
We define:
\begin{itemize}
\item
${\tt A}^n := \sspl1n01$,
\item
$[n] := {\tt B}{\tt A}^n = \sspl 1n1{n+1}$.
\end{itemize}

We say that $\alpha$ is an {\tt A}-string if ${\tt B}\not\preceq\alpha$.

\begin{theorem}[\pamo]\label{freesmurf}
${\tt A}^n$ is {\tt B}-free, in other words, we cannot have ${\tt A}^n = \alpha {\tt B}\beta$, for any
$\alpha$ and $\beta$ in the special monoid.
\end{theorem}

\begin{proof}
We have:
\[ \sspl abcd \sspl 1011 \sspl efgh = \sspl \dots \dots {(c+d)e+dg}\dots.\]
 Moreover, $(c+d)e+ dg \geq de \geq 1$.
\end{proof}

\noindent
Does \pamo\ prove that every Markov {\tt A}-string is of the form ${\tt A}^n$? In Theorem~\ref{maatjassmurf}, we will show
that this is not the case. 

We have:
\begin{theorem}[\pamo] \label{unitweesmurf}
Let $\alpha$ and $\alpha'$ be in the special linear monoid.
Suppose $\alpha \sing n = \alpha'\sing {n'}$. Then $\alpha = \alpha'$ and $n=n'$.
\end{theorem}

\begin{proof}
We have $\alpha{\tt B}{\tt A}^n = \alpha'{\tt B}{\tt A}^{n'}$. Suppose $n<n'$.
Then, 
\[ \alpha{\tt B} = \alpha{\tt B}{\tt A}^{n}{\tt A}^{-n}  = \alpha'{\tt B}{\tt A}^{n'}{\tt A}^{-n} = \alpha'{\tt B}{\tt A}^{n'-n}.\] 
But this is impossible, since $ \alpha{\tt B}$ ends with {\tt B} and $\alpha'{\tt B}{\tt A}^{n'-n}$ ends with {\tt A}.
Similarly for $n'< n$. So $n= n'$, and hence $\alpha = \alpha'$.
\end{proof}

In the next theorem, we need the power of  \pamres. In the context of \pamres, 
given $a$ and $b$, there are unique $m$ and $r$ such that $b= am+r$ and $r<a$.
We will  write $\divi ba$ for $m$ and $\rema ba$ for $r$.
 
\begin{theorem}[\pamres]\label{substasmurf}
Every $\alpha$ in the special linear monoid is of one of the forms ${\tt A}^n$ or $\beta\sing n$. 
More specifically, if $\alpha = \sspl abcd$. Then, $\alpha ={\tt A}^{\divi ba}$ or $\alpha = \beta\sing {\,{\divi ba}\,}$.
\end{theorem}

\begin{proof}
Let $\alpha = \sspl abcd$. We work in the extension to the full ring. 
Let $\beta_0 := \sspl a{\rema ba}c{d-{\divi ba}c}$. 
Clearly, $\beta_0{\tt A}^{\divi ba} = \alpha$. 
We see that ${\sf det}(\beta_0)=1$ and, hence, that $d-\divi bac$ must be positive. 
If $\rema ba =0$, then either $\beta_0= \oslash$ or $\beta_0$ ends with {\tt B}.
If $\rema ba> 0$, then $\beta_0$ ends with {\tt B}. So, either $\beta_0 =\oslash$ or,
for some $\beta_0 = \beta{\tt B}$, for some $\beta$. 
\end{proof}

We will show, in Theorem~\ref{habermassmurf},  that Theorem~\ref{substasmurf} is not derivable in \pamo.

\begin{cor}[\pamres]
Every {\tt B}-free string is of the form ${\tt A}^n$.
\end{cor}

 \subsection{Interpretation of \futc1\ in $\pamo$}\label{luckysmurf}
   
   In this section, we give an interpretation $\mf U$ of \futc1, based on a translation $\uptheta$, in 
   $\pamo$. We have already studied the translation $\upeta$ in Section~\ref{matersmurf} that supports an interpretation of
 $\tcc{2}[{\tt a},{\tt b}]$  in $\pamo$.

   We define the translation $\uptheta$.
  \begin{itemize}
\item
$\updelta_\upgamma^{\mf o}(x)$ iff $x=x$.
\item
$ {\sf D} :=\updelta_\uptheta^{\mf s}$ consists of $\oslash$ plus all matrices
of the form ${\tt B}\alpha$, where $\alpha$ is in the special linear monoid. Here the matrices are given as
syntactic quadruples. 
\item
$\oslash_\uptheta := \oslash$.
 \item
 $[n]_\uptheta := {\tt B}{\tt A}^n = \sspl 1n1{n+1}$.
 \item
 $\digamma_\uptheta \sspl abcd := \tupel{a,b,c,d} := \tupel{a,\tupel{b,\tupel{c,d}}}$,\\
 where $\tupel{x,y}:= (x+y)^2+x$ is the non-surjective variant on the Cantor Pairing employed by \jer\ in \cite{jera:sequ12}.
 \item
  $\alpha \star_\uptheta \beta := \alpha\beta$.
  \item
  $x=^{\mf{oo}}_{\uptheta}y :\iff x=y$.
  \item
  $\sspl abcd =_{\uptheta}^{\mf {ss}}\sspl{a'}{b'}{c'}{d'} :\iff a=a' \wedge b=b' \wedge c=c'\wedge d=d'$. 
\end{itemize}

As promised in Section~\ref{mediatorsmurf}, our translation $\uptheta$ coincides with $\upeta \circ \upgamma$ on
 the string part. The embedding function is obtained by taking $Gx := {\tt A}^x$. We need the fact that ${\tt A}^x$ is {\tt B}-free.
 The Frege function $F$ is as specified above. 

We note that our translation is quantifier-free. In the case of the container string domain {\sf D}, we can write:
\begin{itemize}
\item
${\sf D}(a,b,c,d) :\iff ad = bc+1 \wedge (( b=0 \wedge c=0)  \vee ( a\leq c \wedge  b\leq d))$.
\end{itemize}
This means that our interpretation is fully quantifier-free. As a consequence the corresponding interpretation will support a functor from models to models with as
morphisms embeddings of models. 

We note that, for example, \emph{being a component of a container string} is not quantifier-free. 
  
  \begin{theorem}[$\pamo$]\label{goudensmurf}
  The translation $\uptheta$ supports an interpretation of \futc1. 
  \end{theorem}
   
   \begin{proof}
   The axioms \axtcu1--4 are immediate, by Theorem~\ref{aktievesmurf}, since on the string part $\uptheta$ and $\upgamma \circ \upeta$ coincide.
   By Theorem~\ref{freesmurf}, ${\tt A}^n$ is {\tt B}-free.
   So, we can again use  Theorem~\ref{aktievesmurf} to show that $[n]$ is an atom. The injectivity of 
   $[\cdot]$ is immediate as is the validity of the Frege Axiom~\axtcu{\ref{fregesmurf}}.
   \end{proof}
   
   We note that we only used that \pam\ interprets  $\tcc{1}[{\tt a},{\tt b}]$.
   
   Theorem~\ref{goudensmurf} in combination with theorem~\ref{koperensmurf} provides an alternative proof of 
   \jer's result that $\pam$ is sequential. \jer's result is better since he proves that the weaker theory \pamj\ is sequential. 
   However,
   we can do a little bit better than our present result by proving that a theory \pamp\ that is strictly between \pamj\ and \pam\ is sequential. See Appendix~\ref{negatievesmurf}.
   
   \begin{ques}
   {\small Can a sharper analysis deliver a proof of the sequentiality of \pamj\ using Markov coding?}
   \end{ques}
   
  We note that there is also a Frege-function-free formulation of the relevant insight. There is
  an $\mf o$-direct interpretation of \utc\ in \pam\ such that the interpretion of the string part is
  one-dimensional.

\subsection{Interpretation of \futc{2}\ in \pamres}
We use the translation $\uptheta$ of Section~\ref{luckysmurf}.

\begin{theorem}[\pamres]
$\uptheta$ carries an interpretation of  theory $\futc {2}$.
\end{theorem}

We need only verify \axtcu\ref{urthor}. This follows immediately from the following theorem.

\begin{theorem}[\pamres]\label{hoplasmurf}
Consider any $\alpha$ in {\sf D}. We have $\alpha= \oslash$ or $\alpha=\beta[n]$, for some $\beta$ in {\sf D} and some $n$.
More specifically, in the second case, if $\alpha =\sspl abcd$, then $\alpha =\beta [ \,\divi ba\,]$.

In other words, Euclidean Division \axpam\ref{plimp16} implies the Stack Principle  \axtcu\ref{urthor} under our
translation $\uptheta$
\end{theorem}

\begin{proof}
Theorem~\ref{substasmurf} tells us that $\alpha$ is of one of the forms ${\tt A}^{\divi ba}$ or $\beta[\,\divi ba\,]$, where $\beta$ is in the special monoid.

In the first case, $\alpha$ is either $\oslash$ or ${\tt A}{\tt A}^{\divi ba-1}$. If $\alpha =\oslash$, we are done. 
By Theorem~\ref{hippesmurf}, the other case is impossible, since $\alpha = {\tt B}\gamma$.

So, suppose $\alpha = \beta[\,\divi ba\,]$, where $\beta$ is in the special monoid. It suffices to show that $\beta$ is in {\sf D}.
If $\beta$ is empty, we are done. Suppose not. We have $\alpha = {\tt B}\gamma$ and, by Theorem~\ref{sterkesmurf}, that
$\beta = {\tt A}\delta$ or $\beta = {\tt B}\delta$. In the first case, we would have ${\tt B}\gamma ={\tt A}\delta[\,\divi ba\,]$. This is impossible
by Theorem~\ref{hippesmurf}. In the second case, we are done. 
\end{proof}

\subsection{Reverse Mathematics of the Stack Principle for Container Strings}\label{reversosmurf}
We treat the reverse mathematics of the Stack Principle for container strings \axtcu\ref{urthor} with respect to the translation $\uptheta$. We show that this principle
is equivalent to Euclidean Division \axpam\ref{plimp16}\ over \pamo. Our result employs an insight due to Naoya Kato from
\cite{kato:equi26}.

Let us call the $\uptheta$-translation of  \axtcu\ref{urthor} simply   \axtcu\ref{urthor}${}^\uptheta$. So:
\begin{itemize}
\item[ \axtcu\ref{urthor}${}^\uptheta$.]
Every $\gamma$ in {\sf D} is either $\oslash$ or of the form $\delta[n]$, where $\delta$ is in {\sf D}.
\end{itemize}

We will prove:
\begin{theorem}
Over \pamo, the principle Euclidean Division  \axpam\ref{plimp16} and the Stack Principle for container strings
\axtcu\ref{urthor}${}^{\,\uptheta}$ are equivalent.
\end{theorem}

The proof will follow from the combination of  Theorem~\ref{kaassmurf}, Lemma~\ref{marathonsmurf} and Theorem~\ref{tangosmurf}.

\begin{theorem}[\pamo]\label{kaassmurf}
The principle \axtcu\ref{urthor}${}^{\,\uptheta}$ is equivalent to: 
\begin{itemize}
\item[\dag.]
Every $\alpha$ in the special linear monoid is either of the form ${\tt A}^n$ or $\beta[n]$.
\end{itemize}
\end{theorem}

\begin{proof}
To show that \axtcu\ref{urthor}${}^{\uptheta}$  follows from (\dag), 
we just need to follow the proof of Lemma~\ref{hoplasmurf}, noting that Euclidean Division only enters that proof via
the use of Theorem~\ref{substasmurf}.

We prove (\dag) from  \axtcu\ref{urthor}${}^{\uptheta}$. Suppose $\alpha$ is in the special linear monoid. Then ${\tt B}\alpha$ is in {\sf D}.
It follows that ${\tt B}\alpha = \oslash$ or ${\tt B}\alpha = \delta[n]$, where $\delta$ is in {\sf D}.
The first case is impossible. Suppose $\delta = \oslash$. In that case ${\tt B}\alpha = {\tt B}{\tt A}^n$ and it follows by
Theorem~\ref{smeuigesmurf}, that $\alpha = {\tt A}^n$.  
Suppose $\delta = {\tt B} \beta$, where $\beta$ is in the special linear monoid. It follows that ${\tt B}\alpha = {\tt B}\beta[n]$,
and, hence, by Theorem~\ref{hippesmurf}, that $\alpha =\beta[n]$.
\end{proof}

We  prove the equivalence of a weaker principle \axpam \ref{plimp16}$^-$ to the property (\dag) of
Theorem~\ref{kaassmurf}.
We define:
\begin{itemize}
\item
${\sf euc}(a,b) :\iff \exists x\, \exists r\, (b = ax+r \wedge r<a)$.
We call $(a,b)$ in {\sf euc} a \emph{Euclidean pair}.
\item
${\sf cop}(a,b) :\iff \exists c\,\exists d\, (ad-bc = 1)$.
Such pairs $(a,b)$ are known as \emph{co-prime} or \emph{co-maximal} pairs.
\end{itemize}
We note that we have a minus in our form of the B\'ezout identity since we are talking about non-negative elements.
The principle \axpam\ref{plimp16}\ can be written as: $\vdash x\neq 0 \to {\sf euc}(x,y)$. Moreover,
$\pamj \vdash {\sf cop}(x,y) \to x\neq 0$.
We consider the following weaker principle.
\begin{itemize}
\item[{\axpam \ref{plimp16}}$^-$]
$\vdash {\sf cop}(x,y) \to {\sf euc}(x,y)$  
\end{itemize}

We first need a lemma.

\begin{lem}[\pam]\label{vissmurf}
Suppose $\alpha := \sspl abcd$ is in the special linear monoid and $\alpha$ is of the form ${\tt A}^n$ or $\beta[n]$.
Then, the pair $a$,$b$ is Euclidean. Moreover, if $c\neq 0$, then the pair $c$, $d$ is also Euclidean.
\end{lem}

\begin{proof}
The case of ${\tt A}^n$ is immediate. 
Suppose $\alpha =\beta [n]$.
Let $\beta{\tt B} = \sspl efgh$.
By Theorem~\ref{ducksmurf}, we have $ \scol fh < \scol eg$. 
We have $\alpha = \sspl e{ne+f}g{ng+h}$. It follows that $e=a$ and $g=c$ and $b= na+f$ and $d= nc+h$.
Moreover, we have (1) ($f < a$ and $h \leq c$) or (2) ($f \leq  a$ and $h < c$). In case (1) the pair $(a,b)$ is Euclidean. So, we are done. 
Suppose we are in Case (2). In case $f<a$, we are again done. So we assume that $f=a$. We have $b = na+a$, so $b= (n+1)a+0$ and,
hence $(a,b)$ is Euclidean. (Alternatively, we can argue that Case (2) cannot obtain. In case $a=f$, our matrix $\beta{\tt B}$ will be
$\sspl each$. Since the determinant of the matrix is 1, it follows that $a=1$ and $h-c=1$, contradicting the fact that $h$ is supposed to
be $<c$.) 

The case for $c,d$ is similar: from $d=nc+h$ and $h \leq c$, we obtain division in case $h<c$ and division with remainder zero if $h=c$.
\end{proof}

What happens if in the matrix $\beta{\tt B}$, in the proof of Lemma~\ref{vissmurf}, we have $c=h$?
Then our matrix $\beta{\tt B}$ is of the form $\sspl aecc$. It follows that $c=1$ and $a=e+1$.
So our matrix is of the form $\sspl {e+1}e11 = {\tt A}^e{\tt B}$.

\begin{lem}[\pamo]\label{marathonsmurf}
The principle \axpam\ref{plimp16}$^-$ is equivalent to: 
\begin{itemize}
\item[\dag.]
Every $\alpha$ in the special linear monoid is either of the form ${\tt A}^n$ or $\beta[n]$, where $\beta$ is in the special linear monoid.
\end{itemize}
\end{lem}

\begin{proof}
That (\dag) follows from  \axpam\ref{plimp16}$^-$  is just the proof of Theorem~\ref{substasmurf} noting that its use of Euclidean Division
is limited to co-prime pairs $(a,b)$.

Suppose we have (\dag). We consider any co-prime pair $(a,b)$. Let $c$ and $d $ be such that $ad-bc=1$.
Let $\alpha := \sspl abcd$. We may now apply Lemma~\ref{vissmurf}.
\end{proof}

Finally, we have Kato's Theorem.

\begin{theorem}[Naoya Kato]\label{tangosmurf}
Over \pamo, the principles \axpam\ref{plimp16}$^-$ and \axpam\ref{plimp16} are equivalent.
\end{theorem}

\noindent
For the proof, see \cite{kato:equi26}.

\subsection{The Second Standard Model of \pamres}\label{hulksmurf}
In this subsection, 
we give a characterisation  of the Markov strings in  the model
$\mathbb M_2 := (\mathbb Q[\X]\cdot \X + \mathbb Z)^{\mathsmaller{\geq 0}}$. 
The model $\mathbb M_2$ is the most salient model, apart from $\mathbb N$, of \pamres.
For this reason, we like to call $\mathbb M_2$ \emph{the second standard model of \pamres}.
We will verify that  $\mathbb M_2$  satisfies Euclidean Division.
The models $\mathbb M_0$ and $\mathbb M_1$ will be introduced in the next
Subsection~\ref{alitersmurf}.  

The model $\mathbb M_2$ does not satisfy {\sf IOpen}.
This can be seen by noting that the formula $x^2 <\X$ parametrically
defines the standard natural numbers. Of course, it also follows from
Shepherdson's characterisation of models of {\sf IOpen}. See
\cite{shep:nons64}.

We prove that $\mathbb M_2$ is a model of 
 \pamres, i.e. \pamo\ plus the Euclidean Division Axiom.

\begin{theorem}[$\mathbb M_2$]
We have Euclidean Division. 
\end{theorem}

\noindent
The simplest way of seeing this is simply running mentally through the obvious
long division. However, it is somewhat laborious to describe this
in detail. The proof we give uses that we already did the work for
$\mathbb Q[\X]$.

\begin{proof}
We give our proof in the context of the ambient model $\mathbb Q[\X]$. 

Let $A$ and $B \neq 0$ in $\mathbb M_2$ be given. We want to show that there are $P$ and $R$ in $\mathbb M_2$ such that
$A = PB+R$ with $R<B$. There is an $A'$ in $\mathbb Q[\X]$ such that $A = A'\X+a_0$. Clearly, $A' \geq 0$.
Since $\mathbb Q[\X]$ is a Euclidean domain, there are $P'$ and $R'$ in $\mathbb Q[\X]$ such that
$A'= P'B+R'$ and the degree of $R'$ is strictly smaller than the degree of $B$. 
Clearly $P' \geq 0$.
We have:
$A= P'\X B+ R'\X + a_0$. The degree of $ R'\X + a_0$ is smaller or equal to the degree of $B$.

\begin{itemize}
\item
Suppose $B \leq R'\X + a_0$. We can find an $n\geq 0$ such that 
$0 \leq R'\X + a_0 - nB < B$. We take $P := P'\X+ n$ and $R = R'\X + a_0 - nB$. Clearly
$P$ and $R$ are in $\mathbb M_2$.
\item
Suppose $R'\X+a_0 < B$.
\begin{itemize}
\item
Suppose $P'\neq 0$. We can find an $m$ such that $0 \leq R'\X  + a_0+ mB < B$.
We take $P:= P'\X - m$ and $R := R'\X + a_0 +mB$. Clearly
$P$ and $R$ are in $\mathbb M_2$.
\item
Suppose $P'=0$. In this case $A = R'\X+a_0 < B$.
We take $P := 0$ and $R := A$. \qedhere
\end{itemize}
\end{itemize}
\end{proof}

We proceed to prove a normal form theorem for the elements of $\mathbb M_2$. We will need some preparation.
We first give a basic result about pairs of non-negative rationals. We use the result  in the proof of 
Theorem~\ref{geweldigesmurf}. We think its use can be avoided but we deem the proof with the result is a bit neater.
Rational numbers have their well-known fractional normal form, to wit the fraction $(m,k)$, where 
$m$ and $k$ are coprime. In analogy,
we consider triples $(m,n,k)$ of non-negative integers, where $k\neq 0$. These represent the rational numbers $\frac mk$ and $\frac nk$.
We define:
\begin{itemize}
\item
$(m,n,k) \sim (m',n',k')$ iff $\frac mk = \frac {m'}{k'}$ and $\frac nk = \frac {n'}{k'}$.
\item
$(m,n,k)$ is \emph{irreducible} iff the set $\verz{m,n,k}$ is co-prime, i.o.w., if the
greatest common divisor of $m$, $n$, $k$ is 1.
\end{itemize}

\begin{lem}[$\mathbb M_2$]
Each $\sim$ equivalence class contains a unique irreducible element.
\end{lem}

\begin{proof}
We define a reduction system on triples.
Consider a triple $(m,n,k)$. Suppose $(m,n,k)$ are all divisible by a prime $\pi$.
Then, we can make a reduction step to $(\frac m\pi,\frac n \pi,\frac k\pi)$.
Clearly, reduction sequences terminate. Moreover, the reduction steps from
$(m,n,k)$ are orthogonal, so we have the strong diamond property.
This tells us that $(m,n,k)$ reduces to a unique normal form. It is immediate that
this normal form is irreducible.

Suppose $(m,n,k)\sim (m',n',k')$ are irreducible. Let $k^\ast$ be the least common multiple of $k$ and $k'$. 
Let $(m^\ast,n^\ast,k^\ast) := (\frac mk \cdot k^\ast, \frac nk\cdot k^\ast, k^\ast)$. Since $(m^\ast,n^\ast,k^\ast)$ reduces both to
$(m,n,k)$ and to $(m',n',k')$, their normal forms are identical. 
\end{proof}

We note that the elements of a normal form of $(m,n,k)$ are less than or equal to the corresponding elements of
$(m,n,k)$, moreover, the elements of the normal form have the same ordering.

For non-negative rationals $p$, $q$ of which at least one is non-zero, we define $\nnorm pq = \ell$ if, for some $m$, $n$, $k$, the triple $(m,n,k)$ is the normal form of $(p,q)$ and
$\ell$ is the maximum of $m$ and $n$. We note that $\nnorm pq=\nnorm qp$.
\begin{lem}[$\mathbb M_2$]\label{bijtertjesmurf}
Suppose $p$ and $q$ are positive rationals with $q<p$. Suppose   $0\leq p-sq$, where $s$ is a positive integer. Then,
$\nnorm{p-sq}q < \nnorm pq$.
\end{lem}
\begin{proof}
Let $(m,n,k)$ be a normal form for $p$, $q$. Then $(m-sn,n,k)$ represents $(p-sq,q)$.
Let the normal form of $(p-sq,q)$ be $(m',n',k')$. 
We have ${\sf max}(m',n') \leq {\sf max}(m-sn, n) < m$.
\end{proof}

Consider any pair of elements $A$, $B$ of $\mathbb M_2$, where at least one of $A$, $B$ is non-zero. We define:
\begin{itemize}
\item
$\norm{A,B}  := \omega^m+n$, where $m$ is the maximum of the degrees of $A$ and $B$ and
$n=0$, if the degrees are different, and $n = \nnorm{a_m}{b_m}$, otherwise.
Here $a_m$ is the coefficient of $\X^m$ in $A$ and $b_m$ is the coefficient of $\X^m$ in $B$.
We stipulate that the degree of the polynomial 0 is $-1$. 
\end{itemize}

\noindent
Clearly, $\norm{A,B}=\norm{B,A}$.

We note that $\mathbb Q[\X]\cdot \X + \mathbb Z$ is not a factorisation ring (just think of the element $\X$), so,
\emph{a fortiori}, it is not a Euclidean domain. Yet, in $\mathbb M_2$, we get something that is a bit
like the defining property of a Euclidean domain, the difference being that our analogue of the Euclidean function is
not unary but binary and that its well-founded range is an ordinal. Here is the basic insight.

\begin{theorem}[$\mathbb M_2$]\label{dagobertsmurf}
Suppose $A = PB + R$, $B \neq 0$ and $R < B \leq A$. Then
$\norm {R,B} < \norm {A,B}$.
\end{theorem}

\begin{proof}
Suppose that the degree of $B$ is strictly less than the degree of $A$. In that case,
since, the degree of $R$ is less than or equal to the degree of $B$, the degree of $R$ is also strictly less than the degree of $A$. So, we are done.

Suppose $A$ and $B$ have the same degree, say $n$. In this case $P = m>0$.
\begin{itemize}
\item
Suppose $b_n < a_n$.
We have $\norm{A,B} = \omega^n+ \nnorm {a_n}{b_n}$ and 
$\norm{R,B} = \omega^n + \nnorm{a_n - mb_n}{b_n}$.
So, by Lemma~\ref{bijtertjesmurf}, we are done.
\item
Suppose $b_n=a_n$.  In that case $\norm{A,B} = \omega^n+ m$, for some positive $m$, since $b_n \neq 0$.
On the other hand, comparing leading coefficients, we find that $P=1$ and the degree of $R$ is strictly smaller
than the degree of $B$. So, $\norm{R,B} = \omega^n$ and, thus, we are done.\qedhere
\end{itemize}
\end{proof}

We can now prove our promised  theorem characterising the Markov strings of $\mathbb M_2$.
Suppose $\alpha := \sspl ABCD$. \emph{Par abus de langage}, we define $\norm \alpha := \norm{A,B}$.

\begin{theorem}[$\mathbb M_2$]\label{geweldigesmurf}
Any Markov string can be  uniquely written as finite alternating product of strings of the form
${\tt A}^P$ and ${\tt B}^Q$, where $P$ and $Q$ are non-zero elements of $\mathbb M_2$. 
 Note that we allow the empty product which delivers the empty string. 
\end{theorem}

\begin{proof}
We provide an algorithm that calculates the promised string representation.
This algorithm is  a variation on the Euclidean Algorithm. We specify the steps.
We consider $\alpha = \sspl ABCD$.
\begin{itemize}
\item
If $\alpha = \emp$, we are done.
\item
Suppose $\scol AC < \scol BD$. In other words, suppose $\alpha$ ends with {\tt A}.
Let $\alpha' := \sspl A{\rema BA}C{D-{\divi BA}C}$. By Theorem~\ref{substasmurf}, we have
$\alpha' \in \mathbb M_2$ and $\alpha'{\tt A}^{\divi BA} = \alpha$. Clearly,
$\alpha'$ is either $\emp$ or it ends with {\tt B}. Finally, by Theorem~\ref{dagobertsmurf},
$\norm {\alpha'} < \norm \alpha$.
\item
Suppose $\scol BD < \scol AC$. In other words, suppose $\alpha$ ends with {\tt B}.
When, $B=0$, we find that $\alpha = {\tt B}^C$ and we are done. Suppose $B\neq 0$.
Let $\alpha' := \sspl {\rema AB}{B}{C-{\divi AB}D}D$.
Here $\alpha'$ need not be in $\mathbb M_2$, but we do have $\alpha' {\tt B}^{\divi AB} = \alpha$.
Moreover, the determinant of $\alpha' $ is 1.
\begin{itemize}
\item
In case $C-{\divi AB}D \geq 0$, we find that $\alpha'$ is in $\mathbb M_2$.
Also, $\alpha'$ is either $\emp$ or ends with {\tt A}. Moreover,  by Theorem~\ref{dagobertsmurf},
$\norm {\alpha'} < \norm \alpha$.
\item
Suppose $C' := C-{\divi AB}D < 0$. We have $(\rema AB)D - BC'=1$. Since $B \neq 0$, we have $-BC' \geq 1$. Moreover, $\rema AB \geq 0$.
 Hence,  $\rema AB=0$.
 We find:\\
\hspace*{0.2cm} $\alpha = \sspl 0 1{-1} D {\tt B}{\tt B}^{-1} {\tt B}^{\divi AB} = \sspl 11{D-1}D {\tt B}^{\divi AB-1} = {\tt B}^{D-1}{\tt A}{\tt B}^{\divi AB-1}$.\\
We note that both $D>0$ and $\divi AB >0$, so that  ${\tt B}^{D-1}$ and ${\tt B}^{\divi AB-1}$ are indeed in $\mathbb M_2$.
\end{itemize}
\end{itemize}
Thus each of our steps either lowers $\norm \alpha$ or leads directly to termination. This means that our procedure terminates. It clearly delivers the promised normal form.

It is easy to see that our normal form will be unique. We view normal forms as alternating sequences of ${\tt A}^P$ and ${\tt B}^Q$ for
varying non-empty $P$ and $Q$.  Suppose, for example, that we have normal forms $\mc A$ and $\mc A'$, where
$\mc A =\mc B{\tt A}^P$ and $\mc A' = \mc B'{\tt A}^{P'}$. Here $P$ and $P'$ are non-empty and $\mc B$ and $\mc B'$ are normal forms that are either empty or end with {\tt B}.
Suppose the values of $\mc B$ and $\mc B'$ are respectively $\beta$ and $\beta'$. We have: 
$\beta {\tt A}^P=\beta' {\tt A}^{P'}$,
where $\beta$ is either empty or ends with {\tt B} and, similarly, for $\beta'$.
Using an argument in the style of the proof of Theorem~\ref{unitweesmurf}, we find that $P=P'$ and $\beta = \beta'$.
  So, $\mc B$ and $\mc B'$ have the same values. Thus, we have reduced our claim to the same claim about
shorter normal forms.
\end{proof} 

\begin{cor}[${\mathbb M}_2$]
The container strings of ${\mathbb M}_2$ codify precisely the finite sets of elements. 
\end{cor}

\begin{cor}[${\mathbb M}_2$]
The standard natural numbers can be defined in ${\mathbb M}_2$.
\end{cor}

There are many ways to prove this. We just add a new one here.

\begin{proof}
We work in ${\mathbb M}_2$. We have: $x$ is a standard natural number if it is
an element of a Markov set which is closed under predecessor for all non-zero elements.
\end{proof}

We end the list of corollaries with a result that is really a corollary to the proof of Theorem~\ref{geweldigesmurf}(ii).

\begin{cor}[${\mathbb M}_2$]\label{sustitutesmurf}
Suppose $P$ is in ${\mathbb M}_2$ and $P$ has non-zero degree. Let $\alpha$ be a Markov string of ${\mathbb M}_2$. 
Suppose $\alpha$'s  normal form is given by an alternating sequence $\sigma$ of {\tt A}'s and {\tt B}'s with 
exponents $Q_0$, \dots, $Q_{n-1}$. Then, $\alpha[\X := P]$ has the normal form given by $\sigma$
with exponents $Q_i[\X := P]$.   
\end{cor}

\begin{proof}
We verify that each step of the proof of Theorem~\ref{geweldigesmurf}(ii) remains valid
after the substitution.
\end{proof}

What do the profiles of strings in ${\mathbb M}_2$ look like?
We define $\polo := \upomega + \mathbb Z \cdot \mathbb Q + \breve \upomega$.
Here we interpret the product of orderings antilexicographically. 
Thus,
 $\mathbb Z \cdot \mathbb Q$ is $\mathbb Q$ copies of $\mathbb Z$.
 We use $\breve{(\cdot)}$ for the reverse ordering.
 Let the degree of $P$ be non-zero. The profile of $\sspl 1P01$ is
$\overbrace{{\tt A}\dots {\tt A}}^{\polo}$ and the profile of
$\sspl 10P0$ is $\overbrace{{\tt B}\dots {\tt B}}^{\polo}$.
If the degree of $P$ is 0, then the picture of ${\tt A}^P$ is a finite string of {\tt A}'s and, similarly,
for ${\tt B}^P$. The profile of $\alpha$ is the concatenation of the pictures associated with
the components of its normal form.

\subsection{The Second and the Third Standard Model of \pamo}\label{alitersmurf}
In this subsection we will have a brief look at two further models. These will be submodels of $\mathbb M_2$.
This means that, in these models we can only define the finite sets via container strings.

\begin{theorem}\label{gnapgnapsmurf}
Let $\mc R$ be an ordered subring of $\mathbb Q[\X]\cdot \X+ \mathbb Z$. Then, the sets based on the
container strings of $\mc R^{\mathsmaller{\geq 0}}$ are precisely the finite subsets of $\mc R^{\mathsmaller{\geq 0}}$.
As a consequence, the natural numbers are definable in $\mc R^{\mathsmaller{\geq 0}}$.
\end{theorem}

\begin{proof}
Clearly all finite subsets are Markov-definable in $\mc R^{\mathsmaller{\geq 0}}$.
Conversely, every Markov representation of a set $X$ in $\mc R^{\mathsmaller{\geq 0}}$ defines
an extension $Y$ of $X$ in $\mathbb M_2$, since the definition of element-hood is purely existential and, hence, upwards preserved. 
Since $Y$ is finite, so must be $X$.
\end{proof}

Let ${\sf Int}(\mathbb Z)$ be the ring of integer-valued polynomials over $\mathbb Z$.
This ring consists of  the polynomials $P(\X)$ of $\mathbb Q[\X]$ such that, for all integers
$z$, we have $P(z)$ is an integer.
See \cite{cahen:what16} for an informative discussion of  ${\sf Int}(\mathbb Z)$.
We add the usual ordering to make it an ordered ring. It is easy to see that this delivers a discrete ordering.
The non-negative part of
${\sf Int}(\mathbb Z)$, thus, will be a model of $\pamo$.

Here are our promised models.
\begin{itemize}
\item
 $\mathbb M_0 :=\mathbb Z[\X]^{\mathsmaller{\geq 0}}$,
 \item
 $\mathbb M_1 :={\sf Int}(\mathbb Z)^{\mathsmaller{\geq 0}}$,
 \end{itemize}
 
 We note that $\mathbb M_0$ and $\mathbb M_1$ are models of \pamo, but not of \pamres.
We show that \axtcu\ref{urthor} fails in the Markov strings of $\mathbb M_0$. 
We define:

\medskip
\begin{itemize}
\item
$\mf A := \spl{9}{3\X+2}{3\X+4}{\X^2+2\X+1}$.
\end{itemize}

\medskip
\noindent
We can easily see that $\mf A$ is indeed a Markov string of $\mathbb M_0$

\begin{theorem}[$\mathbb M_0$]\label{habermassmurf}
The element $\mf A$ is neither of the form
${\tt A}^P$ nor of the form $\beta[P]$. Thus, 
Theorem~\ref{substasmurf} is not derivable in  \pamo\ and the interpretation of  \axtcu\ref{urthor} fails.
\end{theorem}

\begin{proof}
It is easily seen that the pair $(9, 3\X+2)$ is not Euclidean. The desired result is now immediate by Lemma~\ref{vissmurf}.
\end{proof}

What does $\mf A$ look like when we consider it in $\mathbb M_2$? Here is the computation:
\begin{eqnarray*}
\sspl{9}{3\X+2}{3\X+4}{\X^2+2\X+1} & \redux & \sspl{9}{2}{3\X+4}{\frac 23\X+1}{\tt A}^{\frac 13 \X}  \\
& \redux &  \sspl{1}{2}{\frac 13\X}{\frac 23\X+1}{\tt B}^4{\tt A}^{\frac 13 \X} \\
& \redux & \sspl{1}{0}{\frac 13\X}{1}{\tt A}^2{\tt B}^4{\tt A}^{\frac 13 \X} \\
& \redux & {\tt B}^{\frac 13 \X}{\tt A}^2{\tt B}^4{\tt A}^{\frac 13 \X} 
\end{eqnarray*}

\begin{theorem}[$\mathbb M_0$]
The string $\mf A$ has the profile
\[\overbrace{{\tt B}{\tt B} \dots}^{\omega\times}\overbrace{\dots {\tt A}{\tt A}}^{\breve\omega\times}.\]
\end{theorem}

\begin{proof}
Every occurrence in $\mf A$ in $\mathbb M_0$ is also an occurrence in $\mathbb M_2$.
We now simply check which occurrences in  $\mathbb M_2$ are also in $\mathbb M_0$.
For example, we have occurrences of the form ${\tt B}^{\frac 13 \X}{\tt A}^2{\tt B}^4{\tt A}^{ q\X+z}$, where
$q< \frac 13$.  We have, in the lower right corner, the polynomial $\frac 23 \X+1 + (q\X+z)(3\X+4)$.
The coefficient of $\X^2$ in that polynomial is $<1$ and, thus, it is not in $\mathbb M_0$.
Clearly, all the occurrences of the form ${\tt B}^{\frac 13 \X}{\tt A}^2{\tt B}^4{\tt A}^{  \frac 13 \X+z}$, for $z\leq 0$
are in $\mathbb M_0$.
\end{proof}

We note that $\mf A' := \mf A[\X := \X-1]$ is $\sspl 9 {3\X-1} {3\X+1} {\X^2}$.
The Markov string $\mf A'$ has $\mathbb M_2$ normal form  ${\tt B}^{\frac 13 \X-\frac 13}{\tt A}^2{\tt B}^4{\tt A}^{\frac 13 \X -\frac 13}$.
The string $\mf A'' := \mf A[\X := 3\X]$ is   $\sspl 9 {9\X+2} {9\X+4} {9\X^2+6\X+1}$. This string has normal form
${\tt B}^{\X}{\tt A}^2{\tt B}^4{\tt A}^{ \X}$. Thus, $\mf A''$ has profile
\[\overbrace{{\tt B}{\tt B} \dots}^{\omega\times}\overbrace{\dots {\tt B}{\tt B}}^{\breve\omega\times}{\tt A^2}{\tt B}^4
  \overbrace{ {\tt A}{\tt A}\dots}^{\omega\times}\overbrace{\dots {\tt A}{\tt A}}^{\breve\omega\times}.\] 
  Thus, we see that substitution does not preserve profile.


We briefly discuss a second Markov string that delivers an {\tt A}-string that is not of the form ${\tt A}^P$. Thus, $\pamo$ does not prove
that all {\tt A}-strings are of this form. 

\medskip
\begin{itemize}
\item
$\mf B := \spl{5\X+7}{\X^2-2}{25}{5\X -7}$.
\end{itemize}

\medskip\noindent
We reduce $\mf B$ to its normal form in $\mathbb M_2$.
\begin{eqnarray*}
\sspl{5\X+7}{\X^2-2}{25}{5\X -7} & \redux & \sspl{5\X+7}{3\frac 35 \X+5}{25}{18}{\tt A}^{\frac 15 \X-1}  \\
& \redux &   \sspl{1\frac 25 \X+2}{3\frac 35\X+5}{7}{18}{\tt B}{\tt A}^{\frac 15 \X-1} \\ 
& \redux & \sspl {1\frac 25\X+2}{\frac 45\X+1}{7}{4}{\tt A}^2{\tt B}{\tt A}^{\frac 15 \X-1} \\
& \redux & \sspl {\frac 35\X+1}{\frac 45\X+1}{3}{4}{\tt B}{\tt A}^2{\tt B}{\tt A}^{\frac 15 \X-1} \\
& \redux & \sspl {\frac 35\X+1}{\frac 15\X}{3}{1}{\tt A}{\tt B}{\tt A}^2{\tt B}{\tt A}^{\frac 15 \X-1} \\
& \redux & \sspl 1{\frac 15\X}{0}{1}{\tt B}^3{\tt A}{\tt B}{\tt A}^2{\tt B}{\tt A}^{\frac 15 \X-1} \\
& = & {\tt A}^{\frac 15\X}{\tt B}^3{\tt A}{\tt B}{\tt A}^2{\tt B}{\tt A}^{\frac 15 \X-1} 
\end{eqnarray*}

\begin{theorem}[$\mathbb M_0$]\label{maatjassmurf}
The string $\mf B$ has the profile
\[\overbrace{{\tt A}{\tt A} \dots}^{\omega\times}\overbrace{\dots {\tt A}{\tt A}}^{\breve\omega\times}.\]
So, $\mf B$ is a Markov {\tt A}-string that is not of the form ${\tt A}^P$.
\end{theorem}

\begin{proof}
We check which occurrences in $\mf B$ in $\mathbb M_2$ are in $\mathbb M_0$.
The only interesting cases are the occurrences ${\tt A}^{\frac 15\X}{\tt B}^3{\tt A}{\tt B}{\tt A}^2{\tt B}{\tt A}^{q \X+z}$, where $q<\frac 15$.
In these cases, the coefficient $q'$ of $\X^2$ in the upper right corner will be strictly smaller than 1. So, the occurrences are not in $\mathbb M_0$. 
\end{proof}

\noindent We note that ${\tt B}\mf B$ is an atom among the container strings in $\mathbb M_0$. However, it is not of the form
$[P]$.

Let $(\cdot)^{\sf t}$ be the transpose and let  $(\cdot)^{\sf at}$ be the anti-transpose. We have:
\begin{theorem}\label{wijsneussmurf}
Let $\mc M$ be any model of \pam. Then, we have in $\mc M$:
${\sf pro}(\alpha^{\sf at})$ is the reverse labeled order of ${\sf pro}(\alpha)$.
Moreover, ${\sf pro}(\alpha^{\sf t})$ is the result of reversing the order of ${\sf pro}(\alpha)$ and interchanging {\tt A}'s and {\tt B}'s.  
\end{theorem}

\begin{proof}
Suppose $\alpha = \beta_0 \xi \beta_1$, where $\xi$ is an atom. Note that $\xi$ and $\beta_1$ are uniquely determined by $\beta_0$ and $\alpha$.
We map the occurrence $\beta_0\xi$ in $\alpha$ to  $\beta_1^{\sf at}\xi$ in $\alpha^{\sf at}$.
Let $\gamma_0\eta$ be any other occurrence in $\alpha$, where $\eta$ is an atom. Let
$\alpha = \gamma_0\eta\gamma_1$. Suppose
$\gamma_0\eta \preceq_{\sf ini} \beta_0\xi$. It follows that $\xi\beta_1 \preceq_{\sf end} \eta\gamma_1$, and,
 thus, $\beta_1^{\sf at} \xi\preceq_{\sf ini} \gamma_1^{\sf at}\eta$. The other direction is similar.
 
 The argument for the transpose is analogous.
 \end{proof}

It follows that, in $\mathbb M_0$, ${\sf pro}(\mf B) = {\sf pro}(\mf B^{\sf at})$. So, the profile does not determine the string.
We can  find a similar example from $\mf A$ using the transpose.

The counter-examples $\mf A$ and $\mf B$ to the Stack Principle for container strings \axtcu\ref{urthor} also work in $\mathbb M_1$. We briefly consider a
counter-example that is not in $\mathbb M_0$.
We use $\scol \X n$ locally for the following polynomial:
\[\frac{\X \dots (\X-n+1)} {n\,!}.\] This is, of course, not to be confused with our use of
column matrices elsewhere in the paper.
We consider an example given by Skolem. 
The example is intended to illustrate that we do not have \emph{the Skolem property} in $\mathbb Z[\X]$, but we do have
it in ${\sf Int}(\mathbb Z)$. See \cite{cahen:what16} for an explanation. Here is the example.
\[ \mf S := \spl {\X^2-6\X+10}{8\scol \X 4 +3}{3}{\X^2+1} \]
We easily verify that $\mf S$ is a Markov string in $\mathbb M_1$ and that
$3$ does not divide $\X^2+1$ with remainder in $\mathbb M_1$. So, by Lemma~\ref{vissmurf}, we are done.

We compute the normal form of $\mf S$ in $\mathbb M_2$.

\begin{eqnarray*}
\sspl {\X^2-6\X+10}{8\scol \X 4 +3}{3}{\X^2+1} & = & \sspl{\X^2-6\X+10}
{\frac 13 \X^4 -2\X^3+ 3\frac 23\X^2 -2\X+3}
{3}{\X^2+1}\\
& \redux & \sspl{\X^2-6\X+10}
{\frac 13\X^2-2\X +3}
{3}{1} {\tt A}^{\frac 13 \X^2}\\
& \redux & \sspl{1}
{\frac 13\X^2-2\X +3}
{0}{1} {\tt B}^3{\tt A}^{\frac 13 \X^2}\\
& \redux & 
{\tt A}^{\frac 13\X^2-2\X +3\,} {\tt B}^3{\tt A}^{\frac 13 \X^2}\\
\end{eqnarray*}

\begin{theorem}[$\mathbb M_1$] \label{markovsmurf}
$\mf S$ is a Markov {\tt A}-string that is not of the form ${\tt A}^P$. Its profile is
the labeled order type of $\omega+ \omega\times \mathbb Z + \breve\omega \times \mathbb Z + \breve\omega$ labeled by {\tt A}.
\end{theorem}

\begin{proof}
Every occurrence of {\tt B} in $\mf S$, computed in $\mathbb M_2$, has $\frac 13 \X^2 -2\X+3$ in the upper right
corner, but  $\frac 13 \X^2 -2\X+3$ is not $\mathbb M_1$. 

We compute the initial occurrences of {\tt A}. These must be of the form ${\tt A}^P$, where $P$ is integer-valued and
$P\leq \frac 13\X^2-2\X+3$. Since the smallest non-negative integer-valued quadratic polynomials have main coefficient $\frac 12$, we see
that $P$ must be of the form $a\X+b$, where $a$ and $b$ are integers and, whenever $a=0$, then $b \geq 0$. This gives us the promised initial
ordering.

To find the final occurrences of {\tt A}, we can, in the light of Remark~\ref{leftrightsmurf}, read $\mf S$ backwards using right occurrences instead of
left ones.
Alternatively, we can use Theorem~\ref{wijsneussmurf}.
We take the anti-transpose of $\mf S$, which is ${\tt A}^{\frac 13 \X^2}{\tt B}^3{\tt A}^{\frac 13\X^2-2\X +3\,}$,
repeat the previous argument for the anti-transpose and then anti-transpose back.
\end{proof}

\begin{ques}{\small
The {\tt A}-strings \textup(a.k.a{.} the {\tt B}-free strings\textup) of $\mathbb M_0$ and $\mathbb M_1$ are closed under matrix multiplication/concatenation since we have
the Editors Principle. We note that these extend the strings of the form ${\tt A}^n$. We can consider matrix multiplication/concatenation
as an addition analogue.
Can something informative be said about the theory of addition in both these cases?
Can anything like multiplication of {\tt A}-strings be defined?}
\end{ques}

\begin{ques}{\small
Can we informatively characterise the normal forms in $\mathbb M_2$ of the Markov strings of $\mathbb M_0$?
Can we do the same for $\mathbb M_1$?
}
\end{ques}

\section{Coda}
In this paper, we developed with some care two ways of coding container strings in arithmetic that use
the interpretation of a string theory as intermediary. The aim was to get all the pieces in place and
to push our knowledge of these two strategies a bit further.
Many questions remain.

Given Smullyan's idea of coding strings in arithmetic, there are still many ways to proceed and code container strings.
We have studied one specific way. The natural arithmetical  theory to study our strategy turned out to be 
\pamsmu. What about the other ways of coding described in Section~\ref{stus} (and, possibly, some I missed)?
Under what precise assumptions do they work?
A second question here is to develop the model theory of $\pamsmu$, of which next to nothing is known.

We studied the Markov strings for two base theories. The first one was $\pamo$. 
The second one is $\pamres$ in which we have the Stack Principle for container strings as an extra.
Can we develop the model theory of
\pamo\ a bit further with an eye on the Markov-style container strings of the models?

We hope the reader will feel inspired to look into some of the many questions raised in the paper.


\appendix

\section{{\sf Iopen} does not imply the Powers Existence Principle}\label{luiesmurf}
In this Appendix, we show that {\sf IOpen} does not imply \axpam\ref{smurfin}.

\medskip
The Shepherdson model $\mathbb S$ of {\sf IOpen} consists of all forms \[ P(\X) := a_n\X^{n/q}+ a_{n-1}\X^{(n-1)/q}+\dots +a_1\X^{1/q}+a_0,\] where $q>0$,
the $a_i$ are real algebraic over $\mathbb Q$, $a_0\in \mathbb Z$, and $a_n>0$ if $n>0$ and
$a_n\geq 0$, if $n=0$. The operations and ordering are the obvious ones. See \cite{shep:nons64} or, e.g., \cite{marker:ende91}.
We note that we have Euclidean Division in this model. So, we can use the Tarski definition of power of 2.

We  show that no non-standard element of $\mathbb S$ is a power of 2.
Consider any non-standard element $\alpha = P(\X) $. The fact $\alpha$ is non-standard
means that $n \neq 0$. Suppose $a_0=0$. In this case $a$ is divisible by 3, and we are done.
Suppose $a_0 \neq 0$.
In this case, \[\alpha = P(\X) = |a_0| \cdot (2\cdot \frac{\frac{P(\X)}{|a_0|}-3}{2}+3)\] 
We can easily see that the second term of the product is in $\mathbb S$, and we are done.

The result also follows from \cite[Example 4.9]{jera:theo24}. \jer\ shows that the Principle
\[\forall x\, \exists u \mathbin{\geq} x\, \forall y\, (0<y<x \to \exists v\, (v\leq y < 2v \wedge v\mid u)).\]
does not follow from {\sf IOpen}. The model $\mf M$ of {\sf IOpen} considered by \jer\ has Unique Factorisation,
i.o.w, it satisfies \axpam\ref{supersmurf}. Its extension with a negative part 
is, thus, a GCD domain.\footnote{This model was constructed by Stuart J. Smith.
See \cite{smit:buil93}.}
It follows, by Theorem~\ref{mangasmurf}, that we have \axpam\ref{babysmurf} in the model.
So, the example tells us that ${\sf IOpen}+\axpam\ref{supersmurf}$ does not prove \axpam\ref{smurfin}.

\section{More on Partitions}\label{smulpaapsmurf}
If we have Left Cancellation, our category of partitions collapses into a partial order.

\begin{theorem}
In a model of \tcc1\ we have the following.
Suppose $f,g: \alpha\to \beta$. Then $f=g$.
\end{theorem}

\begin{proof}
Suppose $\alpha= (u_0,\dots, u_{k-1})$ and $\beta = (v_0,\dots v_{n-1})$.
Clearly $f$ and $g$ cannot differ at $u_0$.
Suppose $i$ is the smallest number such that $f(u_{i+1}) \neq g(u_{i+1})$. 
Say $f(u_i)=g(u_i)=v_j$.
Clearly, one of $f$, $g$ must yield a different value going from $u_i$ to $u_{i+1}$.
Without loss of generality, we may suppose this is $f$. 
We find $f(u_{i+1}) = v_{j+1}$. If $g$ also delivered a different value, we must have
$g(u_{i+1})=v_{j+1}$. \emph{Quod non.} So, $g(u_{i+1}) = v_j$.
Suppose $s$ is the smallest number such that $f(u_s)=g(u_s)=v_j$ and $t$ is the largest number so that
$g(u_t)= v_j$. We have seen that $t>i$. Then $u_s\dots u_i = v_j = u_s\dots u_i u_{i+1}\dots u_t$.
Ergo, by Left Cancellation,    $u_{i+1}\dots u_t= \estr$. But this is impossible, since $u_{i+1}\dots u_t$
is a non-empty product of non-empty strings.
\end{proof}

In case we have Left Cancellation, our construction of refinements in the proof of Theorem~\ref{plezantesmurf} delivers a pull-back
in the category of partitions and partition embeddings.

\begin{theorem}
In the category of partitions and partition embeddings, in a model of \tcc1, any two arrows to $(w)$ have a pull-back.
\end{theorem}

\begin{proof}
Fix any model of \tcc1. We run through the proof of Theorem~\ref{plezantesmurf} again and show that 
  we obtain a pull-back. 

We prove our theorem by course-of-values induction on the sum of the lengths of our partitions, where we construct
$\gamma$ recursively. Let $\alpha\times\beta$ be the result of our construction. Here the presupposition for definedness
is $\eva\alpha=\eva\beta$.

The case that the sum of the lengths of $\alpha$ and $\beta$ is 0 is immediate. If the sum is not zero, both lengths
must be non-zero.
Suppose $\alpha=\alpha_0\pcomp (u)$ and $\beta = \beta_0\pcomp (v)$.
By Theorem~\ref{lokismurf}, we find a unique $z$, such that  
such that 
(a) 
$\eva{\alpha_0} z = \eva{\beta_0}$  and  $u= z v$, or
(b) $\eva{\alpha_0} = \eva{\beta_0} z$   and $z u= v$.
Moreover, if both cases apply, then $z=\estr$.
We take:
\begin{itemize}
\item
$\alpha \times \beta := ((\alpha_0 \pcomp \psing z) \times \beta_0)\pcomp (v)$, if Case (a) applies.
\item
$\alpha \times \beta := (\alpha_0  \times (\beta_0\pcomp \psing z))\pcomp (u)$, if Case (b) applies.
\end{itemize}

We note that, if both cases apply, $z=\estr$ and, hence $u=v$. So the values delivered by the cases are the same. 

Suppose $\delta= (d_0,\dots, d_{p-1})$ is a common refinement of $\alpha$ and $\beta$.
Let the  witnessing functions for the fact that
$\delta$ is a common refinement be $f$ and $g$.
We show that there is an embedding of $\delta$ in  $\alpha \times \beta$. This embedding will be automatically unique. 
Without loss of generality, we may assume that Case (a) applies. Let $z$ be the witnessing number. 
Let $s$ be the largest number such that, for all $s'<s$,  $f(d_{s'})$ is in $\alpha_0$ and let $t$ be the largest number such that
for all $t'<t$, $g(d_{t'})$ is in $\beta_0$. Suppose $s \leq t$. Then $d_0\dots d_{s-1} = \eva{\alpha_0}$ and
$d_0\dots d_{t-1}=\eva{\beta_0}$. It follows that $d_{s}\dots d_{t-1} = z$. 
We can embed $(d_0,\dots, d_{t-1})$ in $\alpha_0 \pcomp \psing z$ by embedding
$(d_0,\dots, d_{s-1})$ in $\alpha_0$ via $f$ and by mapping $d_{s}$, \dots $d_{t-1}$ to $z$.
We can embed $(d_0,\dots, d_{t-1})$ in $\beta_0 $ via $g$.
So,  by the induction hypothesis, we can embed $(d_0,\dots, d_{t-1})$ in 
$((\alpha_0 \pcomp \psing z) \times \beta_0)$, say via $h$.
We note that $zd_{t}\dots d_{p-1} = d_s\dots d_{p-1} = u = zv $, so $d_{t}\dots d_{p-1}=v$.
It follows that we can embed $\delta$ in $((\alpha_0 \pcomp \psing z) \times \beta_0)\pcomp (v)$
by first sending $d_0$, \dots, $d_{t-1}$ to $(\alpha_0 \pcomp \psing z)\times \beta_0$ via $h$ and
then mapping $d_t$, \dots, $d_{p-1}$ to $v$.
\end{proof}

\section{Bi-Cancelation}\label{bismurf}
A principle that we did not introduce in the main body of the paper is Bi-Cancellation.
\begin{itemize}
\item[\axtc12]
$\vdash x=u\ast x\ast v \to (u= \estr \wedge v = \estr)$
\end{itemize}
We indicate the presence of Bi-Cancellation by a superscript {\sf b}.

We show that Bi-Cancellation implies Left Cancellation (and by symmetry also Right Cancellation) over \tc1.
\begin{theorem}
The theory \tcb 1\ extends \tcc 1. In other words, over \tc1, Bi-Cancellation, \axtc12, implies Left Cancellation, \axtc\ref{cancellsmurf}. 
\end{theorem}

\begin{proof}
In the light of Theorem~\ref{villeinesmurf} it suffices to prove Weak Left Cancellation:
if $xv=x$, then $v=\estr$. This follows immediately noting that $xv=x$ is equivalent to
$\estr x v=x$.
\end{proof}

\begin{theorem}[\pamsmu]
We have  Bi-Cancellation via the translation $\upbeta$ of Section~\ref{graaftelsmurf}.
\end{theorem}

\begin{proof}
We reason in \pamsmu. Suppose $u\ast x\ast v = x$. Then, $\ell(u)\cdot \ell(x) \cdot \ell(v) = \ell(x)$.
It follows that $\ell(u)=\ell(v) = 1$. So, $u=v=0= \estr$.
\end{proof}

\begin{theorem}[\pamo]
We have Bi-Cancellation via the translation $\upeta$ of Section~\ref{matersmurf}.
\end{theorem}

\begin{proof}
Let $\alpha = \sspl abcd$ and $\beta = \sspl efgh$. Consider $\alpha\beta$.
We have $a \leq ae+bg$, $b \leq af+bh$, $c\leq ce+dg$, $d \leq cf+dh$. 
If $f> 0$, then, $b < af+bh$. If $g>0$, then $c< ce+dg$. So, unless $\beta$ is the identity,
we have weak growth in all components and strict growth in onbe. Similarly, for $\beta\alpha$.

It follows that unless $\beta$ and $\gamma$ are both the identity, one of the components of $\beta\alpha\gamma$ is strictly
larger that the corresponding component of $\alpha$.
\end{proof}

\begin{theorem}\label{termhersmurf}
\tc1\ plus both left and right cancellation  does not prove bi-Can\-cel\-lation.
\end{theorem}

\begin{proof}
We construct a model of $\tc1$ plus both cancellation axioms, that
refutes \axtc12. Our model will be given using a string rewrite system or SRS
that is strongly normalising and satisfies Church-Rosser.

We consider the strings in the alphabet with {\tt a}, {\tt b}, and {\tt c}. Our single reduction rule is ${\tt abc} \redux {\tt b}$.
We note that different reductions in the same string are non-overlapping and that reductions diminish the length. 
Thus, we have an orthogonal, length-deceasing SRS.
It follows that each string $x$ has a unique normal form, which we denote by ${\sf nf}(x)$.  

We take as the domain of our structure the normal forms of our reduction system. These are, of course, precisely the {\tt abc}-free strings.
The empty string is our unit element and as concatenation we take $x \diamond y := {\sf nf}(x\ast y) $, where $\ast$ is ordinary concatenation of strings.
We follow our usual convention of suppressing $\ast$.

We will use the following symmetry. Let $x^{\sf r}$ be the result of reading $x$ backward and interchanging {\tt a} and {\tt c}. Then,
$\estr^{\sf r}=\estr$ and  $(x\diamond y)^{\sf r} = y^{\sf r} \diamond x^{\sf r}$. 

Trivially, our model refutes  \axtc{12}.  We will verify  that it satisfies \tcc1. Satisfaction of \tc0\ is immediate from the strong normalisation and
from the fact that reduction of a non-empty string can never give the empty string.  

We verify \axtc\ref{editors}. Suppose $x\diamond y = u \diamond v$. We will call the Editors Property for strings here: \emph{the ordinary Editors Property}.
\begin{itemize}
\item
In case neither $xy$ nor $uv$ contain {\tt abc}, we have $x\diamond y = xy$ and $u\diamond v = uv$, so we are done by the
ordinary Editors Property.  
\item
Suppose $xy$ contains {\tt abc} and $uv$ does not. In this case, we have $x=x_0{\tt a}^n{\tt b}$ and $y={\tt c}^ny_0$,
 or $x=x_0{\tt a}^n$ and $y={\tt b}{\tt c}^ny_0$, where either $x_0$ does not end with {\tt a} or $y_0$ does not start with {\tt c}.
 By symmetry, we only need to study the first case.
Suppose, after contraction to normal form,  the occurrence {\tt b} of the child {\tt abc} lies in $u$. We replace this occurrence in $u$ by ${\tt a}^n{\tt b}{\tt c}^n$
obtaining $u^+$. Clearly, $u^+v= xy$. We apply the ordinary Editors Property to $u^+v$ and $xy$. Via reduction, we obtain the desired
witness of the Editors Property for   $x\diamond y = u \diamond v$. The case where the occurence {\tt b} is in $v$ is symmetrical.
\item
Both $xy$ and $uv$ contain {\tt abc} and the child {\tt b} of the occurrence of {\tt abc} in $xy$ after reduction is not identical to the    
 child {\tt b} of the occurrence of {\tt abc} in $uv$. Without loss of generality, we may assume that 
 the $xy$-child comes first. Let the active parts of the reductions be ${\tt a}^n{\tt b}{\tt c}^n$, resp.  ${\tt a}^m{\tt b}{\tt c}^m$.
 The $xy$-child will be in $u$. We replace it by ${\tt a}^n{\tt b}{\tt c}^n$ obtaining $u^+$. The $uv$-child will be in
 $y$. We replace it by ${\tt a}^m{\tt b}{\tt c}^m$ obtaining $y^+$. We have $xy^+ = u^+v$. Applying the ordinary Editors Property
 to this equation and reducing gives us the desired result.
 \item
 Both $xy$ and $uv$ contain {\tt abc} and the child {\tt b} of the occurrence of {\tt abc} in $xy$ after reduction is  identical to the    
 child {\tt b} of the occurrence of {\tt abc} in $uv$. By symmetry, we may assume   $x= z_0{\tt a}^n{\tt b}$ and $y = {\tt c}^nz_1$.
 We  consider the cases that (1) 
  $u= z_0{\tt a}^m{\tt b}$ and $y = {\tt c}^mz_1$ and (2)   $u= z_0{\tt a}^m$ and $y = {\tt b} {\tt c}^mz_1$. 
 
 In case (1), we may assume without loss of generality that $m\leq n$. We take $w:={\tt c}^{n-m}$.
 Then, $xw = z_0{\tt a}^n{\tt b}{\tt c}^{n-m} \redux z_0{\tt a}^m{\tt b} = u$ and $wv ={\tt c}^{n-m}{\tt c}^mz_1 = {\tt c}^nz_1=y$. 
 
 In case (2), by symmetry, we may assume that $m \leq n$. We take $w := {\tt a}^{n-m}{\tt b}$.
 Then, $uw= z_0{\tt a}^m{\tt a}^{n-m}{\tt b}= z_0{\tt a}^n{\tt b}=x$ and $wy= {\tt a}^{n-m}{\tt b}{\sf c}^nz_1= {\tt b}{\tt c}^mz_1 =v$.
\end{itemize}

Finally, we verify the cancellation laws. We can easily do this directly, but since we already verified the Editors Axiom, by
Theorem~\ref{villeinesmurf}, it suffices to derive the Weak Cancellation Axioms.
We treat \axtc\ref{wcancellsmurf}. The other weak cancellation follows by symmetry.
 Suppose $x\diamond z = x$. If $z= \estr$, we are done. If $z\neq \estr$, an interaction must
take place. Whether $x$ ends with {\tt b} or not, an interaction between $x$ and $z$ will put a {\tt b} on a place where $x$ has an {\tt a}.
A contradiction.
\end{proof}

\section{The Sequentiality of \pamp}\label{negatievesmurf}
The theory \pamp\ is the theory \pamj\ plus the following principle.
\begin{itemize}
\item[{\sf uc}]
$\vdash (y \leq x\wedge v\leq u) \to x \cdot v+ y\cdot u \leq x\cdot u + y \cdot v$ 
\end{itemize} 

We first prove that \pamj, \pamp\ and \pam\ are distinct theories.

\begin{theorem}
The theories  \pamp,  and \pam\ do not coincide.
\end{theorem}

\begin{proof}
It is easily seen that $\mathbb N[{\sf X}]$ equipped with the standard lexicographical ordering satisfies
\pamp, but not \pam.
\end{proof}

\noindent Emil \jer\ has provided an example that shows that \pamj\ and \pamp\ do not coincide.

\begin{theorem}[\jer]
 \pamj\ and \pamp\ do not coincide.
\end{theorem}

\begin{proof}
We show that \pamj\ does not prove $ x^2 + y^2 \geq 2xy$,
 the special case of {\sf uc} where $x=u$ and $y=v$.

Let $\mf R_0 = (R_0,0,1,+,\times,<)$ be the polynomial ring $\mathbb Z[{\sf X}]$ ordered
so that $1 \leq  {\sf X} \leq  -{\sf X}^2 \leq  {\sf X}^3 \leq {\sf X}^4 \leq  \dots$
In other words, a non-zero polynomial with leading coefficient $c$ and degree $d$
is positive iff ($d \ne 2$ and $c > 0$) or ($d = 2$ and $c < 0$).

$\mf R_0$ is an integral domain (as algebraically it is the usual polynomial
ring), and it is a discretely ordered abelian group (as its
multiplication-free reduct is isomorphic to the usual order on $\mathbb Z[{\sf X}]$.
with ${\sf X} > \omega$ via negating the ${\sf X}^2$ coefficient, i.e., it's the
lexicographic direct sum of $\omega$ copies of $({\sf Z},+,<)$,where {\sf Z} is the set of integers.

Next, let $\mf R = (R,0,1,+,\times,<)$ be the polynomial ring $\mf R_0[{\sf Y}]$ ordered with
${\sf Y} > R_0$ in the usual lexicographic way. Again, $\mf R$ is an integral
domain and a discretely ordered abelian group, as its $<$-free or $\times$-free
reducts are isomorphic to reducts of the usual ordered ring $\mathbb Z[{\sf X},{\sf Y}]$
with ${\sf X} > \omega$ and ${\sf Y} > {\sf X}^{\omega}$.

Putting $\mf v = {\sf Y} + {\sf X}$, the structure $\mf R$ satisfies:
\begin{itemize}
\item[(1)]  $a \geq 0  \to  ({\sf Y}\cdot a \geq 0 \wedge \mf v\cdot a \geq 0)$.
\end{itemize}

\noindent
 This is immediate since ${\sf Y}\cdot a$ and  $\mf v\cdot a$ have the same leading coefficient as $a$.

Now, let 
$\mf S = (S,0,1,+,\times,<)$ be the substructure of $\mf R$ generated by
${\sf Y}$ and $\mf v$ (where we do not apply subtraction). Then (1) implies
\begin{itemize}
\item[(2)]  $a \geq 0  \to  s\cdot a \geq 0$, for all $s$ in $S$ and $a$ in $R$. 
\end{itemize}

Taking $a = 1$, this ensures that all elements
of $\mf S$ are nonnegative, and (2) also implies that $\mf S$ satisfies the axiom
   $x \leq y  \to  x\cdot z \leq y\cdot z$.

Together with the above-mentioned algebraic properties of $\mf R$, this
means that $\mf S$ is a model of \pamj. However, we have
 $  {\sf Y}^2 + \mf v^2 < 2{\sf Y}\cdot\mf  v$ in $\mf S$ and $\mf R$, as
${\sf Y}^2 - 2{\sf Y}\cdot \mf v + \mf v^2 = {\sf X}^2 < 0$.
\end{proof}

We use \docr\ for the theory of discretely ordered commutative rings.
We  provide translations carrying interpretations as shown in the diagram below.
We will verify that these translations do indeed support the desired interpretations. 

\[
\begin{tikzcd}[row sep = large, column sep = large]
& \docr \arrow{dl}{\sf z}   \\
 \pamp \arrow{r}{\sf emb}    &  \pam\arrow{u}{\sf nn} 
  \end{tikzcd}
  \]

This diagram commutes in  ${\sf INT}_1$ with the exception that the arrows along the path \pamp\ to \pam\ to \docr\ to \pamp\
do not deliver the identity arrow.\footnote{The category ${\sf INT}_1$ is the category of theories and interpretations in which two interpretations
are the same iff there is, verifiably in the target theory, a definable isomorphism between them. See e.g. \cite{fried:biin25} for details.}
It follows that  \pam\ and \docr\
are bi-interpretable. We will explain that we can improve this last result to the
definitional equivalence of \pam\ and \docr.  

The translation {\sf nn} is relativisation to the non-negative numbers. Clearly {\sf nn} carries an interpretation of \pam\ in \docr.
The translation {\sf z} is simply the usual pairs construction
of the natural numbers. 

\begin{itemize}
\item
The domain of {\sf z} consists of all pairs $(x,y)$.\footnote{We have polynomial pairing in \pamj, but for the
moment we are content to use `syntactic pairs'.}
\item
$(x,y) =_{\sf z} (u,v)$ iff $x+v = u+y$.
\item
$(x,y) \leq_{\sf z} (u,v)$ iff $x+v \leq u+y$.
\item
$0_{\sf z} = (0,0)$.
\item
$1_{\sf z} = (1,0)$.
\item
$(x,y)+_{\sf z}(u,v) = (x+u,y+v)$.
\item
$(x,y) \times_{\sf z} (u,v) = (x\cdot u + y\cdot v, x\cdot v+ y \cdot u)$. 
\item
$-_{\sf z}(x,y) = (y,x)$.
\end{itemize}

The verification that {\sf z} does indeed carry  an interpretation of \docr\ in \pamp\ proceeds for the most part in
\pamj\ and is entirely routine. We treat the single place where the extra axiom {\sf uc} is used. We need it to show that
the product of non-negative numbers is non-negative.

Suppose $0_{\sf z}\leq_{\sf z}(x,y)$ and $0_{\sf z} \leq (u,v)$.
This means $y \leq x$ and $v \leq u$.
We have to show that 
$ 0_{\sf Z} \leq_{\sf z} (x,y) \times_{\sf z} (u,v)$.
I.e., $x \cdot v + y\cdot u \leq   x\cdot u + y\cdot v$.
But this is precisely what {\sf uc} tells us.

In the context of \docr, we define {\sf F} as follows.
 \begin{itemize}
 \item
  ${\sf F}(x,y,z) : \iff y \geq 0 \wedge z\geq 0 \wedge y-z=x$. 
 \end{itemize}
 It is routine to show  that {\sf F} defines an isomorphism in \docr\ from
 ${\sf id}$ to ${\sf nn}\circ {\sf z}$. 
 
We turn to the mapping {\sf G}.
 The domain of the translation
  $ {\sf z} \circ {\sf emb} \circ{\sf nn}$ is formed by the $(x,y)$ such that $y\leq x$
 Let us call the domain
  ${\sf N}^\ast$.
 The operations and relations are simply the operations and relations of {\sf z} restricted to ${\sf N}^\ast$.
 
 We note that ${\sf z} \circ {\sf emb} \circ {\sf nn}$ is equal to ${\sf z}\circ {\sf nn}$. Moreover,
 $(x,0)$ is in ${\sf N}^\ast$ by \axpam\ref{plimp12}.
 
  We define  {\sf G} in the context of \pamp. This formula will carry several morphisms. 
  We define {\sf G}.
 \begin{itemize}
 \item
  ${\sf G}(x,y,z) : \iff (y,z) \in {\sf N}^\ast \wedge z+x = y$.  
 \end{itemize}
 
 In \pamp\ we may verify that {\sf G} defines an injective morphism from {\sf id} to ${\sf z} \circ {\sf nn}$.
 We conclude that in any model $\mc M$ of $\pamp$, we have a definable internal model $\mc N$ of $\pam$ and that ${\sf G}$ is
a definable embedding of $\mc M$ in $\mc N$. In other words, $\mc N$ is an extension of $\mc M$. It follows that:

\begin{theorem}
For every sentence $\phi$, we have $\pamp\vdash \phi$ iff ${\sf Th}_\forall(\pam) \vdash \phi$, where  
${\sf Th}_\forall(\pam)$ is the theory axiomatised by the universal consequences of \pam.
\end{theorem}

\begin{proof}
Suppose $\pamp \vdash \phi$. Then, certainly, the universal consequences of $\pam$ prove $\phi$, since
\pam\ extends \pamp\ and \pamp\ is axiomatised by universal formulas. 

We turn to the converse direction.
Suppose $\psi$ is a universal consequence of $\pam$ and $\pamp+ \psi \vdash \phi$. It is sufficient to show that
$\pamp \vdash \psi$. Suppose not. Then, there is a model $\mc M$ of $\pamp+\neg\,\psi$. We can extend $\mc M$
to a model $\mc N$ of $\pam$. Since $\neg\, \psi$ is existential, $\mc N\models \neg\,\psi$. \emph{Quod non.}
\end{proof}

\begin{theorem}
The theory \pamp\ is sequential.
\end{theorem}

\begin{proof}
Consider the interpretation of \pam\ in \pamp. 
We can modify this to a one-dimensional interpretation since we have extensional pairing in \pamj, and, thus,
\emph{a fortiori} in \pamp. Say this interpretation is {\sf K}. Since we have an embedding {\sf G} of the \pamp-numbers into the {\sf K}-numbers, 
we can use the sequentiality of \pam, to provide the desired sequences (or adjunctive sets)
for \pamp. These sequences or sets are coded in \pamp.

Suppose $\in$ is the adjunctive set relation in \pam. Then, we can define the desired adjunctive relation, say $\in^\ast$, in \pamp\ by:
\begin{itemize}
\item
$ x\in^\ast y \;\;:\iff\;\; y\in \updelta_{\sf K} \wedge {\sf G}x \in^{\sf K} y$.
\end{itemize}
Suppose $z$ in $\updelta_{\sf K}$ is a $\in^{\sf K}$-empty set. Then, clearly,
$z$ will also be $\in^\ast$-empty. Also, let $u$ be an $\in^\ast$-set. Then, by definition,
$u$ is an $\in^{\sf K}$-set  in $\updelta_{\sf K}$. Let $v$ be any element.
Then, ${\sf G}v$ is in $\updelta_{\sf K}$. We can take an $\in^K$ adjunction $w$ of $u$ with ${\sf G}v$.
This, clearly, is an $\in^\ast$-adjunction of $u$ with $v$. Thus, we have adjunctive set-theory in \pamp.
\end{proof}

It is easy to see that {\sf G} defines a bijection in $\pam$.  
 It follows that {\sf G} carries an isomorphism between {\sf id} and ${\sf z} \circ {\sf nn}$ in \pam.
 Thus, we find that $\docr$ and $\pam$ are bi-interpretable. By the main result of
 \cite{fried:biin25} in combination with the fact that \pam\ has an extensional pairing function, we
 find that  $\docr$ and $\pam$ are definitionally equivalent.

\end{document}